\documentclass[a4paper, 11pt]{article}
\usepackage [a4paper,left=2.5cm,bottom=2.5cm,right=2.5cm,top=2.5cm]{geometry}
\usepackage[english]{babel}
\usepackage{amssymb}
\usepackage{amsmath,amsthm, dsfont}
\usepackage{subcaption}
\usepackage{bm}
\usepackage{tabularx,array}
\usepackage{graphicx, url}
\usepackage{enumitem} 

\theoremstyle{plain}
\counterwithin{equation}{section}
\theoremstyle{plain}

\newtheorem{theorem}{Theorem}[section]
\newtheorem{remark}[theorem]{Remark}
\newtheorem{lemma}[theorem]{Lemma}
\newtheorem{corollary}[theorem]{Corollary}
\newtheorem{proposition}[theorem]{Proposition}

\newtheorem{definition}{Definition}

\usepackage{color}
\usepackage[dvipsnames]{xcolor}
\newcommand{\modch}{\color{black}}

\newcommand{\argmin}[1]{\underset{#1}{\operatorname{arg}\!\operatorname{min}}\;}
\newcommand{\one}{\mathds{1}}

\newcommand{\E}{\mathbb{E}}
\newcommand{\R}{\mathbb{R}}

\renewcommand{\P}{\mathbb{P}}

\newcommand{\modar}{\color{black}}
\newcommand{\rev}{\color{black}}
\newcommand{\revd}{\color{black}}
\newcommand{\revnew}{\color{black}}

\newcommand{\tronc}[2]{\left[#1\right]_{#2}}
\usepackage{mathtools}
\DeclarePairedDelimiter{\abs}{\lvert}{\rvert}
\DeclarePairedDelimiter{\norm}{\lVert}{\rVert}
\usepackage{diagbox,hhline}

\title{Minimax rate for multivariate data under componentwise local differential privacy constraints}
\author{Chiara Amorino\thanks{ Universit\'e du Luxembourg, L-4364 Esch-Sur-Alzette, Luxembourg. The author gratefully acknowledges financial support of ERC Consolidator Grant 815703 “STAMFORD: Statistical Methods for High Dimensional Diffusions”.} \qquad Arnaud Gloter \thanks{Laboratoire de Math\'ematiques et Mod\'elisation d'Evry, CNRS, Univ Evry, Universit\'e Paris-Saclay, 91037, Evry, France.}}

\begin{document}
\maketitle

\begin{abstract}
%
%
	Our research analyses the balance between maintaining privacy and preserving statistical accuracy when dealing with multivariate data that is subject to \textit{componentwise local differential privacy} (CLDP). With CLDP, each component of the private data is made public through a separate privacy channel. This allows for varying levels of privacy protection for different components or for the privatization of each component by different entities, each with their own distinct privacy policies. 
	{\revd
	It also covers the practical situations where it is impossible to privatize jointly all the components of the raw data.}
	We develop general techniques for establishing minimax bounds that shed light on the statistical cost of privacy in this context, as a function of the privacy levels $\alpha_1, \dots , \alpha_d$ of the $d$ components. \\
 We demonstrate the versatility and efficiency of these techniques by presenting various statistical applications. Specifically, we examine nonparametric density and {\revnew joint moments} estimation under CLDP, providing upper and lower bounds that match up to constant factors, as well as an associated data-driven adaptive procedure. {\revnew Additionally, we conduct a detailed analysis of the effective privacy level, exploring how information about a private characteristic of an individual may be inferred from the publicly visible characteristics of the same individual.}

\noindent
 \textit{Keywords: local differential privacy, minimax optimality, rate of convergence, non parametric estimation, density estimation} 
 
\noindent
\textit{AMS 2020 subject classifications: Primary 62C20, secondary 62G07, 62H05, 68P27}

\end{abstract}

\section{Introduction}
In the current era of information technology, protecting data privacy has become a significant challenge for statistical inference. With the widespread collection and storage of massive amounts of data, including medical records, social media activity and smartphone user behavior, individuals are increasingly reluctant to share sensitive information with companies or state officials. To address this issue, researchers in computer science and related fields have produced a vast literature on constructing privacy-preserving data release mechanisms. Real-world applications have also emerged, with companies such as Apple \cite{Tan17}, Google \cite{Erl14} and Microsoft \cite{Din17} developing data analysis methodologies that achieve strong statistical performance while maintaining individuals' privacy. This interest has been driven by regulatory pressure and the need to comply with privacy laws (see for example \cite{For21, Ari21}).

A highly effective method of protecting data from privacy breaches consists in differential privacy (see the landmarks \cite{Dwo06, Dwo08} as well as \cite{Dwo04, Evf03}). {\revnew It involves introducing carefully calibrated random noise to statistical computations or data summaries. This ensures that the overall statistical properties of the dataset remain intact while making it difficult to infer specific information about any individual data point. 
}

There are two main types of differential privacy: local privacy and central privacy. Local privacy involves privatizing data before sharing it with a data collector, while central privacy involves a centralized curator who maintains the sample and guarantees that any information it releases is appropriately private. While the local model provides stronger privacy protections, it also involves some loss of statistical efficiency. Nevertheless, major technology companies such as Apple and Google (see \cite{Apple} and \cite{Google}, respectively) have adopted local differential privacy protections in their data collection and machine learning tools to protect sensitive data. In this paper, the focus is on the local version of differential privacy, which is formally defined in Section \ref{S:Pb_formulation}.

Recently, there has been growing interest in studying differential privacy from a statistical inference perspective. {\rev The seminal work by Warner \cite{War65} introduced randomized responses, which have since become one of the primary randomization techniques employed in differential privacy.} However, modern research has produced mechanisms for a wide range of statistical problems, including mean and median estimation in \cite{Martin}, hypothesis testing (see for example \cite{Kai14, Jos19, Ber20, Lau22}), robustness \cite{Li23}, change point analysis \cite{Ber21, Li22} and nonparametric estimation \cite{BRS23, But, Kroll, BGW21}, among others.
In light of the increasing significance of data protection, it is crucial to find a balance between statistical utility and privacy: it is essential to ensure that data remains protected from privacy breaches while also allowing for the extraction of useful information and insights. Therefore, finding the optimal balance between these two aspects has become increasingly important.

This paper examines $n$ independent and identically distributed multivariate datasets with law $\bm{X}= (X^1, \dots , X^d)$, subject to what we call \textit{componentwise local differential privacy constraints}. \textit{Componentwise local differential privacy (CLDP)} is a term here introduced and refers to the method of separately making each component public through different privacy channels. This approach can be beneficial as different components may require varying levels of privacy protection {\modar or can not be privatized jointly}. We denote by $\alpha_j$ the amount of privacy ensured to the component $X^j$. Intuitively, $\alpha_j = 0$ guarantees perfect privacy while as $\alpha_j$ increases towards infinity, the privacy constraints become less strict. The reader can refer to Equation \eqref{eq: def local privacy} for a precise definition. The study focuses on exploring the trade-off between privacy protection and efficient statistical inference and aims to determine the optimal mechanisms for preserving privacy in this context.

Let us take the example where data is collected from $n$ individuals, comprising $d$ different aspects of their life. For instance, one component could represent data related to sport
practice, which is widely available on phone applications,
while another component could concern medical expenses. It is evident that disclosing information about the first component has a different impact than disclosing the same about the second, given that maintaining high confidentiality for medical bills is more important. Since some of the components may be correlated, it is necessary to work within a framework that takes this into account. One could wonder if it is possible to extract sensitive information on one component by taking advantage of the fact that on another component, possibly correlated to the first, a smaller amount of privacy is guaranteed. 

We give a first answer to this question in Proposition \ref{prop: proba info}, where we quantify the amount of private information $X^1$ carried by the privatized views of the other components in term of the dependence of $(X^2, \dots, X^d)$ on $X^1$ and of privacy levels $\alpha_1$, $\bar{\alpha}$, where $\bar{\alpha}$ is an upper bound for the levels of privacy ensured on the components $X^2, \dots , X^d$. 

{\rev 
	{\revnew Related to the issue of a varying level of privacy, the reference
	\cite{Wang_Xu_21} studies the problem of label privacy, considering in particular regression estimation when only the response variables are considered as private. However, this situation is extreme as only one component is sent through a channel and it does not compare to ours, where privacy is kept for all the variables.} 
	Also, an emerging literature \cite{Bie_et_al_22} studies the situation where a subset of individuals in the sample 
	allows to access their raw data. In such scenario the privacy constraint is not constant through the sample, but the situation is different with our case where the privacy level is varying through the different variables. In \cite{DR19}, the authors establish results in the case where the privacy parameter can have the same size as the dimension $d$. Specifically, in Section 3, they investigate how the estimation of $d$-dimensional quantities is affected by the presence or absence of correlation among the coordinates.}

Our research is also motivated by situations where different components can not be {\modar privatized} jointly. It may occur when practical aspects prevent the joint components to be gathered by the same organism prior to the privatisation mechanism.
Consider the situation where two different entities have collected data on different aspects of the life of $n$ individuals. These two entities could be, for instance, a health insurance company and a tax office having respectively collected health and income data on a population. Let denote by $(X^j_i)_{i=1,\dots,n}$ the data set belonging to the entity  $j$, with $j\in\{1,2\}$ and assume that none of the two entities is disposed to publicly reveal their data. A statistician interested in inferring the joint law of $\bm{X}=(X^1,X^2)$ would face a componentwise privatisation. Indeed, each entity can still {\modar privatize} its own data on the individual $i$, independently of the knowledge of the data owned by the other entity. Such mechanism yields to the privatization of the vector $\bm{X}_i=(X^1_i,X^2_i)$ component by component. {\revnew  This situation shares 
	some 
	similarities with the 
	federated learning framework, where data
	are collected and stored by different clients. However, in the federated learning situation each client observes the full vector of variables on its own subset of individuals (e.g. see \cite{GuChen24}). Componentwise differential privacy is useful when each data holder knows a subset of variables on the common population, and one wants to infer the whole vector of variables.}

These examples are encouraging us to explore the balance between statistical utility and individual privacy (on a componentwise basis) for the people from whom data is obtained.
By using a framework that considers componentwise local differential privacy constraints, we are able to identify optimal privacy mechanisms for some statistical problems and characterize how the optimal rate of estimation varies as a function of the privacy levels $\alpha_j$'s. 

With a similar goal in mind but in the case where all the components of one vector are made public through the same privacy channel, Duchi, Jordan and Wainwright proposed  the private version 
of the Le Cam, Fano and Assouad lemmas (see \cite{DJW13a, DJW13b, DJW14, Martin}). They provide minimax rates of convergence for specific estimation problems under privacy constraints through a case-by-case study. 
Rohde and Steinberger's research \cite{RS20}, published in 2020, takes a different approach by developing a general theory, similar to that of Donoho and Liu in \cite{Don91}, to characterize the differentially private minimax rate of convergence using the moduli of continuity.

It is worth highlighting that the minimax approach developed under local differential privacy constraints in \cite{Martin} enables the authors to examine the private minimax rate of estimation for various classical problems, including mean, median, and density estimation. These bounds have proven to be vital in other research studies that analyze the impact of privacy constraints on the convergence rate of various estimation problems, such as those discussed in \cite{Kroll}, \cite{But}, \cite{Sar21}, or \cite{GyoKro}, to name a few. The broad range of applications and their diversity demonstrate the significant impact of such research.   However, despite the existence of multivariate data, there is currently no work that considers separate components made public with varying levels of privacy {\rev and through independent channels}. 
Our goal is to address this research gap.

{\revnew The key contributions of this article span conceptual, technical, and practical aspects of differential privacy. First, we introduce a novel differential privacy framework that addresses privacy challenges in multivariate data analysis. Second, we establish bounds on the contraction of Kullback-Leibler (KL) divergence resulting from passing data through $d$ private channels, as stated in Theorem 3.1. These results hold for general interactive mechanisms and provide insight into how the optimal convergence rate depends on the privacy levels $\alpha_1, \dots, \alpha_d$, thereby quantifying the statistical price of privacy. Third, we demonstrate the utility of the KL contraction bounds by deriving minimax lower bounds for two key statistical problems under non-interactive privacy constraints: covariance estimation and nonparametric density estimation. Finally, we propose simple estimators that achieve the derived lower bounds and extend them to adaptive methods, leveraging Lepskii's method to address unknown problem parameters, such as the number of finite moments or Hölder smoothness. These adaptive estimators match the lower bounds up to logarithmic factors, underscoring the efficiency and versatility of our approach.}

To elaborate, we have two sets of raw data samples $\bm{X}=(X^1, \dots , X^d)$ and $\bm{\tilde{X}}=(\tilde{X}^1, \dots , \tilde{X}^d)$, each drawn from a probability distribution $P$ and $\tilde{P}$ respectively. We also have two corresponding sets of privatized samples $\bm{Z}=(Z^1, \dots, Z^d)$ and $\bm{\tilde{Z}}=(\tilde{Z}^1, \dots , \tilde{Z}^d)$, where $Z^j$ and $\tilde{Z}^j$ are the $\alpha_j$-local differential {\modar privatized} views of $X^j$ and $\tilde{X}^j$, respectively. {\revnew Let us denote by $\bm{X}^{(S)}$ the subvector of $\bm{X}$ corresponding to indices in $S$, where $S \subseteq \{1, ... , d \}$}.
The following equation explains the closeness of the laws of the privatized samples based on the proximity of the original laws:
\begin{equation}{\label{eq: main intro}}
	d_{KL}(L_{\bm{Z}}, L_{\bm{\tilde{Z}}}) \leq \Big{(} {\revnew 
		\sum_{\substack{S \subseteq \{1, ... , d \} \\ S \neq \emptyset}} \prod_{h \in S} (e^{\alpha_{h}} - 1) d_{TV}(L_{\bm{X}^{(S)}}, L_{\tilde{\bm{X}}^{(S)}})} \Big{)}^2,
\end{equation}
{\rev where we refer to \eqref{eq: def dtv} and \eqref{eq: def dkl} for a formal definition of the distances introduced above.}
Here, {\revnew $L_{\bm{X}^{(S)}}$} represents the law of the marginals of $\bm{X}$.  

This bound is useful for proving lower bounds for statistical problems 
and so it will be often used with two specific priors that are chosen by statisticians. In this case, it is helpful to use priors that have equal $d-1$ marginals, {\revnew that is $L_{\bm{X}^{(S)}} = L_{\tilde{\bm{X}}^{(S)}}$ for any strict subset $S \subset \{1, ... , d \}$. Then, the bound above} simplifies to the following expression
$$d_{KL}(L_{\bm{Z}}, L_{\bm{\tilde{Z}}}) \le \left( \prod_{j=1}^d (e^{\alpha_j}-1) d_{TV}(P,\tilde{P}) \right)^2.$$
We can compare our result with Theorem 1 of \cite{Martin}, which assumes that only one privacy channel has been used (so $\alpha_1 = \dots = \alpha_d = \alpha$). It provides the result $d_{KL}(L_{\bm{Z}}, L_{\bm{\tilde{Z}}}) \le \min(4, e^{2 \alpha})(e^\alpha -1)^2 d_{TV}(P,\tilde{P})^2$. {\rev In Section \ref{Ss: Comparison LDP}, we demonstrate that componentwise local differential privacy with a privacy parameter $\bm{\alpha} = (\alpha, \dots , \alpha)$ can be viewed as a special case of classical local differential privacy with a privacy parameter $d \alpha$. Consequently, we can derive a crude bound on \eqref{eq: main intro} that yields the same result as in \cite{Martin} (see Remark \ref{r: comp Martin}). However, our findings are generally more precise, enabling us to recover a refined bound that accurately assesses the contributions of the differences between each $k$-dimensional marginal. 
}

Using Equation \eqref{eq: main intro}, we can analyze the rate of convergence for nonparametric density estimation of a vector $\bm{X}$ belonging to an H\"older class $\mathcal{H}(\beta, \mathcal{L})$. We propose a kernel density estimator based on the observation of privatized variables $Z_i^j$, where $i=1, \dots , n $ and $j=1, \dots , d$ (refer to \eqref{eq: kde} for details). By imposing the conditions $\alpha_j \le 1$ and $n \prod_{i=1}^d \alpha_j^2 \rightarrow \infty$, we demonstrate that the $L^2$ pointwise error of this estimator reaches the convergence rate $(\frac{1}{n \prod_{i=1}^d \alpha_j^2})^\frac{\beta}{\beta + d}$. This rate is optimal in a minimax sense for small $\alpha$ (refer to Theorems \ref{th: upper density}, \ref{th: lower density} below).

It is natural to compare the convergence rate of our kernel density estimator with that of non-componentwise local privacy constraints. According to \cite{But} the latter achieves, for $\alpha < 1$, a convergence rate of $(n(e^{\alpha}-1)^2)^{- \frac{2\beta}{2 \beta + 2}} \approx (n \alpha^2)^{- \frac{\beta}{ \beta + 1}} $ for estimating the density of a vector $\bm{X}$ belonging to an H\"older class $\mathcal{H}(\beta, \mathcal{L})$ (see Remark \ref{r: rate density but} below for more details). \\
Our results are consistent with those in \cite{But} when $d=1$, and they provide some extensions for $d>1$. In particular, when $\alpha_1 = \dots = \alpha_d = \alpha$, the role of $\alpha^2$ in \cite{But} is replaced by $\alpha^{2d}$ in our analysis. 

{\rev Furthermore, we provide a detailed analysis of the estimation of the
	joint moment of a $d$-dimensional vector $(X^1,\dots,X^d)$}
	under componentwise privacy constraints, in addition to the density estimation discussed above. Here, we again find that under componentwise privacy mechanism, the quality of the estimation of the {\rev joint moment} is degraded as $\alpha$ becomes small, compared to a joint privacy mechanism (see Remark \ref{R: covariance_dim_d}). {\rev We also draw consequences of these results on the estimation of the covariance and correlation between two variables under componentwise privacy in Section \ref{Sss: application to covariance}.} 

The paper is organized as follows. In Section \ref{S:Pb_formulation}, we provide an introduction to differential privacy, {\rev we present our notation for componentwise local differential privacy and compare it with the classical local differential privacy.} Our main results are presented in Section \ref{S: main}, where we derive bounds on divergence between pairs in Section \ref{S: bounds divergences} and extend them to the case of interactive privatization of independent sampling in Section \ref{S: main sampling}. We demonstrate the practical application of our results in statistical problems in Section \ref{S: applications}. Firstly, in Section \ref{s:privacy level}, we use our techniques to investigate the precision of revealing one marginal of $\bm{X}$ by observing $\bm{Z}$. Next, in Section \ref{s: joint moment}, we focus on the problem of estimating {\rev the joint moment of a vector :} we propose a private estimator and establish upper and lower bounds for its $L^2$ risk in Sections \ref{Sss:joint moment estimator} and \ref{Sss:lower bound joint_moment}, respectively. 
{\rev Section \ref{Sss: application to covariance} deals with the application to the estimation of the covariance between random variables.}
In Section \ref{Sss: adaptive}, we suggest an adaptive procedure for the estimation of the {\rev joint moment}. Then, we examine the problem of nonparametric density estimation in Section \ref{s: density},
using a private kernel density estimator. The convergence rate of the estimator is studied in Section \ref{s: density upper}, while in Section \ref{s: density lower} we establish the minimax optimality of such rate. We conclude the density estimation section by proposing a data-driven procedure for bandwidth selection in Section \ref{s: adaptive density}. Finally, all the remaining proofs are collected in the Appendix.

\section{Problem formulation}\label{S:Pb_formulation}
We consider $\bm{X}_1$,\dots, $\bm{X}_n$ iid data whose law is $\bm{X} = (X^1, \dots , X^d) \in \bm{\mathcal{X}}=\prod_{j=1}^d\mathcal{X}^j$. It can represent the information coming from $n$ different individuals, about $d$ different aspect of their life. For each individual the information is privatized in a different way. Compared to the literature, where all the components relative to the same person are made public through the same channel, we now consider the case where each component is made public separately, that is why we talk of "\textit{componentwise local differential privacy}" (CLDP). 

Let us formalize the framework discussed before. The act of privatizing the raw samples $(\bm{X}_i)_{i = 1, \dots , n}$ and transforming them into the {\modar public} set of samples $(\bm{Z}_i)_{i = 1, \dots , n}$ is modeled by a conditional distribution, called privacy mechanism or channel distribution. 
We assume that each component of a disclosed observation, denoted by $Z_i^j$, is privatized separately and belongs to some space $\mathcal{Z}^j$, which may vary depending on the component $j$. This implies that the observation $\bm{Z}_i$ is an element of the product space $\bm{\mathcal{Z}}:=\prod_{j=1}^d \mathcal{Z}^j$. \\
We also assume that the spaces $\mathcal{X}^j$ and $\mathcal{Z}^j$ are separable {\modar complete} metric spaces, with their Borel sigma-fields defining measurable spaces $(\mathcal{X}^j,\Xi_{\mathcal{X}^j})$ and $(\mathcal{Z}^j,\Xi_{\mathcal{Z}^j})$, respectively, for all $j\in \{1,\dots,d \}$.

The privacy mechanism is allowed to be sequentially interactive, meaning that during the privatization of the $j$-th component of the $i$-th observation $X_i^j$, all previously privatized values $(\bm{Z}_{m})_{m=1,\dots,i-1}$ are publicly available. This leads to the following conditional independence structure, for any $j \in \{1, \dots , d \}$:
$$ \{ {X}_i^j, \bm{Z}_1, \dots , \bm{Z}_{i - 1} \} \rightarrow Z_i^j, \qquad Z_i^j \perp X_k^j \mid \{ {X}_i^j, \bm{Z}_1, \dots , \bm{Z}_{i - 1}\} \, \text{ for } k \neq i.$$
More precisely, for $j=1, \dots, d$ and $i= 1, \dots , n$, given $X_i^j = x_i^j\in\mathcal{X}^j$ and $\bm{Z}_m = \bm{z}_m\in\bm{\mathcal{Z}}$ for $m = 1, \dots , i-1$; the i-th privatized output $Z_i^j\in\mathcal{Z}^j$ is drawn as
\begin{equation}
 {\label{eq: def Z}}
	Z_i^j \sim Q_i^j (\cdot | X_i^j = x_i^j, \bm{Z}_1 = \bm{z}_1, \dots , \bm{Z}_{i-1} = \bm{z}_{i - 1})
\end{equation}
 for Markov kernels $Q_i^j: \Xi_{\mathcal{Z}^j} \times (\mathcal{X}^j \times (\bm{\mathcal{Z}})^{ i - 1}) \rightarrow [0, 1]$.
 The notation $(\bm{\mathcal{Z}}, \Xi_{\bm{\mathcal{Z}}})=(\prod_{j=1}^d \mathcal{Z}^j,\otimes_{j=1}^d \Xi_{\mathcal{Z}^j})$ refers to the measurable space of {\rev privatized} data while $(\bm{\mathcal{X}}, \Xi_{\bm{\mathcal{X}}})=(\prod_{j=1}^d \mathcal{X}^j,\otimes_{j=1}^d \Xi_{\mathcal{X}^j})$ is the space of {\rev sensitive} or raw data. 

 All of the examples presented in Section \ref{S: applications} have raw data that take values in $\bm{\mathcal{X}}=\R^d$. The space of privatized data, denoted by $\bm{\mathcal{Z}}$, can be quite general, as it is selected by the statistician based on a specific privatization mechanism. Nonetheless, in all of the practical examples of privatization that are discussed in Section \ref{S: applications}, the privatized data will {\revd be valued in $\bm{\mathcal{Z}}= \R^{d'}$  with $d' \ge 1$}.



A specific example of the privacy mechanism described earlier is the non-interactive algorithm, where the value of $Z_i^j$ is solely dependent on $X_i^j$. Therefore, Equation \eqref{eq: def Z} no longer contains any correlation with the previously generated $\bm{Z}$ values. In this scenario, 
we eliminate any dependence of the Markov kernels on the observation $i$. However, when different components represent diverse encrypted information associated with the same individual, there is no justification for the distinct components to follow the same distribution. Therefore, it is necessary to consider that different components may have different laws. In the non-interactive case for any $j=1, \dots , d$ and for any $i=1, \dots , n$ the privatized output is given by
\begin{equation*}
Z_i^j \sim Q^j (\cdot | X_i^j = x_i^j). 
\end{equation*}
Although it is usually easier to consider non interactive algorithms, as they lead to iid privatized sample, in some situations it is useful for the channel's output to rely on previous computations. Stochastic approximation schemes, for instance, necessitate this kind of dependency (see \cite{49 Martin}).

It is possible to quantify the privacy through the notion of local differential privacy. For
a given parameter $\bm{\alpha} = (\alpha_1, \dots , \alpha_d), \, \alpha_j \ge 0$, for any $j\in \{1, \dots , d \}$, the random variable $Z_i^j$ is an $\alpha_j$-differentially locally {\modar privatized} view of $X_i^j$ if for all  $\bm{z}_1, \dots , \bm{z}_{i-1} \in  {\modar \bm{\mathcal{Z}}}$ and $x, x' \in {\modar \mathcal{X}^j}$ we have 
\begin{equation}{\label{eq: def local privacy}}
{ \sup_{A \in  \Xi_{\mathcal{Z}^j}} \frac{Q_i^j(A |X_i^j = x, \bm{Z_1} = \bm{z_1}, \dots, \bm{Z}_{i-1} = \bm{z_{i - 1}} )}{Q_i^j(A |X_i^j = x', \bm{Z}_1 = \bm{z}_1, \dots, \bm{Z}_{i-1} = \bm{z}_{i - 1})} \le \exp(\alpha_j).}
\end{equation}
We say that the privacy mechanism { $\bm{Q}_i=(Q^1_i,\dots,Q^d_i)$ for $i=1,\dots,n$} is $\bm{\alpha}$-differentially locally private if each variable $Z_i^j$ is an { $\alpha_j$- differentially} locally private view { of $X^j_i$}. {\modar We denote by $\mathcal{Q}_{\bm{\alpha}}^{(n)}$  the set of all local $\bm{\alpha}$-differential private Markov kernels $(Q^j_i)_{\substack{1\le i \le n\\1\le j\le d\\}}$.}

The parameter $\alpha_j$ quantifies the amount of privacy that is guaranteed to the variable $X_i^j$: setting $\alpha_j = 0$ ensures perfect privacy {for recovering $X^j_i$ from the view of $Z^j_i, \bm{Z}_{1},\dots,\bm{Z}_{i-1}$}, whereas letting $\alpha_j$ tend to infinity softens the privacy restriction. \\
In the non-interactive case, {$\bm{Q}_i=\bm{Q}=(Q^1,\dots,Q^d)$ does not depend on $i$ and} the bound \eqref{eq: def local privacy} becomes
\begin{equation}{\label{eq: def local privacy non-int}}
\sup_{A \in { \Xi_{\mathcal{Z}^j}}} \frac{Q^j(A |X_i^j = x)}{Q^j(A |X_i^j = x')} \le \exp(\alpha_j).
\end{equation}
{Under componentwise local differential privacy the kernels $Q^j(\cdot |X^j_i=x)$  are mutually absolutely continuous for different $x$. 
	Hence, we can suppose that there exists a dominating measure $\mu^j$ on $(\mathcal{Z}^j,\Xi_{\mathcal{Z}^j} )$ such that the kernel $Q^j$ admits a density with respect to $\mu^j$. We denote by $q^j$ this density.}  Then, the property of $\bm{\alpha}$-CLDP defined in \eqref{eq: def local privacy non-int} is equivalent to the following. For all $x, x' \in {\modar \mathcal{X}^j}$ 
\begin{equation}{\label{eq: def density local privacy}}
\sup_{z \in {\modar \mathcal{Z}^j}} \frac{q^j(z |X^j = x)}{q^j(z |X^j = x')} \le \exp(\alpha_j).
\end{equation}

\noindent In this framework, we want to characterize the tradeoff between local differential privacy and statistical utility. In particular, we want to characterize how, for several canonical estimation problems, the optimal rate of convergence changes as a function of the privacy. For this reason, we develop some bounds on pairwise divergences which lead us to the derivation of minimax bounds under CLDP constraints.

{\revnew
In the following, we compare the newly introduced framework, referred to as CLDP, with the well-established Local Differential Privacy (LDP). Before proceeding, we provide a formal definition of LDP, following the notation and framework presented in Section 2.2 of \cite{Martin}. Specifically, consider the process of transforming raw data samples $\{ X_i \}_{i=1}^n$ into private samples $\{ Z_i \}_{i=1}^n$, ensuring $\alpha$-LDP. The full conditional distribution, and consequently the privacy mechanism, can be specified in terms of the conditional distributions $Q_i(Z_i | X_i = x_i, Z_1 = z_1, \dots, Z_{i-1} = z_{i-1})$. 

Local differential privacy imposes certain restrictions on the conditional distribution $Q_i$. For a given privacy parameter $\alpha \ge 0$, the random variable $Z_i$ is said to provide an $\alpha$-differentially private view of $X_i$ if, for all possible values of $z_1, \dots, z_{i-1}$ and for all $x, x' \in \mathcal{X}$, the following condition holds:
$$ \sup_{A \in \Xi_{\mathcal{Z}}} \frac{Q_i(A | X_i = x, Z_1 = z_1, \dots, Z_{i-1} = z_{i-1})}{Q_i(A | X_i = x', Z_1 = z_1, \dots, Z_{i-1} = z_{i-1})} \le \exp(\alpha), $$
where $\Xi_{\mathcal{Z}}$ denotes an appropriate $\sigma$-field on $\mathcal{Z}$.

In the non-interactive case, this bound simplifies to:
$$ \sup_{A \in \Xi_{\mathcal{Z}}} \sup_{x, x' \in \mathcal{X}} \frac{Q(A | X_i = x)}{Q(A | X_i = x')} \le \exp(\alpha). $$

With this definition in mind, we now proceed to compare the new CLDP framework with the traditional LDP, as discussed in the following section.}

{\rev
\subsection{Comparison with LDP}\label{Ss: Comparison LDP}
To better understand the difference between CLDP and LDP, let us focus on the situation of one sample $\bm{X}=(X^1,\dots,X^d) \in \bm{\mathcal{X}}=\prod_{j=1}^d \mathcal{X}^j$ which is publicly displayed through a CLDP mechanism based on $d$ independent channels $\bm{Q}=(Q^1,\dots,Q^d)$ with privacy parameter 
$\bm{\alpha}=(\alpha_1,\dots,\alpha_d)$. Each kernel $Q^j$ is a randomization taking values in some space $\mathcal{Z}^j$.
This CLDP mechanism also induces a LDP mechanism on the whole vector $\bm{X} \in \bm{\mathcal{X}}$ with Markov kernel
${\revnew \overline{Q}}$ taking values in $\bm{\mathcal{Z}}=\prod_{j=1}^d \mathcal{Z}^j$, defined by ${\revnew \overline{Q}}(A^1\times \dots \times A^d \mid \bm{X}=(x^1,\dots,x^d))=\prod_{j=1}^d Q^j(A^j \mid X^j=x^j)$. 
Then, the following lemma shows that ${\revnew \overline{Q}}$ satisfies the classical LDP constraint. 
\begin{lemma} \label{L:comp CLDP_LDP}
	If $\bm{Q}$ satisfies the $\bm{\alpha}$-CLDP constraint with $\bm{\alpha}=(\alpha_1,\dots,\alpha_d)$, then ${\revnew \overline{Q}}$ satisfies ${\revnew \overline{\alpha}}$-LDP constraint with ${\revnew \overline{\alpha}}=\sum_{j=1}^d \alpha_j$.
\end{lemma}
\begin{proof}
	Using the definition of ${\revnew \overline{Q}}$ we write,
	\begin{multline*}
			\sup_{(A^1,\dots,A^d) \in \prod_{j=1}^d \Xi_{\mathcal{Z}^j}} \frac{{\revnew \overline{Q}}(A^1\times\dots \times A^d \mid \bm{X} =( x^1,\dots,x^d ))}
		{{\revnew \overline{Q}}(A^1\times\dots \times A^d \mid \bm{X} =( x^{\prime 1},\dots,x^{\prime d} ))}
	\\
	\le \prod_{j=1}^d 	
	\sup_{A^j \in \Xi_{\mathcal{Z}^j}} \frac{ Q^j(A^j \mid X^j=x^j)}{Q^j(A^j \mid X^j=x^{\prime j})}
			\le \exp\left( \sum_{j=1}^d \alpha^j \right)				
	\end{multline*}
	where we used \eqref{eq: def local privacy non-int}. This implies that the channel ${\revnew \overline{Q}}$ satisfies the classical LDP constraint with level ${\revnew \overline{\alpha}}=\sum_{j=1}^d \alpha^j$.	
\end{proof}
This lemma allows us to see the CLDP kernels as a subset of the LDP kernels with privacy parameter ${\revnew \overline{\alpha}}$.
The main specificity of CLDP type kernels is that they must satisfy the additional constraint to act independently on the different components of the {\revd raw} data $\bm{X}$. This additional constraint reduces the set of possible kernels and makes the inference on the law of the vector $\bm{X}$ harder than without the componentwise constraint. For instance if $\alpha_1=\dots=\alpha_d=\alpha$, we show in Sections \ref{s: density} that the rate of estimation for the joint law of $\bm{X}$ is determined by the growth to infinity of the quantity $n \alpha^{2d}$ when the vector $\bm{X}$ is privatized in a componentwise way, whereas this rate is determined by $n {\revnew \overline{\alpha}}^2=n d^2 \alpha^2$ when the vector $\bm{X}$ is privatized with a classical LDP  channel with privacy parameter ${\revnew \overline{\alpha}}=d\alpha$. It shows that the impact of the componentwise constraint is large when the privacy parameters are small and $n \alpha^{2d} << n d^2 \alpha^2$.
However, we stress that the CLDP constraint may be unavoidable in practice if it is impossible to collect all the components of the sensitive vector $\bm{X}$ prior emitting the public views.

We now give examples of privacy channels satisfying the CLDP constraint and for which the inclusion described in Lemma \ref{L:comp CLDP_LDP} is sharp.

Assume that $\mathcal{X}^j=\mathbb{R}$. Let $T^{(j)}>0$ and $\alpha_j>0$. We recall the definition of the Laplace mechanism channel $Q^j$. This channel is valued in $\mathcal{Z}^j=\mathbb{R}$. It is defined by $Q^j(A\mid X^j=x^j)=\P(Z^j \in A)$ where
$Z^j=\tronc{x}{T^{(j)}} + \frac{2 T^{(j)}}{\alpha_j} \mathcal{E}^j$ with 
$\tronc{x}{T^{(j)}}=\max(\min(x,T^{(j)}),-T^{(j)})$ and $\mathcal{E}^j $ is a Laplace random variable. This channel admits a density with respect to the Lebesgue measure $\mu^j$, given by $q^j(z^j \mid X^j=x^j)= \frac{\alpha_j}{4 T^{(j)}}
\exp\left( \frac{\alpha_j}{2T^{(j)}}  \abs{z^j-\tronc{x^j}{T^{(j)}}} \right)$. 
Then, we have the following result.
\begin{lemma}\label{L: preuve Laplace mech generique}
	1) We have for all $j \in \{1,\dots,d\}$,
	\begin{equation*}
		\sup_{(x,x',z) \in \mathbb{R}^3}
 \frac{q^j(z \mid X^j=x)}{q^j(z \mid X^j=x')}=e^{\alpha_j}.
	\end{equation*}
	In particular $\bm{Q}=(Q^1,\dots,Q^d)$ satisfies the $\bm{\alpha}$-CLDP constraint with
	$\bm{\alpha}=(\alpha_1,\dots,\alpha_d)$.
	
	2) The kernel ${\revnew \overline{Q}}$ satisfies the ${\revnew \overline{\alpha}}$-LDP constraint with ${\revnew \overline{\alpha}}=\sum_{j=1}^d \alpha_j$ but is not 
	$\beta$-LDP for $\beta < \sum_{j=1}^d \alpha_j$.	
\end{lemma}
\begin{proof}
1) From the expression of the density of the Markov kernel $Q^j$, we have
for $(x,x',z') \in \mathbb{R}^3$,
	\begin{align*}
	 \frac{q^j(z |X^j = x)}{q^j(z |X^j = x')} &  =  \exp \Big(- \frac{\alpha_j}{2 T^{(j)}} \abs{z - \tronc{x}{T^{(j)}}} + \frac{\alpha_j}{2 T^{(j)}} \abs { z - \tronc{x'}{T^{(j)}}} \Big) \\
		& \le \exp \Big(\frac{1}{2 T^{(j)} } \alpha_j \abs{ \tronc{x}{T^{(j)}} -  \tronc{x'}{T^{(j)}}} \Big) \\
		& \le \exp(\alpha_j),
	\end{align*}
	where in the second line we used the reverse triangle inequality, and in the third line that $\abs{ \tronc{x}{T^{(j)}} -  \tronc{x'}{T^{(j)}}} \le 2T^{(j)}$. This proves that we have $	\sup_{(x,x',z) \in \mathbb{R}^3}
	\frac{q^j(z \mid X^j=x)}{q^j(z \mid X^j=x')}\le e^{\alpha_j}$ and in turn the kernel $\bm{Q}=(Q^1,\dots,Q^d)$ is $\bm{\alpha}$-CLDP by the definition \eqref{eq: def density local privacy}.
	
	Moreover, if one takes
	$x=T^{(j)}$, $x'=-T^{(j)}$ and $z=T^{(j)}$ we see that the above supremum is exactly $e^{\alpha_j}$.
	
	2) Using the product structure of  the kernel ${\revnew \overline{Q}}$, it is defined from $\bm{\mathcal{X}}=\mathbb{R}^d$ and takes values in the public space $\bm{\mathcal{Z}}=\mathbb{R}^d$ with the associated density
	\begin{equation*}
		{\revnew \overline{q}}(\bm{z}\mid \bm{X}=\bm{x})=\prod_{j=1}^d q^j(z^j |X^j = x^j),
	\end{equation*}
	for all $\bm{x}=(x^1,\dots,x^d) \in \bm{\mathcal{X}}$ and $\bm{z}=(z^1,\dots,z^d) \in \bm{\mathcal{Z}}$. Then, by computations analogous to the first point of the lemma we get 
$ \sup_{\bm{x},\bm{x}',\bm{z}}
	\frac{{\revnew \overline{q}}(\bm{z} \mid \bm{X}=\bm{x})}{{\revnew \overline{q}}(\bm{z} \mid \bm{X}=\bm{x}')}
	= e^{{\revnew \overline{\alpha}}}$. This shows the second point of the lemma.
\end{proof}
By the first point of the above lemma, we see that the product of Laplace kernels is $\bm{\alpha}$-CLDP with
$\bm{\alpha}=(\alpha_1,\dots,\alpha_d)$, and that the values of the privacy parameter can not be reduced.
The purpose of Point 2) is to show that it is impossible, in general, to reduce the value of ${\revnew \overline{\alpha}}$ in Lemma \ref{L:comp CLDP_LDP}.
}

\subsection{Minimax framework}
Before we keep proceeding, we introduce the minimax risk in the classical framework. It will be useful to present the notion of multivariate $\bm{\alpha}$-private minimax rate, which is defined starting from the observation of the privatized outputs $Z_i^j$, for $i\in \{1, \dots , n \}$ and $j \in \{1, \dots , d \}$. \\
\\
Suppose we have a set of probability distributions $\mathcal{P}$ defined on a sample space $\mathcal{X}$, and let $\theta(P)$ be a function that maps each distribution in $\mathcal{P}$ to a value in a set of parameters $\Theta$. The specific set $\Theta$ depends on the statistical model being used. For instance, if we are estimating the mean of a single variable, $\Theta$ will be a subset of the real numbers. On the other hand, if we are estimating a probability density function, $\Theta$ can be a subset of the space of all possible density functions over $\mathcal{X}$. 
Suppose moreover we have a function $\rho$ that measures the distance between two points in the set of parameters $\Theta$ and which is a semi-metric (i.e. it does not necessarily satisfy the triangle inequality). We use this function to evaluate the performance of an estimator for the parameter $\theta$. Additionally, we consider a non-decreasing function $\Phi: \R_{+} \rightarrow \R_{+}$ such that $\Phi(0) = 0$. The classical example consists in taking $\rho(x, y) = |x - y|$ and $\Phi(t) = t^2$.

In a scenario without privacy, a statistician has access to iid observations $\bm{X}_1, \dots, \bm{X}_n$ that are drawn from a probability distribution $P \in \mathcal{P}$. The goal is to estimate an unknown parameter $\theta(P)$ that belongs to a set of parameters $\Theta$. To achieve this goal, the statistician uses a measurable function $\hat{\theta}: \mathcal{X}^n \rightarrow \Theta$. The quality of the estimator $\hat{\theta}(\bm{X}_1, \dots, \bm{X}_n)$ is evaluated in terms of its minimax risk, defined as 
\begin{equation}{\label{eq: classical minimax}}
\inf_{\hat{\theta}} \sup_{P \in \mathcal{P}} \E_P [\Phi(\rho(\hat{\theta}(\bm{X}_1, \dots, \bm{X}_n), \theta(P)))],
\end{equation}
where the inf is taken over all the possible estimators $\hat{\theta}$.

A vast body of statistical literature is dedicated to the development of methods for determining upper and lower bounds on the minimax risk for different types of estimation problems. 

In this paper we want to consider the private analogous of the minimax risk described above, which takes into account the privacy constraints in the multivariate context, where the components are made public separately. Its definition is a straightforward consequence of the $\bm{\alpha}$-CLDP mechanism as in \eqref{eq: def local privacy}. 
Indeed, for any given privacy level $\alpha_j > 0$ we have $\mathcal{Q}_{\bm{\alpha}}$ denoting the set of all the privacy mechanisms having the $\bm{\alpha}$-CLDP property. Then, for any sample $\bm{X}_1, \dots , \bm{X}_n$, any distribution $\bm{Q}^n:= (\bm{Q}_1, \dots , \bm{Q}_n) \in{\modar \mathcal{Q}^{(n)}_{\bm{\alpha}}}$ produces a set of privatized observations which have been made public separately, i.e. $Z_1^1, \dots , Z_1^d, \dots, Z_n^1, \dots , Z_n^d$. We can focus on estimators $\hat{\theta}$ which depend exclusively on the privatized sample and we can therefore write $\hat{\theta}= \hat{\theta}(Z_1^1,\dots , Z_1^d, \dots, Z_n^1, \dots , Z_n^d)$. Then, it seems natural to look for the privacy mechanism $\bm{Q}^n \in
{\modar \mathcal{Q}^{(n)}_{\bm{\alpha}}}$ for which the estimator $\hat{\theta}(Z_1^1, \dots , Z_1^d, \dots, Z_n^1, \dots , Z_n^d)$ performs as good as possible. The performance of the estimator is even in this case judged in term of the minimax risk, which leads us to the following definition. 
\begin{definition}
Given a privacy parameter $\bm{\alpha}= (\alpha_1, \dots , \alpha_d), \, \alpha_j > 0$ and a family of distributions $\theta(P)$, the componentwise $\bm{\alpha}$ private minimax risk in the metric $\rho$ is
$$\inf_{\bm{Q}^n \in \mathcal{Q}^{(n)}_{\bm{\alpha}}} \inf_{\hat{\theta}} \sup_{P \in \mathcal{P}} \E_{P, \bm{Q}^{n}} [\Phi(\rho(\hat{\theta}(Z_1^1, \dots, Z_n^d), \theta(P)))],$$
where the inf is taken over all the estimators $\hat{\theta}$ and all the choices {\modar $(\bm{Q}_1, \dots , \bm{Q}_n) \in {\modar \mathcal{Q}^{(n)}_{\bm{\alpha}}}$ such that 
the data $Z_1^1, \dots, Z_n^d$ are $\bm{\alpha}$-CLDP views of $X_1^1, \dots, X_n^d$ in the sense of \eqref{eq: def local privacy}.
{\revnew The notation $\E_{P, \bm{Q}^{n}}$ emphasizes the dependence on the 
randomization $\bm{Q}^{n}$ in the law of $Z^1_1,\dots,Z^d_n$.}}
\end{definition}

Our main goal consists in proving some sharp bounds on pairwise divergences, as in Section \ref{S: bounds divergences}. From there, it will be possible to derive sharp lower bounds on the 
$\bm{\alpha}$ private minimax risk for the statistical estimation of manifolds canonical problems, see Section \ref{S: applications} for some examples of applications.

\section{Main results}{\label{S: main}}
{\modar In this section, we establish a connection between the proximity of two laws for the private individual variable $\bm{X}$ and the proximity of their corresponding public views under the $\bm{\alpha}$-CLDP property. Then, we explore the usefulness of this result for the privatization of independent samplings.}
\subsection{Bounds on pairwise divergences}{\label{S: bounds divergences}}
 We assume that we are given a pair of distributions $P$ and $\tilde{P}$ defined on a common space ${\bm{\mathcal{X}}}=(\mathcal{X}^1,\dots,\mathcal{X}^d)$,
 and a privatization kernel $\bm{Q}=(Q^1,\dots,Q^d)$ where $Q^j$ is the privatization channel from $\mathcal{X}^j$ to $\mathcal{Z}^j$. We denote by $M$ and $\tilde{M}$ the law of the images of $P$ and $\tilde{P}$ through the operation of privatization. It means that we consider a couple of raw samples $\bm{X}$, $\bm{\tilde{X}}$ with distribution $P$,  $\tilde{P}$, and that the associated privatized samples $\bm{Z}$,  $\bm{\tilde{Z}}$ have distribution denoted by $M$ and  $\tilde{M}$. Consistently with the description in Section \ref{S:Pb_formulation}, each channel $Q^j$ acts on its associated component $X^j$ independently of the other channels. More formally, we can write the correspondence between $P$ and $M$ as
	\begin{equation*}
	M\left( \prod_{j=1}^d A_j\right)=\int_{\bm{\mathcal{X}}} \prod_{j=1}^d Q^j(A_j \mid X^j=x^j)
	P(dx^1,\dots,dx^d),
\end{equation*}
for any $A_j \in \Xi_{\mathcal{Z}^j}$.

Before we keep proceeding, let us introduce some notation. We denote as $d_{TV}(P_1, P_2)$ the total variation distance between the two measures $P_1$ and $P_2$:
\begin{equation}{\label{eq: def dtv}}
	d_{TV}(P_1, P_2)=\int \abs{dP_1-dP_2}=\int \abs{\frac{dP_1}{dP_1+dP_2}(x)-\frac{dP_2}{dP_1+dP_2}(x)} (P_1+P_2)(dx).
\end{equation}
 Moreover, we denote as $d_{KL}(P_1, P_2)$ the Kullback divergence between the two measures $P_1$ and $P_2$,
\begin{equation}{\label{eq: def dkl}}
d_{KL}(P_1, P_2)=\int \log\big( \frac{dP_2}{dP_1}\big) dP_2
\end{equation} 
 for $P_2$ absolutely continuous to $P_1$. 
 
Finally, we denote as {\revnew $\bm{X}^{(S)}$ the subvector of $\bm{X}$ corresponding to indices in $S$, where $S \subseteq \{1, ... , d \}$ and by $L_{\bm{X}^{(S)}}$} the law of the marginals of $\bm{X}$. According to this notation, {\revnew for $S = \{1, ... , d \}$} it is clearly {\revnew $L_{\bm{X}^{(S)}} = P$ and $L_{\tilde{\bm{X}}^{(S)}} = \tilde{P}$.} \\
\\
Our main result gives an intuition on how close the two output distributions shall be, depending on how close the laws of the {\revnew privatized} data were. Its proof can be found at the end of this section.  

\begin{theorem}
 Let $\alpha_j \ge 0$ and assume that $\bm{Q}=(Q^1,\dots,Q^d)$ guarantees the $\bm{\alpha}$-CLDP constraint as defined by the condition  \eqref{eq: def local privacy non-int}. Then,
$$d_{KL}(M, \tilde{M}) \le \Big{(} {\revnew \sum_{\substack{S \subseteq \{1, ... , d \}\\ S \neq \emptyset }} \prod_{h \in S} (e^{\alpha_{h}} - 1) d_{TV}(L_{\bm{X}^{(S)}}, L_{\tilde{\bm{X}}^{(S)}})} \Big{)}^2.$$

 \noindent In the case where $\alpha_1 = \dots = \alpha_d =: \alpha$, it reduces to
\begin{equation}
	\label{eq: maj KL par TV cas alpha constant}
	d_{KL}(M, \tilde{M}) \le \Big{(} {\revnew \sum_{ \substack{S \subseteq \{1, ... , d \}\\ S \neq \emptyset } } (e^{\alpha} - 1)^{|S|} d_{TV}(L_{\bm{X}^{(S)}}, L_{\tilde{\bm{X}}^{(S)}})} \Big{)}^2.
\end{equation}
\label{th: main bound}
\end{theorem}
{\revnew In the sequel, for the sake of shortness, we remove the condition $S\neq\emptyset$ when possible, as we can consider that 
	$d_{TV}(L_{\bm{X}^{(\emptyset)}}, L_{\tilde{\bm{X}}^{(\emptyset)}}) =0$.}
\begin{remark}
To better understand the formula in the statement of the theorem the reader may observe that, for $d=2$, the left hand side of our main bound is 
$$ \Big{[}(e^{\alpha_1} - 1) d_{TV} (L_{X^1}, L_{\tilde{X}^1} ) + (e^{\alpha_2} - 1) d_{TV} (L_{X^2}, L_{\tilde{X}^2} ) + (e^{\alpha_1} - 1)(e^{\alpha_2} - 1) d_{TV} (L_{(X^1, X^2)}, L_{(\tilde{X}^1, \tilde{X}^2)} ) \Big{]}^2. $$
For $d=3$ it is instead 
\begin{align*}
\Big{[} \sum_{i=1}^3 (e^{\alpha_i} - 1)  d_{TV} (L_{X^i}, L_{\tilde{X}^i} ) & + \sum_{1 \le i < j \le 3} (e^{\alpha_i} - 1)(e^{\alpha_j} - 1)  d_{TV} (L_{(X^i, X^j)}, L_{(\tilde{X}^i, \tilde{X}^j)} ) \\
& + (e^{\alpha_1} - 1)(e^{\alpha_2} - 1)(e^{\alpha_3} - 1) d_{TV} (L_{(X^1, X^2, X^3)}, L_{(\tilde{X}^1, \tilde{X}^2, \tilde{X}^3)} )  \Big{]}^2. 
\end{align*}
\end{remark}

\begin{remark}{\label{r: comp Martin}}
In the mono-dimensional case, where $\bm{\mathcal{X}}=\mathcal{X}^1$ and $\bm{\alpha} =\alpha_1$, we recover a bound similar to the one in Theorem 1 of \cite{Martin}, which is
	\begin{equation} \label{eq: control Martin rappel}
d_{KL}(M, \tilde{M}) \le \min(4, e^{2 \alpha}) (e^{\alpha} -1)^2 d_{TV}(P,\tilde{P})^2.		
	\end{equation}
In the multidimensional setting with $\alpha_1=\dots=\alpha_d=\alpha$, if we use in 
	\eqref{eq: maj KL par TV cas alpha constant} the crude bound 
	\begin{equation*}
		{\revnew 
			\forall S \subseteq \{1,\dots,n\},\quad
			d_{TV}(L_{\bm{X}^{(S)}}, L_{\tilde{\bm{X}}^{(S)}})}\le d_{TV}(P,\tilde{P}),
	\end{equation*}
	 we obtain
\begin{equation}{\label{eq: comp Martin}}
 d_{KL}(M, \tilde{M}) \le \left( \sum_{k=1}^d (e^\alpha-1)^k  \binom{n}{k} d_{TV}(P,\tilde{P}) \right)^2 =
(e^{\alpha d}-1)^2  d_{TV}(P,\tilde{P})^2,
\end{equation}
where we have used Newton's binomial formula.
It is important to note that when $d>1$, the inequality stated in Theorem 1 of \cite{Martin} (referred to as \eqref{eq: control Martin rappel}) is still valid. {\rev Recalling the discussion in Section \ref{Ss: Comparison LDP}, the CLDP channel with privacy parameter $\bm{\alpha}=(\alpha,\dots,\alpha)$ is a special case of LDP channel with privacy parameter $d \times \alpha$. Hence,
	the results given by the bounds \eqref{eq: control Martin rappel} and  
	 \eqref{eq: comp Martin} are the same, up to a constant.}
However, our result is generally more precise. 
{\rev Indeed, our analysis takes into account the fact that the individual components of the vector have been made public independently, which allows to recover the more precise upper bound \eqref{eq: maj KL par TV cas alpha constant}, where the contribution of  the differences between each $k$-dimensional marginals is assessed.

A main purpose of controls like \eqref{eq: maj KL par TV cas alpha constant}--\eqref{eq: comp Martin} is to theoretically assess the minimum loss of information about the raw law when the data are transmitted through the channel. Our bound   \eqref{eq: maj KL par TV cas alpha constant} reveals 
that, in the case of a CLDP channel, when $\alpha$ gets small the information about the joint law is lost faster than the information about the marginal laws.
}
\end{remark}

\noindent As we will see in next section, the bound on pairwise divergences gathered in Theorem \ref{th: main bound} is particularly helpful when one wants to show lower bound on the minimax risk, in order to illustrate the optimality of a proposed estimator, in a minimax sense. \\
In this case one can propose two priors whose marginal laws are all the same but for the last term, where the whole vector is considered. Then, our main result reduces to the bound below. 

\begin{corollary}
Let us consider a couple of raw samples $\bm{X}$, $\bm{\tilde{X}}$ with distributions $P$, $\tilde{P}$ and the associated couple of privatized samples $\bm{Z}$, $\bm{\tilde{Z}}$ with distributions $M$, $\tilde{M}$. Assume moreover that, for any 
{\revnew strict subset $S \subset \{1, ... , d \}$,} it is {\revnew $L_{\bm{X}^{(S)}} = L_{\tilde{\bm{X}}^{(S)}}$}. Then, 
$$d_{KL}(M, \tilde{M}) \le  \Big( \prod_{i = 1}^d (e^{\alpha_i} - 1)^{2} \Big) d_{TV}(P, \tilde{P})^2.$$
\label{cor: main}
\end{corollary}
\noindent  As we will see in next section, the bound {\modar stated} in the corollary {\modar is extremely useful to assess minimax risks under the $\alpha$-CLDP property.} 
{\rev Indeed, under the assumption of the corollary, the reduction of the distance between the private laws and the public laws is maximal when the $\alpha$'s are small. This is a worst case scenario when a statistician wants to determine from the public view which of the two private laws, $P$ or $\tilde{P}$, is the true one. Hence, this result is crucial in the proof of lower bounds related to estimation problems.}

\noindent The rest of this section is devoted to the proof of Theorem \ref{th: main bound}.

\begin{proof}[Proof of Theorem \ref{th: main bound}]
To prove our main theorem we introduce some notation. 
\noindent We first recall that $q^j(z^j\mid x^j)$ is the density of the the law of $Z^j$ conditional to $X^j=x^j$ with respect to a dominating measure $\mu^j(dz^j)$. We denote by $q(z^1,\dots,z^d)$ the density of the law of
	 $(Z^1,\dots,Z^d)$, which exists with respect to the reference measure $\bm{\mu}(d\bm{z}):=\mu^1(dz^1)\times\dots\times\mu^d(dz^d)$. 
	 In a more general way, we examine a collection of $d$ symbols $\zeta^1,\dots,\zeta^d$, which can take one of three possible values: $\zeta^j = dx^j$, $\zeta^j = z^j$, or $\zeta^j = \emptyset$. We define a vector $\bm{W}$ such that the $j$-th component of $\bm{W}$, denoted $W^j$, takes the value of $X^j$ if $\zeta^j = dx^j$, takes the value of $Z^j$ if $\zeta^j = z^j$, and is removed entirely if $\zeta^j = \emptyset$. We denote as $q(\zeta^1, \dots , \zeta^d)$ the Markovian kernel such that
	 \begin{equation*}
q(\zeta^1, \dots , \zeta^d) \times \prod_{j : \zeta^j=z^j} \mu^j(dz^j)
\end{equation*}
is the law of $\bm{W}$.
For example,
\begin{equation*}
q(dx^1, \dots , dx^i, \emptyset, \dots , \emptyset, z^{i +j+1}, \dots , z^d)\mu^{i+j+1}(dz^{i+j+1})\times\dots\times \mu^d(dz^d)   \end{equation*}
is the law of $(X^1, \dots , X^i, Z^{i + j + 1}, \dots , Z^d)$ and we thus have
\begin{multline*}
\E\left[f(X^1,\dots,X^i,Z^{i+j+1},\dots,Z^d)\right]
=\int_{\mathcal{X}^1\times\dots\times\mathcal{X}^i\times\mathcal{Z}^{i+j+1}\times\dots\times\mathcal{Z}^{d}} f(x^1,\dots,x^i,z^{i+j+1},\dots,z^d) \\q(dx^1, \dots , dx^i, \emptyset, \dots , \emptyset, z^{i +j+1}, \dots, z^d)\mu^{i+j+1}(dz^{i+j+1})\times\dots\times \mu^d(dz^d),
\end{multline*}
for any positive measurable function $f$.
Such Markovian kernels $q(\zeta^1, \dots , \zeta^d)$ exist for all choices of symbols $\zeta^j$. Indeed, it is possible to disintegrate the law of $(W^j)_{j:\zeta_j\neq\emptyset}$ with respect to the law of  $(W^j)_{j : \zeta^j = z^j}$ and use that the law of $(W^j)_{j : \zeta^j = z^j}$ admits a density with respect to $\prod_{j: \zeta^j=z^j} \mu^j(dz^j)$.
With a slight abuse of notation we consider
 $q(\emptyset,\dots,\emptyset)=1$. It is consistent with the fact that, when removing one marginal $W^j=X^j$ (or $W^j=Z^j$) from a random vector, the corresponding probability measure is integrated with respect to the variable $x_j$ (or $z_j$). Hence, when removing  ultimately all the variables the probability integrates to $1$, yielding to the notation $q(\emptyset,\dots,\emptyset)=1$.  
  Let us stress that these notations are cumbersome as we are dealing with general variables $\bm{X}$ and $\bm{Z}$. For instance, in the simple case where $\bm{\mathcal{X}}=\bm{\mathcal{Z}}=\R^d$, $\bm{X}$ with density on $\R^d$, and privacy channels having densities $q^j(z^j\mid x^j)$ with respect to the Lebesgue measure, we would have simply defined $q(\zeta^1,\dots,\zeta^d)$ as the density of the variables $(W^j)_{j:\zeta^j\neq\emptyset}$.

	
 We introduce analogously $\tilde{q}(\zeta_1, \dots , \zeta_d)$, which corresponds to the law of $\bm{\tilde{X}}$ and  $\bm{\tilde{Z}}$ in the same way as above. 

 Then, using these notations and the fact that the law of $Z^1$ conditional to 
	$(X^1,Z^2,\dots,Z^d)$ is {\modar given,   from the definition of the privacy mechanism,} by $Q^1(dz^1\mid X^1=x^1)=q^1(z^1 \mid x^1) \mu^1(dz^1)$, it is 
 \begin{equation} \label{eq: q_zi depuis q1 et qx1zi}
		q(z^1, \dots, z^d)= \int_{\mathcal{X}^1}  q^1(z^1 | x^1) q(dx^1, z^2, \dots , z^d).
	\end{equation}

Finally, we introduce the function $l[\zeta^1, \dots , \zeta^d] : \mathcal{Z}^1\times\dots\times\mathcal{Z}^d \to \mathbb{R}_+$ as below:
\begin{align}\label{eq: def l z1 zd}
	l[\zeta^1,\dots,\zeta^d]:=&
		|q(z^1, \dots, z^d) - \tilde{q}(z^1, \dots, z^d)| , ~  \text{ for $\bm{\zeta} = \bm{z}$},	
		\\ \nonumber
		l[\zeta^1,\dots,\zeta^d]:=&\prod_{j : \zeta^j \neq z^j} q^j(z^j | x^j_*) \prod_{j: \zeta^j = dx^j} |e^{\alpha_j} - 1| \times\\
		&\qquad\qquad\qquad\qquad		\label{eq: def l zeta1 zetad}
		\int_{\prod\limits_{j: \zeta^j = dx^j}\mathcal{X}^j}  |q(\zeta^1, \dots , \zeta^d) - \tilde{q}(\zeta^1, \dots , \zeta^d)|,~  \text{ for $\bm{\zeta} \neq \bm{z}$},
\end{align}
%
where 
$ q^j(z^j | x^j_*) := \inf_{x^j}  q^j(z^j | x^j)$ and 
$\bm{\zeta}=(\zeta^1,\dots,\zeta^d)$. 
 To clarify the notation, let us stress that the integration variables in \eqref{eq: def l zeta1 zetad} are the $\zeta^j$ such that $\zeta^j=dx^j$. Moreover, $l[\zeta^1,\dots,\zeta^d]$ is a function of $(z^1,\dots,z^d)$ whatever is the choice $(\zeta^1,\dots,\zeta^d)$. Indeed, the variable $z^j$ appears either in the product $ \prod_{j : \zeta^j \neq z^j} q^j(z^j | x^j_*) $ when $j$ is such that $\zeta^j\neq z^j$, or in the integral when $\zeta^j=z^j$. To give an example, the quantity
	$l[dx^1,\dots,dx^i,\emptyset,\dots,\emptyset,z^{i+j+1},\dots,z^d]$ is equal to
	\begin{multline*}
		\prod_{l=1}^{i+j} q^l(z^l | x^l_*) \prod_{l=1}^i |e^{\alpha_l} - 1|  \times
		\\\int_{\prod\limits_{l=1}^{i}\mathcal{X}^l}  |q(dx^1,\dots,dx^i,\emptyset, \dots,\emptyset, z^{i+j+1},\dots, z^d) - \tilde{q}(dx^1,\dots,dx^i,\emptyset, \dots,\emptyset, z^{i+j+1},\dots,  z^d)|,
	\end{multline*} 
	which is clearly a function of $(z^1,\dots,z^d)$.
		In the scenario where $\bm{\zeta}=\bm{z}$, the equation \eqref{eq: def l zeta1 zetad} aligns with \eqref{eq: def l z1 zd}, except that the integration variables disappear. Then, in the right-hand side of \eqref{eq: def l zeta1 zetad} the integral no longer appears.
		We also specify that $l[\emptyset,\dots,\emptyset]=0$, which results from the fact that, abusing the notations, we have $q(\emptyset,\dots,\emptyset) =\tilde{q}(\emptyset,\dots,\emptyset)=1$.

Our main result heavily relies on the following proposition. Its  proof, based on a recurrence argument, can be found in the Appendix.

\begin{proposition}{\label{prop: recurrence}}
Let the function $l[\zeta^1, \dots , \zeta^d]$ be defined according to \eqref{eq: def l z1 zd} and \eqref{eq: def l zeta1 zetad}. Then, under the hypothesis of Theorem \ref{th: main bound}, the following inequality holds true. 
\begin{align}{\label{eq: goal main A}}
 l[z^1, \dots , z^d] \le   \sum_{(\zeta^1, \dots , \zeta^d) \in \prod_{j=1}^d\{ \emptyset, {dx^j} \}} l[\zeta^1, \dots , \zeta^d].
\end{align}
\end{proposition}

 Recalling \eqref{eq: def l z1 zd}, we remark that the left hand side of \eqref{eq: goal main A} assesses the difference between the densities of $\bm{Z}$ and $\bm{\tilde{Z}}$, while the right hand side only relies on the laws of $\bm{X}$, $\bm{\tilde{X}}$ as the symbol $\zeta^j=z^j$ disappears in the sum. From Proposition \ref{prop: recurrence} we have 
\begin{align}{\label{eq: start step3 4.5}}
 |q(\bm{z}) - \tilde{q}(\bm{z})| & =l[z^1, \dots , z^d] 
 \le \sum_{\bm{\zeta}: \zeta^j \in \{ \emptyset,  dx^j \}}  l(\bm{\zeta}) \\
 & = \sum_{\bm{\zeta}: \zeta^j \in \{ \emptyset, {\modar dx^j} \} } \prod_{j = 1}^{d} {\modar q^j}(z^j | x^j_*) \prod_{j: \zeta^j =  dx^j} (e^{\alpha_j} - 1) {\modar \int_{\prod\limits_{j: \zeta^j = dx^j}\mathcal{X}^j}  |q(\bm{\zeta}) - \tilde{q}(\bm{\zeta})|}.
\end{align}
By the definition of total variation distance, the quantity above can be seen as 
\begin{align}\nonumber
& \prod_{j = 1}^{d} {\modar q^j}(z^j | x^j_*) \sum_{\bm{\zeta}: \zeta^j \in \{ \emptyset, x^j \} } ~
\Bigg\{
 \big( \prod_{j: \zeta^j = dx^j} (e^{\alpha_j} - 1)\big) {\modar  d_{TV}(L_{(X^j)_{j :  \zeta^j = dx^j }}, L_{(\tilde{X}^{j})_{ j :  \zeta^j =dx^j }} )} \Bigg\} \\
& \nonumber
= \prod_{j = 1}^{d} {\modar q^j}(z^j | x^j_*) \sum_{k = 1}^d  
{\modar 
\sum_{\substack{ j_1< \dots< j_k \\ \text{ $k$ different indexes} \\ \text{
			in $\{1,\dots,d\}$ } } } }
		\Bigg\{
		\big(
	 \prod_{i=1}^k (e^{\alpha_{j_i}} - 1) \big) ~d_{TV}(L_{(X^{j_1}, \dots , X^{j_k})}, L_{(\tilde{X}^{ {j_1}}, \dots , \tilde{X}^{{j_k}})}) \Bigg\}
\\
&=\label{eq: upper bound diff qz qz' par somme}
{\revnew \prod_{j = 1}^{d} q^j(z^j | x^j_*) 
\sum_{ S \subseteq \{1, ... , d \}} \prod_{h \in S} (e^{\alpha_{h}} - 1) d_{TV}(L_{\bm{X}^{(S)}}, L_{\tilde{\bm{X}}^{(S)}}).}
\end{align}
To conclude the proof we observe it is 
\begin{align}\nonumber
d_{KL}(M, \tilde{M}) & = \int_{{\modar \mathcal{Z}^1\times\dots\times\mathcal{Z}^d}} q(\bm{z}) \log (\frac{q(\bm{z})}{\tilde{q}(\bm{z})}) {\modar d\bm{\mu}(\bm{z})} + \int_{{\modar \mathcal{Z}^1\times\dots\times\mathcal{Z}^d}} \tilde{q}(\bm{z}) \log (\frac{\tilde{q}(\bm{z})}{q(\bm{z})}) {\modar d\bm{\mu}(\bm{z})} \\
& {\label{eq: middle step3 5}}
= \int_{{\modar \mathcal{Z}^1\times\dots\times\mathcal{Z}^d}} (q(\bm{z}) - \tilde{q}(\bm{z})) \log (\frac{q(\bm{z})}{\tilde{q}(\bm{z})}) {\modar d\bm{\mu}(\bm{z})}. 
\end{align}
Then, Lemma 4 in \cite{Martin} entails $|\log \frac{q(\bm{z})}{\tilde{q}(\bm{z})} | \le \frac{|q(\bm{z}) - \tilde{q}(\bm{z})|}{\min (q(\bm{z}), \tilde{q}(\bm{z}))}$.
 In order to study the denominator, we write 
\begin{align*}
q(\bm{z}) & ={\modar  \int_{{\modar x^1 \in \mathcal{X}^1}} q^1(z^1 | x^1) q(dx^1, z^2, \dots , z^d) }  \ge {\modar q^{1}}(z^1 | x^1_*) q ({\modar \emptyset, } z^2, \dots , z^d).
\end{align*}
We iterate the arguing above, to recover 
\begin{align*}
q(\bm{z}) & \ge \prod_{j = 1}^{d-1} {\modar q^j}(z^j | x^j_*) {\modar \int_{{\modar x^d \in \mathcal{X}^d}} {\modar q^d}(z^d | x^d)  q(\emptyset, \dots,\emptyset, dx^d)} \ge \prod_{j = 1}^{d} {\modar q^j(z^j | x^j_*),} 
\end{align*}
{\modar where we used $\int_{x^d \in \mathcal{X}^d}q(\emptyset, \dots,\emptyset, dx^d)=1$.}
An analogous lower bound holds true for $\tilde{q}(\bm{z})$. Thus, using also the bound in \eqref{eq: start step3 4.5}--\eqref{eq: upper bound diff qz qz' par somme}, we obtain 
\begin{align} \nonumber
&\abs{\log (\frac{q(\bm{z})}{\tilde{q}(\bm{z})})}\le \frac{\abs{q(\bm{z})-\tilde{q}(\bm{z})}}{\min(q(\bm{z}), \tilde{q}(\bm{z}))} \\ {\label{eq: diff q in proof main}}
& \qquad \le {\revnew \sum_{ S \subseteq \{1, ... , d \}} \prod_{h \in S} (e^{\alpha_{h}} - 1) d_{TV}(L_{\bm{X}^{(S)}}, L_{\tilde{\bm{X}}^{(S)}})}.  
\end{align}
We replace it in \eqref{eq: middle step3 5} which, together with \eqref{eq: start step3 4.5}--\eqref{eq: upper bound diff qz qz' par somme}, implies 
	\begin{multline*}
		d_{KL}(M, \tilde{M}) \le \Big{(} {\revnew \sum_{ S \subseteq \{1, ... , d \}} \prod_{h \in S} (e^{\alpha_{h}} - 1) d_{TV}(L_{\bm{X}^{(S)}}, L_{\tilde{\bm{X}}^{(S)}})} \Big{)}^2 \\ \times \int_{\mathcal{Z}^1\times\dots\times\mathcal{Z}^d} \prod_{j = 1}^{d} q^j(z^j | x^j_*) d\bm{\mu}(\bm{z}).
	\end{multline*}
The proof of Theorem \ref{th: main bound} is then complete, as the last integral can not be larger than one. 
\end{proof}
{\modar \begin{remark}\label{R: diff q}
		Let us stress that our proof also provides a control on the difference between the densities of $\bm{Z}$ and $\bm{\tilde{Z}}$ given by \eqref{eq: diff q in proof main}.
\end{remark}}

\subsection{Application to privatization of independent sampling}{\label{S: main sampling}}

This section applies the previously proven results to a scenario where the original samples $\bm{X}_1, \dots , \bm{X}_n$ composed by independent vectors distributed according to $\bm{X}$, are transformed into privatized samples $\bm{Z}_1, \dots , \bm{Z}_n$. The notation used in this section is consistent with that used in Section \ref{S:Pb_formulation} and will be used throughout the rest of the paper. Specifically, $X_i^j$ refers to the $j$-th component of the $i$-th individual $\bm{X}_i$. \\
\\
Assume we 
sample a random vector $(\bm{X}_1, \dots , \bm{X}_n)$ 
with product measure of the form {\modar $P^n(d\bm{x_1}, \dots , d\bm{x_n}):= \prod_{i = 1}^n P_i(d\bm{x_i})$.} We draw then an $\bm{\alpha}$ componentwise local differential private view of the sample $(\bm{X}_1, \dots, \bm{X}_n)$ through the privacy mechanism $\bm{Q}^n= (\bm{Q}_1, \dots, \bm{Q}_n)$, where $\bm{Q}_i= (Q_i^1, \dots , Q_i^d)$. 
The privatized samples $(\bm{Z}_1, \dots , \bm{Z}_n)$ 
is distributed according to some measure $M^n$. 
As we consider also the case where the algorithm is interactive, in general the measure $M^n$ is not in a product form (with respect to $i$). However, the proposition on tensorization inequality that follows will yield a result similar to that provided by independence. It will prove especially useful for applications. \\

\begin{proposition}{\label{prop: tensorization}}
Let $\alpha_j \ge 0$ and assume that $\bm{Q}^n$ guarantees the $\bm{\alpha}$- CLDP constraint as defined by the condition \eqref{eq: def local privacy}. Then, for any paired sequences of independent vectors $(\bm{X}_1, \dots , \bm{X}_n)$ and {\modar $(\bm{\tilde{X}}_1, \dots , \bm{\tilde{X}}_n)$} of distributions $P^n = \prod_{i = 1}^n P_i$ and $\tilde{P}^{ n} = \prod_{i = 1}^n \tilde{P}_i$ respectively, we have
\begin{equation}
\label{eq: dk M Mtilde cooperative}
d_{KL}(M^n, \tilde{M}^{ n}) \le \sum_{h = 1}^n \Big{(} {\revnew \sum_{ S \subseteq \{1, ... , d \}} \prod_{j \in S} (e^{\alpha_{j}} - 1) d_{TV}(L_{\bm{X}_h^{(S)}}, L_{\tilde{\bm{X}}_h^{(S)}})} \Big{)}^2.
\end{equation}
\end{proposition}

Assume now that the samples $(\bm{X}_1, \dots , \bm{X}_n)$ and $(\bm{\tilde{X}}_1, \dots , \bm{\tilde{X}}_n)$, in addition to being independent, are identically distributed {\modar and thus with laws
$P^n=P^{\otimes n}$ and $\tilde{P}^n=\tilde{P}^{\otimes n}$.} {\rev Moreover, we suppose that all the marginal laws of $\bm{X_h}$ and $\bm{\tilde{X}_h}$ are equals.} {\revnew This means that the law of $\bm{X}_h^{(S)}$ and $\tilde{\bm{X}}_h^{(S)}$ are the same for any strict subset $S \subset \{1, ... , d \}$.} Then, in analogy to Corollary \ref{cor: main}, the proposition above leads to the following corollary.

\begin{corollary}{\label{cor: main samples}}
Let $\alpha_j \ge 0$ and assume that $\bm{Q}^n$ guarantees the $\bm{\alpha}$- CLDP constraint as defined by the condition \eqref{eq: def local privacy}. Then, for any paired sequences of iid vectors $(\bm{X}_1, \dots , \bm{X}_n)$ and $(\bm{\tilde{X}}_1, \dots , \bm{\tilde{X}}_n)$ of distributions 
{\rev ${P}^n={P}^{\otimes n}$} and  $\tilde{P}^n=\tilde{P}^{\otimes n}$ 
which are such that 
{\revnew $L_{\bm{X}_h^{(S)}} = L_{\tilde{\bm{X}}_h^{(S)}}$ for any strict subset $S \subset \{1, ... , d \}$}, we have
\begin{align*}
d_{KL}(M^n, \tilde{M}^{n}) 
& \le n \Big{(} \prod_{j = 1}^d (e^{\alpha_{j}} - 1) \Big{)}^2  {\modar  d^2_{TV}(P^n, \tilde{P}^n).}
\end{align*}
\end{corollary}

The proof of Proposition \ref{prop: tensorization} follows next, while Corollary \ref{cor: main samples} is a direct consequence of the aforementioned proposition,
 remarking that in the two inner sums of \eqref{eq: dk M Mtilde cooperative} the only non-zero term is for {\revnew $S = \{1, ... , d \}$.} 

\begin{proof}[Proof of Proposition \ref{prop: tensorization}] 
We can introduce the marginal distribution of $\bm{Z}_h$ conditioned on $\bm{Z}_1= \bm{z}_1$, \dots , $\bm{Z}_{h- 1}= \bm{z}_{h-1}$. We denote it as 
$$M_h(\cdot | \bm{Z}_1= \bm{z}_1, \dots, \bm{Z}_{h- 1}= \bm{z}_{h-1} ) =: M_h(\cdot | \bm{z}_{1: h-1}). $$
Observe that for any $A_j \in \Xi_{\mathcal{Z}^j}$ and {\revnew $\bm{x} \in \bm{\mathcal{X}}$, $\bm{z}_{1}, \dots, \bm{z}_{h-1} \in \bm{\mathcal{Z}}$} it is
\begin{align*}
M_h(\prod_{j=1}^d A_j | \bm{z}_{1: h-1}) & = \int_{\bm{\mathcal{X}}} \prod_{j = 1}^d Q_h^{j}(A_j | X_h^j= x^{\revnew j}, \bm{Z}_1= \bm{z}_1, ... , \bm{Z}_{h- 1}= \bm{z}_{h-1})  P_{h}(dx^1, \dots , dx^d) \\
& = : \int_{\bm{\mathcal{X}}} \prod_{j = 1}^d Q_h^{j}(A_j | X_h^j, \bm{Z}_{1 : h-1})  P_{h}(dx^1, \dots , dx^d).
\end{align*}
Moreover, we introduce the notation $d_{KL}(M_h, \tilde{M}_h)$ for the integrated Kullback divergence of the conditional distributions on the $\bm{Z}_h$, which is 
\begin{equation*}
    \int_{\bm{\mathcal{Z}}^{h-1}} d_{KL}\Big(M_h (\cdot | \bm{z}_{1: h-1}), \tilde{M}_h (\cdot | \bm{z}_{1: h-1})  \Big) dM^{h-1}(\bm{z}_1, \dots , \bm{z}_{h-1}). 
\end{equation*}
Then, the chain rule for Kullback-Lieber divergences as gathered in Chapter 5.3 of \cite{29 Martin} provides 
$$d_{KL}(M^n, \tilde{M}^{n}) = \sum_{h = 1}^n d_{KL} (M_h, \tilde{M}_h). $$
By the definition of $\bm{\alpha}$-CLDP for sequentially interactive privacy mechanism provided in \eqref{eq: def local privacy}, the distribution $Q_h^{j}(A_j | X_h^j, \bm{Z}_{1 : h-1})$ 
is $\alpha_j$-differentially private for $X_h^j$. We can therefore apply Theorem \ref{th: main bound} on $d_{KL}\Big(M_h (\cdot | \bm{z}_{1: h-1}), \tilde{M}_h (\cdot | \bm{z}_{1: h-1})  \Big)$ which entails, together with the chain rule above,  
\begin{align*}
d_{KL}(M^n, \tilde{M}^{ n}) & = \sum_{h = 1}^n \int_{\bm{\mathcal{Z}}^{h-1}} d_{KL}\Big(M_h (\cdot | \bm{z}_{1: h-1}), \tilde{M}_h (\cdot | \bm{z}_{1: h-1})  \Big) dM^{h-1}(\bm{z}_1, \dots , \bm{z}_{h-1}) \\
& = \sum_{h = 1}^n \int_{\bm{\mathcal{Z}}^{h-1}} \Big{(} {\revnew \sum_{ S \subseteq \{1, ... , d \}} \prod_{j \in S} (e^{\alpha_{j}} - 1) d_{TV}(L_{(\bm{X}_h^{(S)}| \bm{z}_{1: h-1})}, L_{(\tilde{\bm{X}}_h^{(S)}| \bm{z}_{1: h-1})})} \Big{)}^2 \\
&  \times  dM^{h-1}(\bm{z}_1,\dots , \bm{z}_{h-1}),
\end{align*}
where we have denoted as {\revnew $L_{(\bm{X}_h^{(S)}| \bm{z}_{1: h-1})}$} the conditional distribution of {\revnew $\bm{X}_h^{(S)}$} given the first $h-1$ values $\bm{Z}_1, \dots , \bm{Z}_{h-1}$. Clearly in an analogous way {\revnew $L_{(\tilde{\bm{X}}_h^{(S)}| \bm{z}_{1: h-1})} $} is the conditional distribution of {\revnew $\tilde{\bm{X}}_h^{(S)}$} given the first $h-1$ values $\bm{Z}_1, \dots , \bm{Z}_{h-1}$. However, by construction, the random variables $\bm{X}_h$ are 
 independent, which implies that {\revnew $L_{(\bm{X}_h^{(S)}| \bm{z}_{1: h-1})} = L_{\bm{X}_h^{(S)}}$ and $L_{(\tilde{\bm{X}}_h^{(S)}| \bm{z}_{1: h-1})} = L_{\tilde{\bm{X}}_h^{(S)}}$}. It yields 
\begin{align*}
 d_{KL}(M^n, \tilde{M}^{n}) & \le  \sum_{h = 1}^n \Big{(} {\revnew \sum_{ S \subseteq \{1, ... , d \}} \prod_{j \in S} (e^{\alpha_{j}} - 1) d_{TV}(L_{\bm{X}_h^{(S)}}, L_{\tilde{\bm{X}}_h^{(S)}})} \Big{)}^2 \\
 & \times \int_{\bm{\mathcal{Z}}^{h-1}} dM^{h-1}(\bm{z}_1, \dots , \bm{z}_{h-1}).
\end{align*}
The proof is then concluded once we remark that the integral in $dM^{h-1}(\bm{z}_1, \dots , \bm{z}_{h-1})$ {\modar is equal to $1$.}
\end{proof}

{\rev

	\subsection{Contraction on $f$-divergence}
	In this section we present a result similar to \eqref{th: main bound}, but where we control a family of $f$-divergences different from the Kullback Lieber distance. For $f : \mathbb{R}_+ \to \mathbb{R} \cup \{\infty\}$ with $f(1)=0$ we define the $f$ divergence between two distributions as $D_f(P \| Q) =\int f\left(\frac{dP}{dQ}\right) dQ$. In the sequel, we focus on $f$-divergences with $f(t)=f_l(t)=\abs{t-1}^l$ for $l>1$.
	
	We are in the same context and use the same notations as in Section \ref{S: bounds divergences}. Recall that $\bm{Q}=(Q^1,\dots,Q^d)$ is a family of $d$ kernels, where $Q^j$ is a randomization from some space $\mathcal{X}^j$ to $\mathcal{Z}^j$. In this section we assume that for all $j$ there exists a reference measure $\mu^j$ on $\mathcal{Z}^j$ such that $\frac{dQ^j(z^j \mid X_j=x)}{d\mu^j (z^j)}$ is given as a positive density 
	$q(z^j \mid X^j=x) $. Moreover, we do not assume that the $\bm{\alpha}$-CLDP \eqref{eq: def local privacy non-int} constraint is valid. Instead, we assume that for some $l>1$,
	\begin{equation}\label{eq : divergence CLDP} 
		\forall j \in \{1,\dots,d\},
		\sup_{x,x^{\prime} \in \mathcal{X}^j}
		D_{f_l}(Q^j(\cdot\mid X^j=x') \| Q^j(\cdot\mid X^j=x) )\le (\varepsilon_j)^l,
	\end{equation}
	with some positive constants $\varepsilon_1,\dots,\varepsilon_d$.
	
	The following proposition is an extension of Proposition 8 in \cite{Duchi_Ruan24}. Its proof is given in the Appendix.
	\begin{proposition} \label{P: ajout main en divergence}
		 Assume $l>1$. We have
		\begin{equation}\label{eq: main with f divergence}
			D_{f_l}(M\|\tilde{M})^{1/l} \le {\revnew \sum_{ S \subseteq \{1, ... , d \}} \prod_{j \in S} \varepsilon_j \, d_{TV}(L_{\bm{X}^{(S)}}, L_{\tilde{\bm{X}}^{(S)}})}.
		\end{equation}
	\end{proposition}

}
\section{Applications to statistical inference}{\label{S: applications}}
{\modar In this upcoming section, one objective is to demonstrate the usefulness of the bounds on divergences between distributions that have been accumulated in Section \ref{S: main}. 
We will show versatile applications of these bounds in different statistical problems. First, we study how information about a private characteristic of an individual can be revealed by the public views of other characteristics of the same individual. For this problem, the results obtained in Section \ref{S: main} are insightful tools that lead us to introduce 
the quantity \eqref{eq: Delta ind} as the main parameter for measuring information leakage. 
In the following, we will also provide details on the estimation of {\revnew  joint moments} and density in a locally private and multivariate context. For these statistical problems, we construct explicit {\revnew estimators. The results of Section \ref{S: main} can be used to derive the rate optimality of these estimators relying on the Le Cam's two points method}. Finally, we propose adaptive versions of our estimators.}

{\revnew Let us stress that the paper \cite{Martin} proposes extension of Fano and Assouad methods in the context of local differential privacy. Generalization of such methods for componentwise LDP is an interesting issue. Such extensions would be useful in statistical problems where the Le Cam's two points method is insufficient to provide useful lower bounds. We leave these extensions for further research.
}

\subsection{Effective privacy level}{\label{s:privacy level}}

When the data $X^1,\dots,X^d$ are disclosed by independent channels and with different privacy levels $\alpha_1,\dots,\alpha_d$, a natural question is how precisely the value of one marginal, say $X^1$, could be revealed by the observations of $Z^1,\dots,Z^d$, which are publicly available. {\revd This question leads to relate the values $\alpha_1,\dots,\alpha_d$ with the effective level of protection for the raw data $X^1$.}

The case where some variable ${X}^1\in\mathcal{X}^1$ is privatized using a Markov kernel into the public data ${Z^1}\in\mathcal{Z}^1$ is the situation studied in \cite{WassermanZhou10} and \cite{Martin}. It is known from \cite{Martin} that if the privacy channel is $\alpha$-LDP, then  for all $x^1,x^{1\prime} \in {\mathcal{X}^1}$ and $\psi : {\mathcal{Z}^1} \to \{x^1,x^{1\prime}\}$ 
we have
\begin{equation}\label{eq: proba error disclose}
	\frac{1}{2} \P\left( \psi(Z^1)\neq x^1 \mid X=x^1 \right) 
	+\frac{1}{2} \P\left(  \psi(Z^1)\neq x^{1\prime} \mid X=x^{1\prime} \right) \ge \frac{1}{1+e^\alpha},	
\end{equation}
 This means that even if someone accesses two values $x^1$ and $x^{1\prime}$ from the raw data set, it will be impossible for them to determine with a high level of certainty which of the values corresponds to a specific observation, denoted as $Z^1$. Any attempt to make a decision in this regard would result in an error, albeit with minimal probability. 

If a vector $\bm{X}$ is privatized with independent channel for each components {\revnew and} the components of $\bm{X}$ are independent, then the result of \cite{Martin} applies componentwise. Indeed,  the observations of $Z^2,\dots,Z^d$ carry no information about the value of $X^1$ and thus recovering information on $X^1$ from $\bm{Z}$ or from $Z^1$ is equivalent and 
a result like \eqref{eq: proba error disclose} applies with $\alpha=\alpha_1$.  

The situation is more intricate if the components of $\bm{X}$ are dependent, as the observation of $Z^2,\dots,Z^d$ brings informations on $X^1$. In the extreme situation where all the components of $\bm{X}$ are almost surely equal $X^1=\dots=X^d$, it is clear that the observation of $\bm{Z}=(Z^1,\dots,Z^d)$ is a repetition of $d$ independent views of the same raw data $X^1$ with different privacy level. Thus, {\modar this mechanism is equivalent to the privatization of the single variable $X^1 \in \mathcal{X}^1$ through a channel taking the 
	value 
$\bm{Z}=(Z^1,\dots,Z^d)\in\bm{\mathcal{Z}}$. By independence of the $Z^j$'s conditional to $X^1$ and \eqref{eq: def local privacy non-int}, we can check that this mechanism is a non componentwise $\alpha$-LDP view of $X^1$ with 
$\alpha=\sum_{j=1}^d \alpha_j$. In turn, the lower bound \eqref{eq: proba error disclose} for deciding the value of $X^1$ from the observation of $\bm{Z}$ holds true with $\alpha=\sum_{j=1}^d \alpha_j \ge \alpha_1$. This evaluates how the privacy of $X^1$ is deteriorated by observation of the side-channels $Z^2,\dots,Z^d$ in the worst case scenario $X^1=X^2=\dots=X^d$.
}

In intermediate situation, we need to introduce some quantity which
 assesses the independence of $(X^2,\dots,X^d)$ on $X^1$.
Let us denote by $q(dx^2,\dots,dx^d\mid X^1=x^1)$ the conditional distribution of $X^2,\dots,X^d$ conditional to  $X^1=x^1$. We let
\begin{equation} \label{eq: Delta ind}
	\Delta_\text{ind}:=\sup_{x^1,x^{1\prime } \in \mathcal{X}^1} d_{TV}\left(q(dx^2,\dots,dx^d\mid X^1=x^1),q(dx^2,\dots,dx^d\mid X^1=x^{1\prime})\right),
\end{equation}
which quantifies how close $(X^2,\dots,X^d)$ are from being independent from $X^1$.
We have indeed $\Delta_\text{ind} \in [0,2]$ and $\Delta_\text{ind}=0$ when $X^1$ is independent from $(X^2,\dots,X^d)$. We let $m(z_1,\dots,z_d \mid X^1=x^1)$ the density with respect to $\bm{\mu}$ of the law of $(Z^1,\dots,Z^d)$ conditional to $X^1=x^1$. We have, using the notation of Section \ref{S: main}
\begin{equation}\label{eq: density m par kernel}
	m(z^1,\dots,z^d \mid X^1=x^1)=\int_{\prod\limits_{j=2}^d \mathcal{X}^j} 
	\prod_{j=1}^d q^j(z^j\mid x^j) q(dx^2,\dots,dx^d\mid X^1=x^1).
\end{equation}
To elaborate, the function $\bm{z} \mapsto m(\bm{z}\mid  X^1=x^1)$ is the density of the channel which gives $\bm{Z}$ as a public view of the marginal $X^1$, gathering the information directly revealed by the channel $q^1$ and indirectly by the channels $q^j$, $j \ge 2$. The following proposition gives an upper bound for the privacy level of this channel. It is essentially a consequence of Theorem \ref{th: main bound}.

\begin{proposition}{\label{prop: proba info}}
Let $\alpha_j \ge 0$, and assume that $\bm{Q}=(Q^1,\dots,Q^d)$ guarantees the $\bm{\alpha}$-CLDP constraint as defined by the condition  \eqref{eq: def local privacy non-int}. Assume that
there exists $\alpha_{\text{max}}$ such that $\alpha_j \le \alpha_{\text{max}}$ for $j\in \{2,\dots,d\}$.
	Then, we have
	\begin{equation}\label{eq: main first application}
\sup_{x^1,x^{1\prime}\in \mathcal{X}^1} ~ \frac{m(z^1,\dots,z^d \mid X^1=x^1)}{m(z^1,\dots,z^d \mid X^1=x^{1\prime})}
\le \exp\left( \alpha_1 + \alpha_{\text{max}}\times(d-1) \Delta_\text{ind}\right).
	\end{equation}
	
\end{proposition}
\begin{remark}
	If $x^1,x^{1\prime}\in \mathcal{X}^1$ and $\psi: \bm{\mathcal{Z}} \to \{x^1,x^{1\prime}\}$ is any measurable {\rev map}, then the average probability of mispredicting $X^1$ from $\bm{Z}$, 
	$\frac{1}{2} \P\left(\psi(\bm{Z})\neq x^1 \mid X^1=x^1 \right)
	+\frac{1}{2} \P\left(  \psi(\bm{Z})\neq x^{1\prime} \mid X^1=x^{1\prime} \right)  $ is lower bounded by the same quantity as in Equation \eqref{eq: proba error disclose} where $\alpha$ is replaced by $\alpha_1 + \alpha_{\text{max}}\times(d-1) \Delta_\text{ind}$.
\end{remark}

\begin{proof}[Proof of Proposition \ref{prop: proba info}]
We will apply the results of Section \ref{S: main} with two well chosen probabilities $\widetilde{P}$, $\widetilde{P}'$ on $\widetilde{\bm{\mathcal{X}}}=\prod_{j=2}^d \mathcal{X}^j$. We fix $x^1,x^{1\prime}\in\mathcal{X}^1$ and let $P$ be the measure on $\widetilde{\bm{\mathcal{X}}}$ given by 
$$
\widetilde{P}(dx^2,\dots,dx^d)=q(dx^2,\dots,dx^d\mid X^1=x^1).
$$
We define $\widetilde{P}'$ analogously with $x^{1\prime}$ in place of $x^1$. We denote by $\widetilde{M}$ the measure on $\widetilde{\bm{\mathcal{Z}}}=\prod_{j=2}^d \mathcal{Z}^j$ of the privatized view of $P$ through the kernel $\widetilde{Q}=(Q^2,\dots,Q^d)$. In an analogous way, $\widetilde{M}'$ is the law of a privatized version of $\widetilde{P}'$. 
{\modar Let us denote by $\widetilde{m}(z^2,\dots,z^d)$ and $\widetilde{m}^\prime(z^2,\dots,z^d)$  the densities of $\widetilde{M}$ and $\widetilde{M}'$. As emphasized in Remark \ref{R: diff q}, the equation \eqref{eq: diff q in proof main} provides a control on the difference between $\widetilde{m}$ and 
	$\widetilde{m}'$, which yields to}
	\begin{equation*}
		\frac{\abs{\widetilde{m}(z^2,\dots,z^d)-\widetilde{m}^\prime(z^2,\dots,z^d)}}{\widetilde{m}^\prime(z^2,\dots,z^d)}
		\le 
		{\revnew \sum_{ S \subseteq \{2, \dots , d \}}
\prod_{j\in S} (e^{\alpha_j}-1)
		d_{TV}\left(\tilde{P}_{\mid \bm{X}^{(S)}} ,\tilde{P}'_{\mid \bm{X}^{(S)}}\right)}
	\end{equation*}
{\revnew where $\tilde{P}_{\mid \bm{X}^{(S)}}$  is the restriction of the measure $\tilde{P}$ on $\prod_{j\in S} \mathcal{X}^{j}$, and 
$\tilde{P}'_{\mid \bm{X}^{(S)}}$ is defined analogously.} We use the bound
{\revnew $	d_{TV}\left(\tilde{P}_{\mid \bm{X}^{(S)}} ,\tilde{P}'_{\mid \bm{X}^{(S)}}\right)\le 
d_{TV}\left(\tilde{P},\tilde{P}^\prime\right)$} to deduce,
\begin{align*}
	\frac{\abs{\widetilde{m}(z^2,\dots,z^d)-\widetilde{m}^\prime(z^2,\dots,z^d)}}{\widetilde{m}^\prime(z^2,\dots,z^d)}
&	\le 	{\revnew \sum_{ S \subseteq \{2, \dots , d \}}
	\prod_{j\in S} (e^{\alpha_j}-1)
	d_{TV}\left(\tilde{P}_{\mid \bm{X}^{(S)}} ,\tilde{P}'_{\mid \bm{X}^{(S)}}\right)}
	\\
&	\le \sum_{k=1}^{d-1} \binom{{\revnew d}-1}{k} (e^{\alpha_{\text{max}}}-1)^k
	d_{TV}\left(\tilde{P},\tilde{P}^\prime\right), \text{ using $\alpha_j \le \alpha_{\text{max}}$ for $j \ge 2$,}
	\\
	&\le \left[ e^{\alpha_{\text{max}}\times({\revnew d}-1)} -1\right] d_{TV}\left(\tilde{P},\tilde{P}^\prime\right), \text{ from the binomial formula.} 
\end{align*}
The definitions of $\widetilde{P}$ and $\widetilde{P}^\prime$ as conditional distributions imply that  $d_{TV}\left(\tilde{P},\tilde{P}^\prime\right) \le \Delta_{\text{ind}}$, and thus we deduce
\begin{equation*}
	\frac{\widetilde{m}(z^2,\dots,z^d)}{\widetilde{m}^\prime(z^2,\dots,z^d)} \le 1 + \left[ e^{\alpha_{\text{max}}\times({\revnew d}-1)} -1\right] \Delta_{\text{ind}}.
\end{equation*}
Using the simple inequality $1+(e^\alpha-1)q \le e^{\alpha q}$ for $\alpha,q \ge 0$, we get
\begin{equation} \label{eq: ratio m tilde m tilde prime}
	\frac{\widetilde{m}(z^2,\dots,z^d)}{\widetilde{m}^\prime(z^2,\dots,z^d)} \le  e^{\alpha_{\text{max}}\times({\revnew d}-1) \Delta_{\text{ind}}}.
\end{equation}
Recalling that $\widetilde{m}(z^2,\dots,z^d)$ is the density of the privatized view of $\tilde{P}$ through $\widetilde{Q}=(Q^2,\dots,Q^d)$, we have
\begin{equation*}
		\widetilde{m}(z^2,\dots,z^d)=\int_{\prod\limits_{j=2}^d \mathcal{X}^j} 
		\prod_{j=2}^d q^j(z^j\mid x^j) q(dx^2,\dots,dx^d\mid X^1=x^1), 
\end{equation*}
and thus by comparison with \eqref{eq: density m par kernel}
\begin{equation*}
m(z^1,\dots,z^d \mid X^1=x^1)=q^{1}(z^1\mid x^1)\widetilde{m}(z_2,\dots,z_d).
\end{equation*}
An analogous relation holds true for $m'$ and in turn,
\begin{equation*}
\frac{m(z^1,\dots,z^d \mid X^1=x^1)}{m(z^1,\dots,z^d \mid X^1=x^{1\prime})}
=
\frac{q(z^1\mid x^1)}{q(z^1 \mid x^{1\prime})}
\frac{\widetilde{m}(z^2,\dots,z^d)}{\widetilde{m}^\prime(z^2,\dots,z^d)}.
\end{equation*}	
Now, the proposition is a consequence of 
\eqref{eq: def density local privacy} and
\eqref{eq: ratio m tilde m tilde prime}.
\end{proof}
{\revnew 
	\begin{remark}
		Let us point out that, except in the case where $X^1$ is independent of $(X^2,\dots,X^d)$, the upper bound 
		\eqref{eq: main first application} is uninformative 
		when $\alpha_{\text{max}} \to \infty$. Indeed, its RHS converges to infinity.  However, it is impossible to find, in general, a confidentiality guarantee as strong as a LDP constraint on a variable $X^1$ when observing publicly a dependent variable $X^2$. It is sufficient to consider $\bm{X}=(X^1,X^2)$ where $X^1$ is non constant and the law of $X^2$ conditional to $X^1=x^1$ is $\mathcal{N}(x^1,1)$. In this setup, by writing the ratio of Gaussian densities, we show $\sup_{x^1,x^{1\prime}, x^2}\frac{q(dx^2 \mid x^1)}{q(dx^2 \mid x^{1\prime})}=\infty$. Thus, we see that in this example, corresponding formally to a choice $\alpha_2=\infty$, the LHS of \eqref{eq: main first application} must be equal to $+\infty$.
	\end{remark}
}

\subsection{Locally private {\rev joint moment} estimation}{\label{s: joint moment}}

In this section we assume that {\rev $\bm{X}=(X^1,\dots,X^d)$ is a $d$-dimensional vector, for which we want to estimate the moment $\gamma=\E\left[\prod_{j=1}^d X^j \right]$ under local differential privacy constraints. Again, the components are made public separately. A particular application of our results for the case $d=2$ will give the estimation of the covariance and correlation discussed in Section \ref{Sss: application to covariance}.}



\subsubsection{Local differential private estimator}\label{Sss:joint moment estimator}

We assume that $\bm{X}_1,\dots,\bm{X}_n$ are $n$ iid copies of {\rev $\bm{X}=(X^1,\dots,X^d)$.} As in this paper we stick to the framework of local differential privacy, we want to introduce {\revnew a privatization} procedure to transform the $X_i^j$ to some 
$$Z_i^j \sim q_j(dz | X_i^j) = q_j(z| X_i^j) dz$$
which satisfies the condition of local differential privacy, as in \eqref{eq: def density local privacy}. In particular, the privacy mechanism we consider in this example is non-interactive. 

It is well-known that adding centered Laplace distributed noise on bounded random variables provides $\alpha$ differential privacy (cfr \cite{Martin}, \cite{Kroll}, \cite{RS20}). This motivates our choice for the anonymization procedure, which consists in constructing the public version of the $\bm{X}_i$ by using a Laplace mechanism with an independent channel for each component. Let us denote $\mathcal{E}^i_j$ for $(i,j)\in\{1,\dots,n\}{\rev \times\{1,\dots,d\}}$ a family of independent random variables, such that $(\mathcal{E}^i_j)_i$ are iid sequences with law $\mathcal{L}(\frac{2T^{(j)}}{\alpha_j}) $ for {\rev $j=1,\dots,d$.} The truncation $T^{(j)}>0$ will be specified later. We assume that the variables $\mathcal{E}^j_i$ are independent from the data $\bm{X}_1,\dots,\bm{X}_n$ and we set
\begin{equation} \label{eq: def Z covariance}
Z^j_i=\tronc{X^j_i}{T^{(j)}}+\mathcal{E}^j_i, \quad \forall (i,j)\in\{1,\dots,n\}{\rev \times\{1,\dots,d\},}
\end{equation}
where $\tronc{x}{T}=\max(\min(x,T),-T)$.

 Denoting by $z\mapsto q^j(z\mid X^j_i=x)$ the density of the privatized data $Z^j_i$ conditional to $X^j_i=x$ for {\rev $j\in\{1,\dots,d\}$,} it is 
{\rev a direct application of Lemma \ref{L: preuve Laplace mech generique} to 
check that the local differential privacy control \eqref{eq: def density local privacy} holds true, as stated in the following lemma.}
\begin{lemma}{\label{l: ldp covariance}}
For any $i \in \{1, \dots, n \}$ and {\rev $j \in \{1,\dots,d\}$,} the random variables $Z^j_i=\tronc{X^j_i}{T^{(j)}}+\mathcal{E}^j_i$, with $\mathcal{E}_i^j$ iid $\sim \mathcal{L}(\frac{2T^{(j)}}{\alpha_j})$, are $\alpha_j$ differential private views of the original $X_i^j$.
\end{lemma}


Assume that, {\rev for $j \in \{1,\dots,d\}$ we have $X^j \in \mathbf{L}^{k_j}$
	for $k_j >1$, with the condition $\sum_{j=1}^d \frac{1}{k_j} <1$.
	By H\"older's inequality, it ensures that $\E[|X^1\times\dots\times X^d|]<\infty$.
The goal is to estimate $\gamma:=\E[X^1\times\dots\times X^d]$.}

 The estimation of {\rev the expectation of the marginals  $m^{(j)}:=\E[X^j]$}
is discussed in \cite{Martin}, from which we recall the result. We will state later the result on the estimation of the cross {\rev term
$\gamma=\E[X^1\times\dots\times X^d]$.}

\noindent Let 
{\rev
\begin{align} \label{eq: def m hat est cov}
&	\hat{m}^{(j)}_n:=\frac{1}{n}\sum_{i=1}^n Z_i^j, \text{for $j\in \{1,\dots,d\}$},
	\\ \label{eq: def gamma hat est cov}
&	\hat{\gamma}_n:=\frac{1}{n}\sum_{i=1}^n \prod_{j=1}^d Z_i^j.
\end{align}}
\begin{theorem}[Corollary 1 in \cite{Martin}]\label{th: recall Duchi est moyenne}
	Let $0<\alpha_j\le1$. Assume $k_j>1$ and set $\widetilde{T}^{(j)}=(n\alpha^2_j)^{1/(2k_j)}$ for $j \in {\rev \{1,\dots,d\}}$. Then, there exists $c>0$ such that for all $n\ge 1$, $j\in {\rev \{1,\dots,d\}}$,
	\begin{equation*}
		\E \left[\abs*{{\revd \hat{m}^{(j)}_n}-m^{(j)}}^2 \right] \le c (n\alpha^2_j)^{-\frac{k_j-1}{k_j}}.
	\end{equation*}
\end{theorem}
{\revnew We would like to emphasize that the constant \( c \) appearing above is dependent on the true moments of \( X^j \).}

 In Corollary 1 in \cite{Martin}, the privacy level $\alpha$ is the same for all components.  However, the result is useful also in our context, as we plan to apply Corollary 1 in \cite{Martin} with $d=1$ separately to each components of $\bm{X}$.
By \cite{Martin}, the choice of truncation $\widetilde{T}^{(j)}=(n\alpha^2_j)^{1/(2k_j)}$ is optimal when estimating $m^{(j)}$.
The result for the estimation of the {\rev joint moment} is the following. 
\begin{theorem}\label{th: upper bound joint moment covariance}
	Let $\alpha_j\le1$. Assume {\rev  $1/k_1+\dots+1/k_d <1$} and set {\rev $T^{(j)}=(n
	\prod_{l=1}^d \alpha^2_l)^{1/(2k_j)}$} for $j\in {\rev \{1,\dots,d\}}$. Then, there exists $c>0$, such that for all $n \ge 1$,
	\begin{equation}\label{eq: borne sup gamma}
	\E \left[ \abs*{\hat{\gamma}_n - \gamma}^2\right] \le c (n\prod_{j=1}^d\alpha^2_j )^{-\frac{\overline{k}-{\rev d}}{\overline{k}}},\\
	\end{equation}
where {\rev $\overline{k}=d\left(\frac{1}{k_1}+\dots+\frac{1}{k_d}\right)^{-1}>d$ is the harmonic mean of $k_1, \dots k_d$.} The constant $c$ does not depend on {\rev $\alpha_j$,} $n$, as soon as {\rev $n\alpha^2_1\times\dots\times \alpha_d^2 \ge 1$.}
\end{theorem}
{\modar  The proof of this theorem is in the Appendix. It relies on a bias-variance trade-off, and the choice for $T^{(j)}$, given in the statement, is in this regard optimal.}
\begin{remark}{\label{R: covariance_dim_d}}
 The upper bound provided by Theorem 
 \ref{th: upper bound joint moment covariance} can be compared with the case where the private data 
 {\rev $\bm{X}_i = (X_i^1, \dots, X_i^d)$} is disclosed using a single channel that accesses 
 {\rev all components of $\bm{X}_i$ and satisfying LDP constraint with parameter 
 ${\revnew \overline{\alpha}}$}   . 
 In this scenario, we can apply the results of Section 3.2.1 in \cite{Martin} to estimate the mean of the iid {\rev one dimensional} private data $(\Gamma_i)_{i=1,\dots,n}$, where {\rev $\Gamma_i=\prod_{j=1}^dX^j_i$,} using a locally differentially privatized version of $\Gamma_i$. By applying H\"older's inequality, we can see that $\Gamma_1 \in \mathbf{L}^{\overline{k}/{\rev d}}$, and hence by Corollary 1 in \cite{Martin}, there exists an estimator $\tilde{\gamma}_n$ such that
	$$
	\E \left[ (\tilde{\gamma}_n - \gamma)^2\right] \le c (n{\revnew \overline{\alpha}}^2)^{-\frac{\overline{k}/{\rev d}-1}{\overline{k}/{\rev d}}}=c(n{\revnew \overline{\alpha}}^2)^{-\frac{\overline{k}-{\rev d}}{\overline{k}}},
	$$
	where {\rev ${\revnew \overline{\alpha}}$} is the LDP level when disclosing the $\bm{X}_i$'s. \\
    
 We can conclude that the rate exponent for estimating $\gamma$ is unchanged when the data are disclosed using independent channels for each component, compared to a situation where both components can be accessed before publicly releasing the data. However, the effective number of data is reduced from
 {\rev  $n{\revnew \overline{\alpha}}^2$ to $n\prod_{j=1}^d \alpha_j^2$.} 
 	If we consider the special case where {\rev $\alpha_1=\dots=\alpha_d=\alpha$ we know that, by Lemma \ref{L:comp CLDP_LDP}, the CLDP kernel is a special case of LDP kernel with ${\revnew \overline{\alpha}}=d\alpha$. Thus,} 	
 	it is evident that the loss is significant for small values of {\rev $\alpha$, as $n \alpha^{2d}< n d^2 \alpha^2$}. 
 	  On the other hand, the loss can be moderate if {\rev the $\alpha_j$ are close to $1$.} We will see in Section \ref{Sss:lower bound joint_moment} that this loss is unavoidable.
\end{remark}
{\rev 
\begin{remark}
	The construction of the estimator $\hat{\gamma}_n$ necessitates to choose the truncation levels $T^{(j)}$'s. The optimal choice is given in the statement of Theorem \ref{th: upper bound joint moment covariance} and relies on the number of finite moments $k_j$'s of each components of the vector $\bm{X}$. In practice, these constants $k_j$'s are unknown and the choice
	 of the optimal $T^{(j)}$'s seems unfeasible. We will present in Section \ref{Sss: adaptive} an adaptive version of the estimator $\hat{\gamma}_n$ where the choice of the truncation levels is data-driven, while preserving the rate of convergence of the estimator, up to a log-term loss. 
\end{remark}
}

{\revnew
\begin{remark}
We emphasize that the current results are firmly rooted in the scenario where all components are collected privately, i.e., \(\alpha_j \leq 1\) for all \(j\). It is an interesting 
direction to relax this assumption and consider cases where some components are collected privately while others are publicly available (i.e., \(\alpha_j = \infty\) for some \(j \in \{1, \dots, d\}\)).  

On one hand, we believe it is feasible to adapt the proofs of our upper bounds to account for the fact that only certain privacy levels would contribute to the final convergence rates. On the other hand, our current lower bounds become entirely uninformative in this setting, and deriving meaningful lower bounds for such a mixed-privacy framework poses significantly challenges.  

Consequently, we have decided to leave a detailed analysis of scenarios involving both privately and publicly collected variables for further investigation in future work.
\end{remark}
}

{\rev 
\subsubsection{Application to the covariance estimation} \label{Sss: application to covariance}
 We now focus on the estimation of the covariance between two random variables when the associated data are privatized in the componentwise way. For simplicity, we assume that we are dealing with a $2$-dimensional vector
	$\bm{X}=(X^1,X^2)$ with $k_1^{-1}+k_2^{-1} < 1$. It ensures that $\E[|X^1X^2|]<\infty$, by H\"older's inequality. 	The goal is to estimate $\theta:=\text{cov}(X^1,X^2)=\E[X^1X^2]-\E[X^1]\E[X^2]$. We assume that $n$ private
	 data $(\bm{X}_i)_{i=1,\dots,n}$ are iid and that the statistician can observe public views obtained through a CLDP mechanism with parameter $\bm{\alpha}=(\alpha_1,\alpha_2)$. We apply the results of Section \ref{Sss:joint moment estimator}. Consistently with the notations of this section, we have $\hat{\gamma}_n=\frac{1}{n} \sum_{i=1}^n Z^1_i Z^2_i$, and define $\hat{\theta}_n:=	\hat{\gamma}_n-\hat{m}^{(1)}_n\hat{m}^{(2)}_n$. 
		
	The result for the estimation of the covariance is the following. Its proof is given in the Appendix. 
	\begin{corollary}\label{C: upper bound covariance}
		Let $\alpha_j\le1$. Assume $1/k_1+1/k_2 <1$ and set $T^{(j)}=(n\alpha^2_1\alpha^2_2)^{1/(2k_j)}$ for $j\in \{1,2\}$. Then, there exists $c>0$, such that for all $n \ge 1$,
		\begin{equation}\label{eq: borne sup theta}
			\E \left[ \abs*{\hat{\theta}_n - \theta}^2\right] \le c (n\alpha^2_1\alpha^2_2)^{-\frac{\overline{k}-2}{\overline{k}}},
		\end{equation}
		where $\overline{k}=2\left(\frac{1}{k_1}+\frac{1}{k_2}\right)^{-1}>2$ is the harmonic mean of $k_1$, $k_2$. The constant $c$ does not depend on $\alpha_1$, $\alpha_2$ $n$, as soon as $n\alpha^2_1\alpha_2^2 \ge 1$.
	\end{corollary}
}

{\revnew We would like to emphasize that, as in Theorems \ref{th: recall Duchi est moyenne} and \ref{th: upper bound joint moment covariance} above, the constant \( c \) appearing in Equation \eqref{eq: borne sup theta} depends on the true moments of \( \bm{X} \).}

\begin{remark}
	By definitions \eqref{eq: def Z covariance}--\eqref{eq: def gamma hat est cov},  when estimating $\theta=\gamma-m^{(1)}m^{(2)}$, we use  the same truncation levels $T^{(j)}$
		for the estimation of $\gamma$,
		$m^{(1)}$ and $m^{(2)}$. It would be possible to use the optimal levels $\widetilde{T}^{(j)}=(n\alpha^2_j)^{1/(2k_j)}$ for the estimation of $m^{(1)}$, $m^{(2)}$ and the optimal levels $T^{(j)}=(n\alpha^2_1\alpha^2_2)^{1/(2k_j)}$ for the estimation of $\gamma$. However, this approach would necessitate publicly disclosing two values for each private data point: one corresponding to the truncation level $\widetilde{T}^{(j)}$ and one with level ${T}^{(j)}$. As a result, the overall privacy of the procedure would be reduced. In Section \ref{Sss: adaptive}, we will explore another scenario where we must disclose multiple public values for each private data point.
\end{remark}
{\rev
\begin{remark}
Assuming that $k_1>2$ and $k_2>2$, it is also possible to estimate the correlation	between the two variables $X^1$ and $X^2$. From the definition $\mathop{cor}(X^1,X^2)=\mathop{cov}(X^1,X^2)/\sqrt{\mathop{var}(X^1)\mathop{var}(X^2)}$ and since the covariance is estimated according to Corollary \ref{C: upper bound covariance}, it remains to estimate the variance of $X^1$ and $X^2$ and deduce an estimate of the correlation as a ratio. Since the estimation of the variance of each variable reduces to the estimation on the marginal laws, we can use the result in \cite{Martin}, recalled by Theorem \ref{th: recall Duchi est moyenne}. As $|X^1|^2 \in \mathbf{L}^{k_1/2}$, we can derive that the rate of estimation of $\E[|X^1|^2]$ is $(n \alpha_1^2)^{\frac{k_1/2-1}{2 k_1 /2}}=(n \alpha_1^2)^{1/2-1/k_1}$. We can deduce that it is possible to estimate $\mathop{var}(X^1)$ with some estimator converging with rate 
$(n \alpha_1^2)^{1/2-1/k_1}$, and analogously  $\mathop{var}(X^2)$ is estimated with rate 
 $(n \alpha_2^2)^{1/2-1/k_2}$. On the other hand, from Corollary \ref{C: upper bound covariance} the covariance is estimated at rate $(n \alpha_1^2\alpha_2^2)^{\frac{1}{2}-\frac{1}{2k_1}-\frac{1}{2k_2}}$. If $k_1=k_2=k$, and using that $\alpha_j \le 1$, we see that the slower rate is given by the estimation of the covariance. Consequently, the correlation coefficient can be estimated with an estimator having at least a rate given by $(n\alpha_1^2\alpha_2^2)^{\frac{1}{2}-\frac{1}{{\revd \overline{k}}}}$. If $k_1 \neq k_2$, the rate of this estimation procedure is at most the worst of the three rates, which is now dependent on the relative positions of $\alpha_1$ and $\alpha_2$. 
\end{remark}
}

\subsubsection{Lower bound for the {\rev joint moment} estimation}
\label{Sss:lower bound joint_moment}
For {\rev $k_1>1,\dots,k_d>1$ with $1/k_1+\dots +1/k_d<1$, we introduce the notation
\begin{equation*}
	\mathcal{P}_{k_1,\dots,k_d}=\left\{ P, \quad \text{probability on $\mathbb{R}^d$ such that } \E_P[|X^j|^{k_j}] \le 1, \text{ for $j \in \{1,\dots,d\}$ }  \right\},
\end{equation*}
where $\bm{X}=(X^1,\dots,X^d)$ is the canonical random variable on $\mathbb{R}^d$.} For {\rev $P \in \mathcal{P}_{k_1,\dots,k_d}$,} we set
\begin{equation*}
	\gamma(P)=\E_P\left[{\rev \prod_{j=1}^d X^j}\right]. 
\end{equation*}
{ We denote by $\mathcal{Q}_{\bm{\alpha}}$ the set of 
	privacy mechanisms, where for simplicity we restrict ourself to non interactive kernels. Thus, {\rev $\bm{Q}=(Q^1,\dots,Q^d) \in  \mathcal{Q}_{\bm{\alpha}}$} is such that
	$Q^j$ is a Markov kernel from $\mathcal{X}^j=\mathbb{R}$ to some measurable space
	$(\mathcal{Z}^j,\Xi_{\mathcal{Z}^j})$, and the condition \eqref{eq: def local privacy non-int} is satisfied for {\rev $j=1,\dots,d$}. 
}

\noindent The private data are given by the iid sequence $(\bm{X}_i)_{i=1,\dots,n}$.
We assume that the public data $(\bm{Z}_i)_{i=1,\dots,n}$ are given by the non interactive mechanism where the variable $Z_{i}^j$
is drawn according to the law $Q^j(d z \mid X_i^j)$.

\noindent We introduce the minimax risk
 $$
	\mathcal{M}_n(\gamma(\mathcal{P}_{{\rev k_1, \dots, k_d}}),\bm{\alpha})=\inf_{\bm{Q} \in \mathcal{Q}_{\bm{\alpha}}}
	\inf_{\hat{\gamma}_n} 
	\sup_{ {P\in\mathcal{P}_{{\rev k_1, \dots, k_d}}}} 
	\E_P\left[\left( \hat{\gamma}_n - \gamma(P) \right)^2 \right],
	$$ 
where $\hat{\gamma}_n$ is any $\hat{\gamma}_n((Z_i^j)_{\substack{1\le i \le n\\1\le j \le {\rev d}}})$ measurable function from {\rev $\big(\prod_{j=1}^d \mathcal{Z}^j\big)^n$.
taking values in $\mathbb{R}$, with finite second moment.}
\begin{theorem}{\label{th: lower joint}}
	There exists some constant $c$ such that,
	\begin{equation*}
		\mathcal{M}_n(\gamma(\mathcal{P}_{{\rev k_1,\dots,k_d}}),\bm{\alpha}) \ge 
		{\rev c(n\prod_{j=1}^d \abs{e^{\alpha_j} - 1}^{2}  )^{-\frac{\overline{k}-d}{\overline{k}}} }
	\end{equation*}
	for all $n \ge 1$, $n \prod_{j=1}^d \abs{e^{\alpha_j} - 1}^{2} \ge 1$.
\end{theorem}

\begin{remark}
	Comparing with Theorem \ref{th: upper bound joint moment covariance}, we see that when {\rev $\alpha_j\le 1$} the rate $(n\prod_{j=1}^d\alpha_j^2)^{\frac{\overline{k}-{\rev d}}{\overline{k}}}$ achieved by the estimator of Section \ref{Sss:joint moment estimator}. 
	can not be improved.
\end{remark}

\begin{proof}
	{\rev  To get a lower bound for $\mathcal{M}_n(\gamma(\mathcal{P}_{k_1,\dots,k_d}),\bm{\alpha})$ we want to apply the two hypothesis method (see for example Section 2.3 in \cite{Tsy}). We} need to construct $P$ and $P^*$ such that
	\begin{enumerate}
		\item $P$, $P^*$ are elements of {\rev $\mathcal{P}_{k_1,\dots,k_d}$,}
		\item $\exists c > 0$ with $\abs{\gamma(P)-\gamma(P^*) }\ge c (n {\rev \prod_{j=1}^d \abs{e^{\alpha_j}-1}^2})^{-\frac{\overline{k}-{\rev d}}{2\overline{k}}}$,
		\item $\exists \epsilon_0 > 0$ such that $d_{KL} \Big( Law((\bm{Z}_{i})_{i = 1, \dots , n}), Law((\bm{Z}^*_{i})_{i = 1, \dots, n})  \Big) < \epsilon_0 < 2$,
	\end{enumerate} 
	where  $\bm{Z}_i={\rev (Z_i^1,\dots,{\revd Z_i^d})}$ for $i=1,\dots,n$ are the public views of  
	$\bm{X}_i={\rev (X_i^1,\dots,X_i^d)}$, $i=1,\dots,n$ a iid sequence of random variables with law $P$, and $\bm{Z}_i^*={\rev (Z_i^1,\dots,Z_i^d)}$ are the public views of  
	$\bm{X}^*_i={\rev (X_i^{*,1},\dots,X_i^{*,d})}$, $i=1,\dots,n$ a iid sequence of random variables with law $P^*$.
	
	Let $0<\delta<1$ be a parameter whose value will be calibrated later.
	We denote by $P$ the probability on $\mathbb{R}^{\revnew d}$ which makes $\bm{X}={\rev (X^1,\dots,X^d)}$ a discrete random variable taking values in 
		{\rev $\prod_{j=1}^d \{-\delta^{1/k_j},0,-\delta^{1/k_j}\}$ with the joint distribution given
		by
		\begin{equation*}
			P(X^1=a_1 \delta^{-1/k_1},\dots,X^d=a_d \delta^{-1/k_1})=:p_{a_1,\dots,a_d} 
		\end{equation*}
	for all $(a_1,\dots,a_d) \in \{-1,0,1\}^d$  and 
	\begin{equation*}
		p_{a_1,\dots,a_d}=
\begin{cases}
	1-\delta & \text{if $a_1=\dots=a_d=0$,}
	\\
	\delta (\frac{1}{2})^{d} & \text{if $a_j \neq 0, \forall j$,}
	\\
	0 & \text{if $\exists j_1,j_2$ with $a_{j_1}=0$, $a_{j_2}\neq0$.}
\end{cases}	
\end{equation*}
It means that, under $P$, we flip a coin to decide if all the $X^j$ are zero with probability $1-\delta$, and otherwise the $X^j$ are taking the extremal values $\pm \delta^{-1/k_j}$ independently and with equal probability. We can check that $P(X^j=\delta^{-1/k_j})= P(X^j=-\delta^{-1/k_j})=\frac{\delta}{2}$ and $P(X^j=0)=1-\delta$. Thus $\E[|X^j|^{k_j}]=1$ for all $j \in \{1,\dots,d\}$ implying that $P \in \mathcal{P}_{k_1,\dots,k_d}$. Moreover, 
\begin{align*}
	\gamma(P)&=\E_P\left[X^1 \times \dots \times X^d\right]=(1-\delta)\times 0  + \delta 
	\sum_{(a_1,\dots,a_d)\in \{-1,1\}^d} \left(\frac{1}{2}\right)^{d}\prod_{j=1}^d (a_j \delta^{{\revnew -1/k_j}})
	\\
	&= \delta^{1-\sum_{j=1}^d \frac{1}{k_j}} \left(\frac{1}{2}\right)^{d} \sum_{l=0}^d \binom{d}{l} (-1)^l = \delta^{1-\sum_{j=1}^d \frac{1}{k_j}} \left(\frac{1}{2}\right)^{d} (1-1)^d
=0,\end{align*}
where in the second line $l$ is the cardinal of the $a_j$'s equal to $-1$.

We now define $P^*$. To this end, we set for $(a_1,\dots,a_d)\in \{-1,0,1\}^d$:
\begin{equation*}
	h_{a_1,\dots,a_d}:=
		\frac{\delta}{2} \frac{1}{2^{d}} \prod_{j=1}^d a_j, \quad
		p^*_{a_1,\dots,a_d}=p_{a_1,\dots,a_d}+h_{a_1,\dots,a_d}.
\end{equation*}
Remark that these coefficients $p^*$ are those of a probability and {\revd it} allows us to define $P^*$ by $P^*(X^1=a_1 \delta^{-1/k_1},\dots,X^d=a_d \delta^{-1/k_1})=p^*_{a_1,\dots,a_d}$.
Also, we have for all $j\in\{1,\dots,d\}$, $(a_1,\dots,a_{j-1},a_{j+1},\dots, a_d)\in \{-1,0,1\}^{d-1}$ that, $\displaystyle \sum_{b \in \{-1,0,1\}} h(a_1,\dots,a_{j-1},b,a_{j+1},\dots, a_d)=0$. This implies that  
$$ \sum_{b \in \{-1,0,1\}} p(a_1,\dots,a_{j-1},b,a_{j+1},\dots, a_d)=  \sum_{b \in \{-1,0,1\}} p^*(a_1,\dots,a_{j-1},b,a_{j+1},\dots, a_d),$$
and in turn the law of $(X^{1},\dots,X^{j-1},X^{j+1},\dots,X^d)$ is the same under $P$ and $P^*$.
This property is crucial to allow the application of Corollaries \ref{cor: main} or \ref{cor: main samples}.
It also yields that for any $l \in \{1,\dots,d\}$, we have $\E_{P^*}[|X^l|^{k_l}]=\E_{P}[|X^l|^{k_l}]=1$ and thus, $P^* \in \mathcal{P}_{k_1,\dots,k_d}$. 
Moreover,
\begin{align*}
	\gamma(P^*)&=\E_{P^*}\left[X^1 \times \dots \times X^d\right]=\E_{P}\left[X^1 \times \dots \times X^d\right]  + \frac{\delta}{2} 
	\sum_{(a_1,\dots,a_d)\in \{-1,1\}^d} \left(\frac{1}{2}\right)^{d}\prod_{j=1}^d (a_j \delta^{{\revnew -1/k_j}})
	\\
	&= 0 + \frac{1}{2} \delta^{1-\sum_{j=1}^d \frac{1}{k_j}} \left(\frac{1}{2}\right)^{d} 
		\sum_{(a_1,\dots,a_d)\in \{-1,1\}^d}
	(\prod_{j=1}^d a_j)^2
\\	&= \frac{1}{2} \delta^{1-\sum_{j=1}^d \frac{1}{k_j}} \left(\frac{1}{2}\right)^{d} 
\sum_{(a_1,\dots,a_d)\in \{-1,1\}^d} 1 = \frac{1}{2} \delta^{1-\sum_{j=1}^d \frac{1}{k_j}}.
\end{align*}
As a result, we have $\abs{\gamma(P)-\gamma(P^*)}= \frac{1}{2} \delta^{1-\sum_{j=1}^d \frac{1}{k_j}}$.
}
	We apply Corollary {\modch \ref{cor: main samples}} to the sequences of raw samples $(\bm{X}_i)_{i=1,\dots,n}$, $(\bm{X}_i^*)_{i=1,\dots,n}$ whose distributions are $P^{\otimes n}$ and $(P^*)^{\otimes n}$. 
		Indeed, this is permitted  as the {\rev $(d-1)-$dimensional marginal laws} of the {\rev  $d-$dimensional vectors $\bm{X}_i$ and $\bm{X}_i^*$ coincide.}
	 We deduce
		\begin{equation*}
			d_{KL} \Big( Law((\bm{Z}_{i})_{i = 1, \dots , n}), Law((\bm{Z}^*_{i})_{i = 1, \dots , n})  \Big)\le n \times 
			{\rev \prod_{j=1}^d  |e^{\alpha_j} - 1|^{2}  \times  d_{TV}(P, P^*)^2}
		\end{equation*}
	Furthermore, {\rev 
		\begin{align*}
			d_{TV}(P, P^*)&=
		\sum_{(a_1,\dots,a_d)\in\{-1,0,1\}^d} 
	|p_{a_1,\dots,a_d}-p^*_{a_1,\dots,a_d}|
	=	\sum_{(a_1,\dots,a_d)\in\{-1,0,1\}^d} 
|h(a_1,\dots,a_d)|
\\ &\le \sum_{(a_1,\dots,a_d)\in\{-1,0,1\}^d} \frac{\delta}{2} \frac{1}{2^d} 
\abs{\prod_{j=1}^d a_j}
 \le   \frac{\delta}{2} \frac{1}{2^d} \sum_{(a_1,\dots,a_d)\in\{-1,1\}^d} 1 \le \frac{\delta}{2},
\end{align*}
and so we get
	\begin{equation*}
		d_{KL} \Big( Law((\bm{Z}_{i})_{i = 1, \dots , n}), Law((\bm{Z}^*_{i})_{i = 1, \dots , n})  \Big)\le n {\rev \times 
		\prod_{j=1}^d |e^{\alpha_j} - 1|^{2} \times \frac{\delta^2}{4}. }
	\end{equation*}
}
	We now set {\rev $\delta=  \left( 2 \prod_{j=1}^d |e^{\alpha_j} - 1|^{2} n \right)^{-1/2}$} which is strictly smaller than $1$ by the assumption {\rev 
			$n \prod_{j=1}^d |e^{\alpha_j} - 1|^{2} \ge 1$.
		} Then, we get, $d_{KL} \Big( Law((\bm{Z}_{i})_{i = 1, \dots , n}), Law((\bm{Z}^*_{i})_{i = 1, \dots , n})  \Big) {\rev \le 1/8  <2}$.	
	Moreover,  {\rev 
	\begin{multline*}
		\abs{\gamma(P)-\gamma(P^*) }=	\frac{1}{2} \delta^{1-\sum_{j=1}^d\frac{1}{k_j}} =  
				\frac{1}{2} \delta^{1-d/\overline{k}}=
		\frac{1}{2^{3/2}} \Big(n \prod_{j=1}^d \abs{e^{\alpha_j} - 1}^{2}\Big)^{-(\frac{1}{2}-\frac{d}{2\overline{k}})}
		\\= \frac{1}{2^{3/2}} \Big( n \prod_{j=1}^d \abs{e^{\alpha_j} - 1}^{2}\Big)^{-\frac{\overline{k}-d}{2\overline{k}}}.
	\end{multline*} 
}
	We have obtained the Points 1--3 stated at the beginning of the proof and the lower bound on 	{\rev $\mathcal{M}_n(\gamma(\mathcal{P}_{k_1,\dots,k_d}),\bm{\alpha}) $} follows.	
\end{proof}
{\rev
	\begin{remark}
		For simplicity, we restrict the discussion on non-interactive mechanisms. However, the Corollary \ref{cor: main samples} is true for interactive and non-interactive mechanism. Thus, the lower bound holds true for interactive mechanism as well. As the statistical procedure described in \ref{Sss:joint moment estimator} is non-interactive, it means that in the statistical problem of moment estimation, non-interactive mechanism are rate efficient.
However, optimality of non interactive mechanisms is problem specific as shown in \cite{Ber20}.	
	\end{remark}
}

\subsubsection{Adpative estimation of the {\rev joint moment}}\label{Sss: adaptive}
As discussed in Section \ref{Sss:joint moment estimator}, 
the privacy procedure proposed in this study requires 
selecting the optimal truncation levels $T^{(j)}$'s, which depend on the number of finite moments {\rev $k_j$'s} for the variables. 
However, in practice, it is unrealistic to assume that the number of finite moments is known in all situations. To address this issue, we propose an adaptive method to estimate the covariance that does not necessitate prior knowledge of {\rev the $k_j$'s} and conforms to the privacy constraint.


The main idea is to send a collection of public data with different truncation levels via the privatization channel, and then let the statistician decide on the optimal truncation level using a penalization method.


 We introduce the following set of truncations :
\begin{align}\nonumber
&	\mathcal{T}:={\rev \prod_{j=1}^d \mathcal{T}^{(j)}}, \text{where for {\rev $j\in \{1,\dots,d\}$}}
	\\ \label{eq: def ens T possible}
&	\mathcal{T}^{(j)}:=\left\{ T^{(j)} \in (0,\infty) \mid T^{(j)}=\frac{n}{2^{r}},\quad 
	\text{ for some $r\in\{1,\dots,\lfloor\log_2(n)\rfloor\}$}\right\}.
\end{align}
Let {\rev $\beta_n^j >0$, for $j=1,\dots,d$  be $d$ parameters that we will specify later. For all $i\in\{1,\dots,n\}$ we are given $\sum_{j=1}^d\mathop{card}(\mathcal{T}^{(j)})=d \lfloor \log_2(n) \rfloor$} independent variables, {$\mathcal{E}^{j,T^{(j)}}_i$ where $T^{(j)}$ ranges in $\mathcal{T}^{(j)}$ and $j$ in $\{1,\dots,d\}$.}
We assume that each of the variables $\mathcal{E}^{j,T^{(j)}}_i$ follows a Laplace distribution with parameter $2\frac{T^{(j)}}{\beta_n^j}$, where $T^{(j)}\in \mathcal{T}^{(j)}$ and $j$ takes values 
{\rev in $\{1,\dots,d\}$}. We further assume that these variables are independent for different values of $i$ ranging from 1 to $n$.

 We define the privatized data {\rev $\bm{Z}_i=(Z_i^1,\dots,Z_i^d) \in \mathbb{R}^{\mathcal{T}^{(1)}}\times \dots \times
\mathbb{R}^{\mathcal{T}^{(d)}}$, $i\in\{1,\dots,n\}$, by
\begin{equation} \label{eq: def Zi adpat}
	\begin{aligned}
		Z^1_i&=(Z^{1,T}_i)_{T\in \mathcal{T}^{(1)}}, &  	Z^{1,T}_i&=\tronc{X_i^1}{T}+\mathcal{E}^{(1),T}_i, &  &\quad \text{for $T \in \mathcal{T}^{(1)}$, $i\in\{1,\dots,n\}$},
	\\
	&~~\vdots   & &~~ \vdots   & &\quad\quad\vdots
	\\
Z^d_i&=(Z^{d,T}_i)_{T\in \mathcal{T}^{(d)}}, &	Z^{d,T}_i&=\tronc{X_i^d}{T}+\mathcal{E}^{(d),T}_i,& &\quad \text{for $T \in \mathcal{T}^{(d)}$, $i\in\{1,\dots,n\}$}.
\end{aligned}\end{equation}
Let us stress that on contrary to the privacy channel defined by \eqref{eq: def Z covariance}, where each data $X_i^j$ is publicly released using a one dimensional noisy view, here each data $X_i^j$ is disclosed through a repetition of $\mathop{card}(\mathcal{T}^{j})$ noisy views, where $\mathop{card}(\mathcal{T}^{j})$ will grow to infinity. This tends to
reduce the privacy of the channel as all these public views contain information on the same private value.
To guarantee that this procedure is compatible with the $\bm{\alpha}$-CLDP constraint, we need that the noise injected in the channel growths with the dimension of the public view.  
}
\begin{lemma}\label{l: privacy adaptive covariance}
	Assume that $\beta_n^j=\frac{{\modar \alpha_j}}{\mathop{card}(\mathcal{T}^{(j)})}=\frac{\alpha_j}{\lfloor \log_2(n)\rfloor}$.
Then, the privacy procedure satisfies the {\modar  ${\bm{\alpha}}$-CLDP constraint as in \eqref{eq: def density local privacy}.} 
\end{lemma}
\begin{proof}
	 For {\rev }$j\in \{1,\dots,d\}$, let us denote by $q^j( (z^{j,T})_{T\in \mathcal{T}^{(j)}} \mid X^j=x)$
the density of the law of {\modar $Z^j_i=(Z^{j,T}_i)_{T \in \mathcal{T}^{(j)}}$} conditional to $X^j_i=x\in\mathbb{R}$.
Then, using the independence property of the variables $(Z^{j,T}_i)_{T \in \mathcal{T}^{(j)}}$, we have
\begin{align*}
	\frac{q^j( (z^{j,T})_{T\in \mathcal{T}^{(j)}} \mid X^j_i=x)}{q^j( (z^{j,T})_{T\in \mathcal{T}^{(j)}} \mid X^j_i=x')}&=
	\frac{\prod_{T\in\mathcal{T}^{(j)}} \exp\left( \abs{z^{j,T}-\tronc{x}{T}} \frac{\beta_n^j}{2 T} \right)}{\prod_{T\in\mathcal{T}^{(j)}} \exp\left( \abs{z^{j,T}-\tronc{x'}{T}} \frac{\beta_n^j}{2 T} \right)}
	\\
	&=\prod_{T\in\mathcal{T}^{(j)}} \exp\left(\frac{\beta_n^j}{2 T}\left\{ \abs{z^{j,T}-\tronc{x}{T}}-\abs{z^{j,T}-\tronc{x'}{T}} 
		\right\} \right)
	\\
	&\le \prod_{T\in\mathcal{T}^{(j)}} \exp\left( \frac{\beta_n^j}{2 T}\abs*{\tronc{x}{T}- \tronc{x'}{T}} \right)
\end{align*}
where we used the inverse triangular inequality in the last line. As $\abs*{\tronc{x}{T}- \tronc{x'}{T}}  \le 2T$ we deduce,
\begin{equation*}
	\frac{q^j( (z^{j,T})_{T\in \mathcal{T}^{(j)}} \mid X^j_i=x)}{q^j( (z^{j,T})_{T\in \mathcal{T}^{(j)}} \mid X^j_i=x')}\le 
	\prod_{T\in\mathcal{T}^{(j)}} \exp\left( \beta_n^j \right)=
	\exp(\mathop{card}(\mathcal{T}^{(j)}) \beta_n^j) \le \exp(\alpha_j),
\end{equation*}
by the choice of $\beta_n^j$.
\end{proof}

{\revnew Note that, although the proof of Lemma \ref{l: privacy adaptive covariance} follows a standard argument, we have chosen to include it here, as it provides an explanation for the necessity of the additional \( \log_2(n) \) term in the scale parameter of the Laplace random variable.}

We construct our adaptive estimator, following Goldenshluger-Lepski method. 
For {\rev $\bm{T}=(T^{(1)},\dots,T^{(d)}) \in \mathcal{T}$,} we set
\begin{equation} \label{eq : gamma hat T adapt}
{\rev \hat{\gamma}^{(\bm{T})}_n=\frac{1}{n}\sum_{i=1}^n  \prod_{j=1}^d Z^{j,T^{(j)}}_i,}
\end{equation} 
and for ${\rev \bm{T}=(T^{(1)},\dots,T^{(d)})} \in \mathcal{T}$, ${\rev \bm{T}'=(T^{\prime(1)},\dots,T^{\prime(d)}) \in \mathcal{T}}$
{\rev
\begin{equation} \label{eq : gamma hat T Tprime adapt}
	\hat{\gamma}^{(\bm{T},\bm{T}^\prime)}_n=\frac{1}{n}
		\sum_{i=1}^n \prod_{j=1}^d Z^{j,T^{(j)}\wedge T^{\prime(j)}}_i.
\end{equation}}
Let us remark that the following commutativity relation hold true: $\hat{\gamma}^{(\bm{T},\bm{T}^\prime)}_n=\hat{\gamma}^{(\bm{T}^\prime,\bm{T})}_n$.
 Based on the upper bound \eqref{eq: upper bound var est gamma} given in the Appendix for the variance of the estimator, we introduce the penalization term for $\bm{T} \in \mathcal{T}$,
 {\rev
\begin{equation} \label{eq: def penalisation est cov}
	\mathbb{V}_{\bm{T}}=\mathbb{V}_{(T^{(1)},\dots,T^{(d)})}= \kappa_n
	\frac{\prod_{j=1}^d \abs{T^{(j)}}^2}{n\prod_{j=1}^d \abs{\beta_n^j}^2}
\end{equation}
}
for $\kappa_n\ge 1$ some sequence tending slowly to $\infty$, which will be specified in Theorem \ref{th: est gamma adaptive}.

\noindent For ${\rev \bm{T}}\in \mathcal{T}$, we define
\begin{equation} \label{eq: def est biais adaptatif est cov}
	\mathbb{B}_{{\rev\bm{T}}}=\sup_{{\rev\bm{T}^\prime }\in \mathcal{T}}
	\left\{   \left( \abs*{\hat{\gamma}^{{\rev (\bm{T},\bm{T}^\prime)}}_n - \hat{\gamma}^{{\rev (\bm{T}^\prime)}}_n }^2 -
	\mathbb{V}_{{\rev \bm{T}^\prime}}\right)_+ \right\}, 
\end{equation}
and set
\begin{equation} \label{eq: def That est cov}
	\widehat{{\rev\bm{T}}}=\argmin{{\rev\bm{T}}\in \mathcal{T}} 
	\left\{ \mathbb{B}_{\rev\bm{T}}+\mathbb{V}_{\rev\bm{T}} \right\}.
\end{equation}
Our adaptive estimator is $\hat{\gamma}_n^{(\widehat{{\rev\bm{T}}})}$.
\begin{theorem}\label{th: est gamma adaptive}
	Assume that {\rev $k_1^{-1}+\dots+k_d^{-1}<1$,} $\beta_n^j=\frac{\alpha_j}{\lfloor \log_2(n) \rfloor}$, for {\rev $j=1,\dots,d$ } and  $\kappa_n=c_0 \log (n)$ for some $c_0>0$. {\modar If $c_0$ is large enough, there exist} $c>0$, $\overline{c}_0>0$, such that
	\begin{equation*}
		\E\left[(\hat{\gamma}^{\widehat{T}}_n - \gamma )^2\right]
		\le c \left( \frac{n{\rev \prod_{j=1}^d\alpha_j^2}}{(\log(n))^{{\rev 2d+1}}}\right)^{-\frac{\overline{k}-{\rev d}}{\overline{k}}}
		+ \frac{c}{{\rev \prod_{j=1}^d \alpha_j^2 } n^{\overline{c}_0}},
	\end{equation*}
	for all $n \ge 1$, $\alpha_j\le 1$, {\rev $(n \prod_{j=1}^d\alpha^2_j)/(\log(n))^{2d+1} \ge 1 $.} Moreover, the constant $\overline{c}_0$ can be chosen arbitrarily large by choosing $c_0$ large enough.
\end{theorem}
\begin{remark}
	{\modar Comparing with Theorem \ref{th: upper bound joint moment covariance}, we observe that the rate of the adaptive version of the estimator worsens by a factor of {\rev $\log(n)^{2d+1}$.} The loss of a $\log(n)$ factor is a well-known characteristic of adaptive methods and is sometimes unavoidable, as mentioned in \cite{BrownLow96}. The additional loss of a $\log(n)^{{\rev 2d}}$ term arises from the disclosure of $\mathop{card}{\mathcal{T}^{(j)}}\asymp\log_2(n)$ observations for each raw data, which increases the variance of the privatization mechanism while maintaining a constant level of privacy, as demonstrated in Lemma \ref{l: privacy adaptive covariance}. This is one reason why, in defining the sets $\mathcal{T}^{(j)}$, we have attempted to minimize their cardinality.}
\end{remark}
\noindent
The proof of the adaptive procedure gathered in Theorem \ref{th: est gamma adaptive} can be found in the Appendix.

\subsection{Locally private multivariate density estimation}{\label{s: density}}
In this section we consider the non-parametric estimation of the density of the vector $\bm{X}=(X^1, \dots , X^d)$, under $\alpha$-CLDP. We will see that, similarly to the case where the components become public jointly, this implies a deterioration on the convergence rate depending on $\bm{\alpha}$ (see for example Section 5.2.2 of \cite{Martin}). 

 Consider $\bm{X}_1, \dots , \bm{X}_n$, $n$ iid copies of $\bm{X}$. We will assume that the density $\pi$ of $\bm{X}$ belongs to an H\"older class $\mathcal{H}(\beta, \mathcal{L})$ (see for example Definition 1.2 in \cite{Tsy}). We aim at estimating such density under componentwise local differential privacy. We recall we reduce to consider the non-interactive privacy mechanism for the statistical applications, in order to lighten the notation.


\subsubsection{Local differential private estimator}{\label{s: density upper}}
In absence of privacy constraints, a well-studied estimator for density estimation consists in the kernel density estimator (see for example Section 1.2 of \cite{Tsy} and Part III of \cite{BosBla}). It achieves the convergence rate $n^{-\frac{2\beta}{2 \beta + d}}$, which has been shown to be optimal in a minimax sense (see Theorem 1.1 in \cite{Tsy} for the monodimensional case). 

We therefore introduce some kernel function $K: \mathbb{R} \rightarrow \mathbb{R}$ satisfying, for all $l \in \left \{ 1, \dots , \beta \right \}$,
\begin{equation}{\label{eq: property kernel}}
\int_\mathbb{R} K(x) dx = 1, \quad \left \| K \right \|_\infty < \kappa, \quad \mbox{supp}(K) \subset [-1, 1], \quad \int_\mathbb{R} K(x) x^l dx = 0.
\end{equation}
Then, as for the estimation of the covariance, we add centered Laplace distributed noise on bounded random variables to obtain $\alpha$-CLDP. 

\begin{lemma}{\label{l: privacy procedure}}
For any $i \in \{1, \dots , n \}$, $j \in \{1, \dots , d \}$ and any $\bm{x_0} \in \R^d$, the random variables 
\begin{equation}{\label{eq: def Zij}}
Z_i^j  := \frac{1}{h} K(\frac{X_i^j - x_0^j}{h}) + \mathcal{E}_i^j,
\end{equation}
with $\mathcal{E}_i^j$ iid $\sim \mathcal{L}(\frac{2 \kappa}{\alpha_j h})$, are $\alpha_j$-differentially private views of the original $X_i^j$.
\end{lemma}
The index $h$ introduced in \eqref{eq: def Zij} is small. In particular, we assume $h < 1$. The proof of Lemma \ref{l: privacy procedure} consists in checking property \eqref{eq: def density local privacy}, similarly as in Lemma \ref{l: ldp covariance}.
{\rev The proof of the lemma is rather close to the proof of Lemma \ref{L: preuve Laplace mech generique} and the detailed proof is left to the Appendix.}
\noindent We introduce the kernel density estimator $\hat{\pi}^Z_h$ based on the discrete observations $Z_i^j$, for $i\in \{1, \dots , n \}$ and $j \in \{1,\dots , d \}$. We define, for any $\bm{x_0} \in \R^d$, 
\begin{equation}{\label{eq: kde}}
 \hat{\pi}^Z_h(\bm{x_0}):= \frac{1}{n} \sum_{i = 1}^n \prod_{j = 1}^d Z_i^j= \frac{1}{n} \sum_{i = 1}^n \prod_{j = 1}^d \Big( \frac{1}{h} K(\frac{X_i^j - x_0^j}{h}) + \mathcal{E}_i^j \Big).
\end{equation}
We now prove an upper bound the $L^2$ pointwise risk, showing that $\hat{\pi}^Z_h$ achieves the convergence rate $(\frac{1}{n \prod_{j = 1}^d \alpha_j^2})^{\frac{\beta}{\beta + d}}$. 
\begin{theorem}{\label{th: upper density}}
Assume $\bm{X}_1, \dots , \bm{X}_n$ are iid copies of an $\R^d$ vector $\bm{X}$ whose density $\pi$ belongs to the H\"older class $\mathcal{H}(\beta, \mathcal{L})$ {\modar and $\bm{x_0}\in \mathbb{R}^d$. Let $0<\alpha_j \le 1 $
for any $j \in \{1, \dots , d \}$. }
If $n \prod_{j = 1}^d \alpha_j^2 \rightarrow \infty$, then 
there exist $c > 0$ and $n_0 > 0$ such {\modar that for any $n \ge n_0$,} 
$$\E[|\hat{\pi}^Z_h(\bm{x_0}) - \pi(\bm{x_0})|^2] \le c  (\frac{1}{n \prod_{j = 1}^d \alpha_j^2})^{\frac{\beta}{\beta + d}}.$$  
\end{theorem}

\noindent This shows that the effects of local differential privacy constraints are severe for non-parametric density estimation, as they lead to a different convergence rate.

\noindent In the case where $\alpha_1 = \dots = \alpha_d$ it is possible to obtain the following result, which provides the threshold which dictates the behaviour of the estimator with respect to the privacy mechanism. Indeed, for $\alpha  \ge n^{\frac{1}{2(2 \beta + d)}}$, we recover the same convergence rate as in absence of privacy.
\begin{theorem}{\label{th: upper density same alpha}}
Assume $\bm{X}_1, \dots , \bm{X}_n$ are iid copies of an $\R^d$ vector $\bm{X}$ whose density $\pi$ belongs to the H\"older class $\mathcal{H}(\beta, \mathcal{L})$, {\modar and $\bm{x_0}\in \mathbb{R}^d$}. 
Then, the following inequalities hold true
\begin{enumerate}
    \item If $\alpha  \ge n^{\frac{1}{2(2 \beta + d)}}$, then
there exist $c > 0$ and $n_0 > 0$ such {\modar that, for any $n \ge n_0$,} 
$$\E[|\hat{\pi}^Z_h(\bm{x_0}) - \pi(\bm{x_0})|^2] \le c  (\frac{1}{n})^{\frac{2 \beta}{2 \beta + d}}.$$
\item If otherwise $\alpha < n^{\frac{1}{2(2 \beta + d)}}$ and $ n \alpha^{2d } \rightarrow \infty$, then there exist $c > 0$ and $n_0 > 0$ such {\modar that, for any $n \ge n_0$,} 
$$\E[|\hat{\pi}^Z_h(\bm{x_0}) - \pi(\bm{x_0})|^2] \le c  (\frac{1}{n \alpha^{2d}})^{\frac{\beta}{\beta + d}}.$$  
\end{enumerate}
\end{theorem}
The proof of these two results can be found in Section \ref{s: Proof locally density} of the Appendix.

\begin{remark}
The above result indicates that a threshold for the behavior of a system with and without privacy is determined by $n^{\frac{1}{2(2 \beta + d)}}$. If $\alpha$ is greater than this value, it means that the level of privacy provided is not significant enough to degrade the convergence rate of our estimator compared to the case without privacy. 
However, if $\alpha$ is smaller than $n^{\frac{1}{2(2 \beta + d)}}$, then the level of privacy provided is sufficient to reduce the statistical utility, leading to a deterioration of the convergence rate as a function of $\alpha$. It is essential to note that it is impossible to achieve perfect privacy even in this context ($\alpha = 0$). The  condition that $n \alpha^{2d} \rightarrow \infty$ 
must indeed be satisfied, which is the price to pay for allowing statistical inference.
\end{remark}

\subsubsection{Lower bound for density estimation}{\label{s: density lower}}
We can now derive minimax lower bound, based on the key result gathered in Theorem \ref{th: main bound} and its consequences.

\begin{theorem}{\label{th: lower density}}
Let $\alpha_j \in (0, \infty)$ for $j \in \{1, \dots , d \}$ and let $\beta, \mathcal{L} > 0$. Then, there exists a constant $c > 0$ such that
$$\inf_{\bm{Q} \in \mathcal{Q}_{\bm{\alpha}}}\inf_{\tilde{\pi}} \sup_{\pi \in \mathcal{H}(\beta, \mathcal{L})} \E[|\tilde{\pi}(\bm{x_0}) - \pi(\bm{x_0})|^2]\ge c \Big( n \, \prod_{j = 1}^d(e^{\alpha_j} -1)^{2} \Big)^{- \frac{\beta}{\beta + d}},$$
for all $n \ge 1$, $n \prod_{j = 1}^d |e^{\alpha_j} - 1| \rightarrow \infty$. The infimum is taken over all the estimators $\tilde{\pi}$ based on the privatized vectors $\bm{Z}_1, \dots , \bm{Z}_n $ and all the non-interactive Markov kernels in $\mathcal{Q}_{\bm{\alpha}}$ guaranteeing ${\bm{\alpha}}$-CLDP.
\end{theorem}

\begin{remark}
When the privacy parameters $\alpha_j$ are small, it is clear that the upper and lower bounds in Theorems \ref{th: upper density} and \ref{th: lower density} match each other. This suggests that the proposed privacy mechanism is optimal (in the minimax sense) as long as a reasonable amount of privacy is ensured for all the components (i.e., $\alpha_j < 1$ for any $j \in \{1, \dots , d \}$).
\end{remark}

\begin{remark}{\label{r: rate density but}}
One can compare the deterioration of the convergence rate gathered in our Theorem \ref{th: lower density} (and the corresponding upper bound in Theorem \ref{th: upper density same alpha}) with the results in \cite{But}, which focuses on estimating the density under privacy constraints using $n$ independent and identically distributed random variables $X_1, \dots, X_n$. Their analysis on Besov spaces $\mathcal{B}_{pq}^s$ under mean integrated $L^r$-risk revealed an elbow effect that led to the optimal (in the minimax sense) convergence rate of $(n(e^{\alpha}-1)^2)^{- \frac{rs}{2s + 2}}$ whenever $p > \frac{r}{s + 1}$ (see Equation (1.2) in \cite{But}). Using our notation, this rate corresponds to $(n(e^{\alpha}-1)^2)^{- \frac{2\beta}{2 \beta + 2}} = (n(e^{\alpha}-1)^2)^{- \frac{\beta}{ \beta + 1}} $ and the condition on $p$ reduces to $2 > \frac{2}{\beta + 1}$, which is always true. Therefore, it is evident that our results match those of \cite{But} when considering $d=1$, but they are in general different as in the case where $\alpha_1 = \dots = \alpha_d = \alpha$ the size $(e^{\alpha}-1)^2$ in \cite{But} is now replaced by $(e^{\alpha}-1)^{2d}$.
\end{remark}

\begin{proof}[Proof of Theorem \ref{th: lower density}]
We can assume without loss of generality that $\bm{x_0} = \bm{0}$, the general case can be deduced by translation. \\
The proof of the lower bound relies on the two hypothesis method, as in Section 2.3 of
 \cite{Tsy}. It consists in proposing $\pi$ and $\pi^*$, densities of $\bm{X}$ and $\bm{X}^*$ and with privatized views $\bm{Z}$ and $\bm{Z}^*$, such that the following three conditions hold true: 
\begin{enumerate}
    \item $\pi$ and $\pi^*$ belong to $\mathcal{H}(\beta, \mathcal{L})$,
    \item $|\pi(\bm{0}) - \pi^* (\bm{0})| \ge \frac{1}{M_n}$,
    \item $\exists c > 0$ such that $d_{KL} \Big( Law((\bm{Z}_{i})_{i = 1, \dots , n}), Law((\bm{Z}^*_{i})_{i = 1, \dots , n})  \Big) < \epsilon_0 < 2$,
\end{enumerate}
where $\frac{1}{M_n}$ is a calibration parameter which will be chosen later, in order to obtain the wanted convergence rate. 
If the constraints above are satisfied, then in the same way as in the proof of Theorem \ref{th: lower joint} it follows there exists $c>0$ such that
\begin{equation}{\label{eq: goal lower}}
\inf_{Q \in \mathcal{Q}_{\bm{\alpha}}} \inf_{\tilde{\pi}} \sup_{\pi \in \mathcal{H}(\beta, \mathcal{L})} \E[|\tilde{\pi}(\bm{0}) - \pi(\bm{0})|^2]\ge c (\frac{1}{M_n})^2.
\end{equation}
Let us define, for any $\bm{x} \in \R^d$, $\pi(\bm{x}) := c_\pi e^{- \eta |x|^2}. $
The constant $\eta$ can be chosen as small as we want, while $c_\pi$ is a normalization constant added in order to get $\int_{\R^d} \pi(\bm{x}) d\bm{x} = 1$. Regarding $\pi^*$, we give it as $\pi$ to which we add a bump. Let $\tilde{\psi}: \R \rightarrow \R$ be a $C^\infty$ function with support on $[-1, 1]$ and such that $\tilde{\psi}(0) = 1$, $\int_{-1}^1 \tilde{\psi}(z) dz = 0.$
Then, we set 
$\pi^*(\bm{x}) := \pi(\bm{x}) + \frac{1}{M_n} \prod_{l = 1}^d \tilde{\psi}(\frac{x^l}{h_n}) =: \pi(\bm{x}) + \frac{1}{M_n} \psi_{h_n}(\bm{x}).  $
As $\frac{1}{M_n}$, $h_n$ will be calibrated later. The two calibration constants satisfy $M_n \rightarrow \infty$ for $n \rightarrow \infty$ and $h_n \rightarrow 0$ for $n \rightarrow \infty$. \\
\\
It is easy to check Condition 1 holds true. Indeed, we can choose $\eta$ small enough to obtain $\pi \in \mathcal{H}(\beta, \mathcal{L})$ and a similar reasoning ensures also that $\pi^* \in \mathcal{H}(\beta, \mathcal{L})$. However, $\left \| \psi_{h_n}^{(k)}\right \|_\infty \le \frac{c}{h_n^k} $ and so in the k-derivative of $\pi^*$ an extra $\frac{1}{M_n h_n^k}$ appears. It implies we have to ask the existence of some constant $c> 0$ such that $\frac{1}{M_n h_n^k} < c$ for any $k \in \{0, \dots, \lfloor \beta \rfloor \}$, in order to obtain $\pi^* \in \mathcal{H}(\beta, \mathcal{L})$. Thus, for some $c > 0$ it naturally arises the condition 
$\frac{1}{M_n h_n^\beta} < c.$ \\
Concerning Condition 2, by construction and from the properties of the function $\tilde{\psi}$ it is 
$$|\pi(\bm{0}) - \pi^*(\bm{0})| = |\frac{1}{M_n} \psi_{h_n}(\bm{0})| = \frac{1}{M_n} \prod_{l = 1}^d |\tilde{\psi}(0)| = \frac{1}{M_n}. $$
We are left to prove Condition 3. We observe that, for any 
{\revnew $S \subset \{1, ... , d \}$,} it is {\revnew $L_{\bm{X}^{(S)}} = L_{\bm{X}^{*,(S)}}$}.
 Indeed, {\revnew for any $k \in \{1, ... , d-1 \}$ and $S = \{i_1, ... , i_k \}$, } the law of {\revnew $\bm{X}^{*,(S)}$} is given by
\begin{align*}
& \int_{\R^{d -k}} \Big( \pi(x^1, \dots , x^d) + \frac{1}{M_n} \psi_{h_n}(x^1, \dots , x^d)\Big) \prod_{j: \, j\notin \{ i_1, \dots , i_k \} } dx^j \\
& = \int_{\R^{d -k}}  \pi(x^1, \dots , x^d) \prod_{j: \, j\notin \{ i_1, \dots , i_k \} } dx^j +  \frac{1}{M_n} \int_{\R^{d -k}} \prod_{l = 1}^d \tilde{\psi}(\frac{x^l}{h_n}) \prod_{j: \, j\notin \{ i_1, \dots , i_k \} } dx^j \\
& = \int_{\R^{d -k}}  \pi(x^1, \dots , x^d) \prod_{j: \, j\notin \{ i_1, \dots , i_k \} } dx^j  = {\revnew L_{\bm{X}^{(S)}}},
\end{align*}
where we have used that the integrals of $\tilde{\psi}$ are $0$ by construction. Thus, we can use Corollary \ref{cor: main samples}. It yields 
\begin{equation}{\label{eq: bound in cor}}
d_{KL} \Big( Law((\bm{Z}_{i})_{i = 1, \dots , n}), Law((\bm{Z}^*_{i})_{i = 1, \dots , n})  \Big) \le  {\modar n \times } \prod_{j = 1}^d |e^{\alpha_j} - 1|^{2} \Big( d_{TV} (Law(\bm{X}), Law(\bm{X}^*))  \Big)^2.
\end{equation}
To conclude, we observe it is 
\begin{align}{\label{eq: end bound}}
\Big( d_{TV} (Law(\bm{X}), Law(\bm{X}^*))  \Big)^2 & \le \Big( \int_{\R^d} \frac{1}{M_n} \psi_{h_n} (x^1, \dots , x^d) dx^1, \dots , dx^d  \Big)^2  \le c \frac{h_n^{2d}}{M_n^2}. 
\end{align}
From \eqref{eq: bound in cor} and \eqref{eq: end bound} we get there exists some constant $c_k > 0$ such that
$$d_{KL} \Big( Law((\bm{Z}_{i})_{i = 1, \dots , n}), Law((\bm{Z}^*_{i})_{i = 1, \dots , n})  \Big) \le c_k n \prod_{j = 1}^d |e^{\alpha_j} - 1|^{2} \frac{h_n^{2d}}{M_n^2}. $$
Hence, Condition 3 holds true up to say that $ c_k n \prod_{j = 1}^d |e^{\alpha_j} - 1|^{2} \frac{h_n^{2d}}{M_n^2}$ is bounded by some $\epsilon_0 < 2$. \\
The constraint given by Condition 1 leads us to the choice $h_n = (\frac{1}{M_n})^\frac{1}{\beta}$, which entails
$c_k n  \prod_{j = 1}^d |e^{\alpha_j} - 1|^{2} \Big(\frac{1}{M_n}\Big)^{\frac{2d}{\beta} + 2} < \epsilon_0$. It holds true if and only if 
$\Big(\frac{1}{M_n}\Big)^{\frac{2d + 2 \beta}{\beta}} < \frac{\epsilon_0}{c_k n \prod_{j = 1}^d |e^{\alpha_j} - 1|^{2}}.$
We therefore choose 
$\frac{1}{M_n} = \Big(\frac{\epsilon_0}{c_k n \prod_{j = 1}^d |e^{\alpha_j} - 1|^{2} }\Big)^{\frac{\beta}{2(d + \beta)}}.$ \\
Equation \eqref{eq: goal lower} concludes then the proof.

\end{proof}

\subsubsection{Adaptive density estimation}{\label{s: adaptive density}}
As seen in previous subsection, the proposed procedure leads us to the choice of a bandwidth which depends on the regularity $\beta$, that is in general unknown. This motivates a data-driven procedure for the choice of $h$. \\
We introduce the set of candidate bandwidths 
\begin{equation}{\label{eq: def Hn}}
H_n := \left \{ h \in (0, 1]: \, \mbox{ such that } \frac{1}{h} = \frac{n}{2^r} \mbox{ for some } r \in \{1, \dots, \lfloor \log_2(n) \rfloor \}\right \}.
\end{equation}
In a similar way as in Section \ref{Sss: adaptive} we introduce, for $j \in \{1, \dots , d \}$, some parameters $\beta_n^j > 0$ that will be better specified later. For any $i \in \{1, \dots , n \}$ and $j \in \{1, \dots , d \}$ the independent variables $\mathcal{E}_i^{j,h}$ are distributed as Laplace random variables with law $\mathcal{L}(\frac{2 \kappa}{h \beta_n^j})$ for all $h \in H_n$, where $\kappa$ is defined in \eqref{eq: property kernel}. We therefore define the privatized data $\bm{Z}_i = (Z_i^1, \dots , Z_i^d) \in \R^{H_n} \times \dots \times \R^{H_n}$ by $Z_i^j = (Z_i^{j,h})_{h \in H_n}$; 
\begin{equation}{\label{eq: def Zh}}
 Z_i^{j,h} := \frac{1}{h} K(\frac{X_i^j - x_0^j}{h}) + \mathcal{E}_i^{j,h}   
\end{equation}
for $h \in H_n$, $i \in \{1, \dots , n \}$ and $j \in \{1, \dots , d \}$. The set of potential estimators is defined accordingly: 
$$\mathcal{F}(H_n) := \left \{ \hat{\pi}^z_h (\bm{x_0}) = \frac{1}{n} \sum_{i = 1}^n \prod_{j = 1}^d Z_i^{j,h} \quad h \in H_n \right \}.$$ 
By following closely the proof of Lemma \ref{l: privacy procedure} it is easy to check the following lemma, whose proof is in the Appendix.
\begin{lemma}{\label{l: privacy adaptive density}}
Assume that $\beta_n^j= \frac{\alpha_j}{\lfloor \log_2 n \rfloor}$ for any $j \in \{1, \dots , d \}$. Then, the privatized variables described in \eqref{eq: def Zh} are $\bm{\alpha}$-local differential private views of the original $X_i^j$. 
\end{lemma}

\noindent As our adaptive procedure is based on Goldenshluger-Lepski method, we want to introduce an auxiliary estimator. For any $h, \eta \in H_n$ we set 
$$\hat{\pi}_{h,\eta}^z (\bm{x_0}) := \hat{\pi}^z_{h \land \eta} (\bm{x_0}) = \frac{1}{n} \sum_{i = 1}^n \prod_{j = 1}^d \frac{1}{h \land \eta} K(\frac{X_i^j - x_0^j}{h \land \eta}) + \mathcal{E}_i^{j,h \land \eta}. $$
Clearly it is $\hat{\pi}_{h,\eta}^z (\bm{x_0}) = \hat{\pi}_{\eta, h}^z (\bm{x_0})$. 
{\modar In the proof of Theorem \ref{th: upper density} given in the Appendix, we obtain the bound \eqref{eq: var} for the variance of the kernel estimator. This yield us to introduce the}
 penalization term $\mathbb{V}_h := a_n \frac{1}{n h^{2d}} \frac{1}{\prod_{j = 1}^d (\beta_n^j)^2}, $
for $a_n \ge 1$ some sequence tending slowly to $\infty$, which will be specified in Theorem \ref{th: adaptive density} below. We also define, for any $h \in H_n$, $\mathbb{B}_h := \sup_{\eta \in H_n} \left \{ ( | \hat{\pi}_{h,\eta}^z(\bm{x_0}) - \hat{\pi}_{\eta}^z(\bm{x_0}) |^2 - \mathbb{V}_\eta )_+ \right \}.$
Heuristically, $\mathbb{V}_h$ plays the role of the variance while $\mathbb{B}_h$ plays the role of the bias; the chosen bandwidth for the adaptive procedure is the one that minimizes their sum, i.e. $\hat{h} := argmin_{h \in H_n} \{ \mathbb{B}_h + \mathbb{V}_h \}.$
The associated adaptive estimator is $\hat{\pi}_{\hat{h}}^z(\bm{x_0})$, for which the following theorem holds true.
\begin{theorem}{\label{th: adaptive density}}
Assume that $\pi \in \mathcal{H}(\beta, \mathcal{L})$ for some $\beta$ and $\mathcal{L} \ge 1$. Moreover, $\beta_n^j = \frac{\alpha_j}{ \lfloor \log_2 n \rfloor}$ for any $j \in \{1, \dots , d \}$ and $a_n = {c_0} \log n$ for some ${c_0} > 0$. {\modar If $c_0$ is large enough, there} exist $c> 0$ and $\bar{c} > 0$ such that 
$$\E [(\hat{\pi}_{\hat{h}}^z(\bm{x_0}) - \pi (\bm{x_0}))^2] \le c (\frac{\log n^{1 + 2d}}{n \prod_{j = 1}^d \alpha_j^2})^{\frac{\beta}{\beta + d}} + \frac{c}{n^{\bar{c}}\prod_{j=1}^d \alpha_j^2}$$
for all $n \ge 1$, $\alpha_j \le 1$ and $\frac{n \prod_{j = 1}^d \alpha_j^2}{\log n^{1 + 2d}} \ge 1$. Moreover, the constant $\bar{c}$ can be chosen arbitrarily large, taking the constant $c_0$ large enough.
\end{theorem}


\begin{remark}
Just as in the case of estimating the covariance, the adaptive version of the density estimation algorithm has a slower convergence rate than the one presented in Theorem \ref{th: upper density}, by a logarithmic factor. 
\end{remark}

\noindent The proof of the data-driven selection of the bandwidth as proposed in Theorem \ref{th: adaptive density} can be found in the Appendix.

\appendix

\section{Appendix}
In this section we provide all the technical proofs. {\rev We will start by proving Propositions \ref{prop: recurrence} and \ref{P: ajout main en divergence}.}
After that, we will give the proofs for the locally private estimation of {\rev the joint moment} and of the density, as presented in Sections \ref{s: joint moment} and \ref{s: density}, respectively.

\subsection{Proof of Proposition \ref{prop: recurrence}}
\begin{proof}
Proposition \ref{prop: recurrence} is proven by recurrence. The proof is split in two steps. 
\vspace{0.3 cm}

\noindent \textbf{Step 1} \\
The aim of this step is to show that
$$l[z^1, \dots , z^d] \le l[\emptyset, z^2, \dots , z^d] + l[dx^1,z^2, \dots , z^d].$$
From the definition of $l$ it is 
$$ l[z^1, \dots , z^d] = |q(z^1, \dots , z^d) - \tilde{q}(z^1, \dots , z^d)|.$$
As observed in \eqref{eq: q_zi depuis q1 et qx1zi}, we have 
$$q(z^1, \dots , z^d) = \int_{\mathcal{X}^1}  q^1(z^1 | x^1) {\modar q( dx^1, z^2, \dots , z^d)}$$
and, in the same way, 
$$\tilde{q}(z^1, \dots, z^d) = \int_{\mathcal{X}^1} q^1(z^1 | x^1) {\modar \tilde{q}(dx^1, z^2, \dots , z^d) }.$$
It follows that
\begin{align}{\label{eq: start step 1}}
 l[z^1, \dots , z^d] = \abs[BBig]{\int_{\mathcal{X}^1} {\modar q^1}(z^1 | x^1)  [q(dx^1, z^2, \dots , z^d) - \tilde{q}(dx^1, z^2, \dots , z^d)] }.
\end{align}
We split the signed measure $q(dx^1, z^2, \dots , z^d) - \tilde{q}(dx^1, z^2, \dots , z^d)$ into its positive and negative part:
\begin{align*}
{\modar l[z^1, \dots , z^d]} & = \Big|{\modar \int_{\mathcal{X}^1} q^1(z^1 | x^1) [q(dx^1, z^2, \dots , z^d) - \tilde{q}(dx^1, z^2, \dots , z^d)]_+  }\\
& + {\modar \int_{\mathcal{X}^1} q^1(z^1 | x^1) [q(dx^1, z^2, \dots , z^d) - \tilde{q}(dx^1, z^2, \dots , z^d)]_{-} \Big|} \\
& \le \Big|\sup_{x^1 \in \mathcal{X}^1} {\modar q^1}(z^1 | x^1) {\modar \int_{\mathcal{X}^1} [q(dx^1, z^2, \dots , z^d) - \tilde{q}(dx^1, z^2, \dots , z^d)]_+ } \\
& + \inf_{x^1 \in \mathcal{X}^1} {\modar q^1}(z^1 | x^1) {\modar\int_{\mathcal{X}^1} [q(dx^1, z^2, \dots , z^d) - \tilde{q}(dx^1, z^2, \dots , z^d)]_{-} } \Big|,
\end{align*}
using the notation $[A]_+$ and $[A]_{-}$ for the positive and negative part of $A$, respectively.
We now have
\begin{align*}
& [q({\modar d}x^1, z^2, \dots , z^d) - \tilde{q}({\modar d}x^1, z^2, \dots , z^d)]_{-} \\
&= [q({\modar d}x^1, z^2, \dots , z^d) - \tilde{q}({\modar d}x^1, z^2, \dots , z^d)] - [q({\modar d}x^1, z^2, \dots , z^d) - \tilde{q}({\modar d}x^1, z^2, \dots , z^d)]_{+}.
\end{align*}

We remark that, if we integrate the last equation with respect to $x^1\in\mathcal{X}^1$, the middle term gives a contribution when $d\ge2$, since $\int_{\mathcal{X}^1} q({\modar d}x_1, z_2, \dots , z_d)  = q({\modar \emptyset , } z_2, \dots , z_d)$ is not simply $1$ as in the mono-dimensional case.
Hence, it provides
\begin{align}{\label{eq: 1.5, middle step 1}}
{\modar l[z^1, \dots , z^d]} & \le \Big|\big(\sup_{x^1 \in \mathcal{X}^1} {\modar q^1}(z^1 | x^1) - \inf_{x^1 \in \mathcal{X}^1} {\modar q^1}(z^1 | x^1) \big){\modar \int_{\mathcal{X}^1} [q(x^1, z^2, \dots , z^d) - \tilde{q}(x^1, z^2, \dots , z^d)]_+ }| \\
& + \inf_{x^1 \in \mathcal{X}^1} {\modar q^1(z^1 | x^1)}|q({\modar \emptyset, } z^2, \dots , z^d) - \tilde{q}({\modar \emptyset, } z^2, \dots , z^d)\Big|. \nonumber
\end{align}
We now observe that 

\begin{align}\nonumber
	\sup_{x^1 \in \mathcal{X}^1} {\modar q^1}(z^1 | x^1) - \inf_{x^1 \in \mathcal{X}^1} {\modar q^1}(z^1 | x^1) 
&=\inf_{x^1 \in \mathcal{X}^1} {\modar q^1}(z^1 | x^1) \left( 
\frac{\sup_{x^1 \in \mathcal{X}^1} {\modar q^1}(z^1 | x^1)}{\inf_{x^1 \in \mathcal{X}^1} {\modar q^1}(z^1 | x^1)}-1
\right)\\ \label{eq: sup q1 - infq1} 
&\le\inf_{x^1 \in \mathcal{X}^1} {\modar q^1}(z^1 | x^1) (e^{\alpha_1}-1)
=q^1(z^1| x^1_*) (e^{\alpha_1}-1),
\end{align}
where we used \eqref{eq: def density local privacy} and the definition  
$q^1(z^1| x^1_*) =\inf_{x^1 \in \mathcal{X}^1} {\modar q^1}(z^1 | x^1)$.
We replace it in \eqref{eq: 1.5, middle step 1}, which entails 
{\modar 
\begin{align*}
l[z^1, \dots , z^d] & \le q^1(z^1|x^1_*) (e^{\alpha_1} - 1) \int_{\mathcal{X}^1}  |q(dx^1, z^2, \dots , z^d) - \tilde{q}(dx^1, z^2, \dots , z^d)| \\
& + q^1(z^1|x^1_*)  |q(\emptyset,  z^2, \dots , z^d) - \tilde{q}( \emptyset,  z^2, \dots , z^d)| \\
& = {\modar l[dx^1, z^2, \dots , z^d]} + {\modar l[\emptyset, z^2, \dots , z^d]},
\end{align*}}
which concludes the proof of Step 1. \\

\vspace{0.3 cm}

\noindent \textbf{Step 2} \\
Assume now that $\zeta^i \in \{ \emptyset,{\modar d}x^i \}$ for any $i$ smaller than some $j$. Then, this step is devoted to the proof of 
$${\modar l[\zeta^1,\dots , \zeta^{j -1}, z^j, \dots , z^d]} \le {\modar l[\zeta^1,\dots , \zeta^{j -1}, \emptyset, z^{j + 1} ,\dots , z^d]} + {\modar l[\zeta^1,\dots , \zeta^{j -1}, dx^j, z^{j + 1}, \dots , z^d]}.$$
By definition it is indeed
\begin{align}{\label{eq: start step2 4}}
 l[\zeta^1,\dots , \zeta^{j -1}, z^j, \dots , z^d] &= \prod_{i = 1}^{j - 1}  q^i(z^i | x^i_*) \prod_{i < j: \zeta^i =  dx^i} (e^{\alpha_i} - 1) \\
& \times  \int_{\prod\limits_{i<j: \zeta^i = dx^i}\mathcal{X}^i}
 |q(\zeta^1,\dots , \zeta^{j -1}, z^j,\dots , z^d) - \tilde{q}(\zeta^1,\dots , \zeta^{j -1}, z^j, \dots , z^d)| 
\nonumber
\end{align}
We observe that, similarly to \eqref{eq: q_zi depuis q1 et qx1zi}, it is
 \begin{equation}\label{eq : q using Markov kernel j}
 	q(\zeta^1,\dots , \zeta^{j -1}, z^j, \dots , z^d) = {\modar \int_{x^j\in\mathcal{X}^j} q^j(z^j | x^j)q(\zeta^1,\dots , \zeta^{j -1}, dx^j, z^{j +1}, \dots , z^d)  }
 	\end{equation}
and also,
\begin{equation}\label{eq : q tilde using Markov kernel j}
	\tilde{q}(\zeta^1,\dots , \zeta^{j -1}, z^j, \dots , z^d) =  \int_{x^j\in\mathcal{X}^j} q^j(z_j | x_j)\tilde{q}(\zeta^1,\dots , \zeta^{j -1}, dx^j, z^{j +1}, \dots , z^d)  .
\end{equation}
Then, 
\begin{align}
& |q(\zeta^1,\dots , \zeta^{j -1}, z^j,\dots , z^d) - \tilde{q}(\zeta^1,\dots , \zeta^{j -1}, z^j, \dots , z^d)| 
\label{eq: diff q step 2 lem principal}\\
& = \Big|\int_{x^j \in \mathcal{X}^j}  q^j(z^j | x^j)[q(\zeta^1,\dots , \zeta^{j -1}, dx^j, z^{j +1}, \dots , z^d)-\tilde{q}(\zeta^1,\dots , \zeta^{j -1}, dx^j, z^{j +1}, \dots , z^d)]\Big|.
\nonumber
\end{align}
Acting as above \eqref{eq: 1.5, middle step 1}, remarking also that 
\begin{equation*}
\int_{{\modar x^j\in\mathcal{X}^j}} q(\zeta^1,\dots , \zeta^{j -1}, {\modar dx^j}, z^{j +1}, \dots , z^d)= q(\zeta^1,\dots , \zeta^{j -1}, {\modar \emptyset, }  z^{j +1}, \dots , z^d),
\end{equation*} it follows that the quantity \eqref{eq: diff q step 2 lem principal} is upper bounded by 
\begin{align*}
& \Big|(\sup_{x^j \in \mathcal{X}^j} {\modar q^j}(z^j | x^j) - \inf_{x^j \in \mathcal{X}^j} {\modar q^j}(z^j | x^j)) \\
& \times {\modar \int_{{\modar q^j \in \mathcal{X}^j}}  [q(\zeta^1,\dots , \zeta^{j -1}, {\modar dx^j}, z^{j +1}, \dots , z^d) - \tilde{q}(\zeta^1,\dots , \zeta^{j -1}, {\modar dx^j}, z^{j +1}, \dots , z^d)]_+ }\Big| \\
& + \inf_{x^j \in \mathcal{X}^j} {\modar q^j}(z^j | x^j)|q(\zeta^1,\dots , \zeta^{j -1},{\modar \emptyset, }  z^{j +1}, \dots , z^d) - \tilde{q}(\zeta^1,\dots , \zeta^{j -1},{\modar \emptyset, }  z^{j +1}, \dots , z^d)|.
\end{align*}
Thus, Equation \eqref{eq: sup q1 - infq1} with $q^j$ in place of $q^1$, entails
\begin{align*}
& |q(\zeta^1,\dots , \zeta^{j -1}, z^j,\dots , z^d) - \tilde{q}(\zeta^1,\dots , \zeta^{j -1}, z^j, \dots , z^d)| \\
& \le  (e^{\alpha_j} - 1) {\modar q^j}(z^j| x^j_*) {\modar \int_{x_j\in\mathcal{X}^j} |q(\zeta^1,\dots , \zeta^{j -1}, dx^j, z^{j +1}, \dots , z^d) - \tilde{q}(\zeta^1,\dots , \zeta^{j -1}, dx^j, z^{j +1}, \dots , z^d)| } \\
& +  {\modar q^j}(z^j | x^j_*)|q(\zeta^1,\dots , \zeta^{j -1},{\modar \emptyset, }  z^{j +1}, \dots , z^d) - \tilde{q}(\zeta^1,\dots , \zeta^{j -1}, {\modar \emptyset, } z^{j +1}, \dots , z^d)|.
\end{align*}
We replace it in \eqref{eq: start step2 4}, obtaining
\begin{align*}
&{\modar l[\zeta^1,\dots , \zeta^{j -1}, z^j, \dots , z^d] }\\
& \le \prod_{i = 1}^{j - 1} {\modar q^i}(z^i | x^i_*) \prod_{i < j: \zeta^i = {\modar dx^i}} (e^{\alpha_i} - 1)  \int_{\prod\limits_{i<j:\zeta^i=dx^i}\mathcal{X}^i} 
	\Big{[} (e^{\alpha_j} - 1) q^j(z^j|x^j_*)  \\
& \times {\modar \int_{x^j\in\mathcal{X}^j} |q(\zeta^1,\dots, \zeta^{j -1}, dx^j, z^{j +1}, \dots , z^d) - \tilde{q}(\zeta^1,\dots , \zeta^{j -1}, dx^j, z^{j +1}, \dots , z^d)| } \\
& + {\modar  q^j(z^j | x^j_*)}|q(\zeta^1,\dots , \zeta^{j -1},{\modar \emptyset, }  z^{j +1}, \dots , z^d) - \tilde{q}(\zeta^1,\dots , \zeta^{j -1}, {\modar \emptyset, } z^{j +1}, \dots , z^d)| \Big{]} \\
& = \prod_{i = 1}^{j} {\modar q^i}(z^i | x^i_*) \Big( {\modar \prod_{i < j : \zeta^i = dx^i}} (e^{\alpha_i} - 1) \Big)(e^{\alpha_j} - 1) \\
& \times  \int_{\big(\prod\limits_{i<j:\zeta^i=dx^i}\mathcal{X}^i\big) \times \mathcal{X}^{j}}  |q(\zeta^1,\dots , \zeta^{j -1}, dx^j, z^{j +1}, \dots , z^d) - \tilde{q}(\zeta^1,\dots , \zeta^{j -1}, dx^j, z^{j +1}, \dots , z^d)|
	\\
& + \prod_{i = 1}^{j} {\modar q^i}(z^i | x^i_*) \prod_{i < j: \zeta^i = {\modar dx^i}} (e^{\alpha_i} - 1) \\
& \times {\modar \int_{\prod\limits_{i<j:\zeta^i=dx^i}\mathcal{X}^i} |q(\zeta^1,\dots , \zeta^{j -1},{\modar \emptyset, }  z^{j +1}, \dots , z^d) - \tilde{q}(\zeta^1,\dots , \zeta^{j -1}, {\modar \emptyset, }  z^{j +1}, \dots , z^d)| 
}
	\\
& = {\modar l[\zeta^1,\dots , \zeta^{j -1}, dx^j, z^{j + 1}, \dots , z^d]} + {\modar l[\zeta^1,\dots , \zeta^{j -1}, \emptyset, z^{j + 1}, \dots , z^d]}.
\end{align*}
Thus, the proof of Step 2 is completed. \\
\\
To conclude, we want to show that Steps 1 and 2 imply the stated result of the proposition. From the two steps above we have, by recurrence, 
\begin{align*}
 l[z^1, \dots , z^d] & \le {\modar l[\emptyset, z^2, \dots , z^d]}+ {\modar l[dx^1, z^2, \dots , z^d]} \\
& \le {\modar l[\emptyset, \emptyset, z^3,\dots , z^d]} + {\modar l[\emptyset, dx^2, z^3 ,\dots , z^d]} + {\modar l[dx^1, \emptyset, z^3, \dots , z^d]} + {\modar l[dx^1, dx^2, z^3 ,\dots , z^d]} \\
& \le \sum_{(\zeta^1, \dots , \zeta^d) \in \prod_{j=1}^d\{ \emptyset, {dx^j} \} } {\modar l[\zeta^1, \dots , \zeta^d]},
\end{align*}
as we wanted.
\end{proof}

{\rev
\subsection{Proof of Proposition \ref{P: ajout main en divergence}}
The proof follows computations in the proof of Proposition 8 in \cite{Duchi_Ruan24}, together with a representation similar to {\revd the result of} Proposition \ref{prop: recurrence}. We start by introducing notations useful for the proof, which are modifications of \eqref{eq: def l z1 zd}. For
$\bm{z}=(z^1,\dots,z^d) \in \bm{\mathcal{Z}}$, $\bm{x}=(x^1,\dots,x^d)
\in \bm{\mathcal{X}}$, 
$\bm{x_0}=(x^1_0,\dots,x^d_0) \in \bm{\mathcal{X}}$:
\begin{align}
	\nonumber
	\delta^j(z^j \mid x^j,x^j_0)=&q^j(z^j \mid x)-q^j(z^j \mid x_0^j), 
	\text{
		$j\in \{1,\dots,d\}$,}\\
		\nonumber
	k[\zeta^1,\dots,\zeta^d]:=&
	q(z^1, \dots, z^d) - \tilde{q}(z^1, \dots, z^d) , ~  \text{ for $\bm{\zeta} = \bm{z}$},	
	\\ \nonumber
	k[\zeta^1,\dots,\zeta^d]:=&\prod_{j : \zeta^j = \emptyset } q^j(z^j | x^j_0)  \times
	\\ \label{eq: def k z1 zd}
	&
	\int_{\prod\limits_{j: \zeta^j = dx^j}\mathcal{X}^j} \prod_{j: \zeta^j = dx^j} \delta^j(z^j\mid x^j,x^j_0) [q(\zeta^1, \dots , \zeta^d) - \tilde{q}(\zeta^1, \dots , \zeta^d)],~  \text{ for $\bm{\zeta} \neq \bm{z}$}.
\end{align}
We recall that with slight abuse of notation, this defines $k[\emptyset,\dots,\emptyset]=0$.

We first prove the following lemma.
\begin{lemma}\label{L: lemma ajoute recurrence k}
	We have $k(z^1,\dots,z^d)=\displaystyle
	\sum_{(\zeta^1, \dots , \zeta^d) \in \prod_{j=1}^d\{ \emptyset, {dx^j} \} } { k[\zeta^1, \dots , \zeta^d]}$.
\end{lemma}
\begin{proof}
	The proof of this lemma is very similar to the proof of Proposition \ref{prop: recurrence} and we omit some details. By induction  it is sufficient to prove that, for $j \in \{1,\dots,d-1\}$, if $\zeta^i \in \{ \emptyset, dx^i \}$ for $j < i$, we have
	\begin{multline} \label{eq : recurrence dans lemme ajoute}
 		k[\zeta^1,\dots , \zeta^{j -1}, z^j, \dots , z^d] = k[\zeta^1,\dots , \zeta^{j -1}, dx^j, z^{j + 1} ,\dots , z^d] + 
		\\ k[\zeta^1,\dots , \zeta^{j -1}, \emptyset, z^{j + 1}, \dots , z^d].		
	\end{multline}
To prove \eqref{eq : recurrence dans lemme ajoute}, using
\eqref{eq : q using Markov kernel j}--\eqref{eq : q tilde using Markov kernel j},
 we write 
$	q[\zeta^1,\dots , \zeta^{j -1}, z^j, \dots , z^d]-\tilde{q}[\zeta^1,\dots , \zeta^{j -1}, z^j,\dots , z^d]$ as
\begin{equation*}
	  \int_{x^j\in\mathcal{X}^j} q^j(z^j | x^j)[q(\zeta^1,\dots , \zeta^{j -1}, dx^j, z^{j +1}, \dots , z^d)- 
\tilde{q}(\zeta^1,\dots , \zeta^{j -1}, dx^j, z^{j +1}, \dots , z^d) ].
\end{equation*}
Then, we plug in the last expression the equality $
q^j(z^j | x^j)= \delta^j(z^j \mid x^j,x_0^j)+q^j(z^j | x^j_0)$. It yields that 
$	q[\zeta^1,\dots , \zeta^{j -1}, z^j, \dots , z^d]-\tilde{q}[\zeta^1,\dots , \zeta^{j -1}, z^j, \dots , z^d]$ is
equal to
\begin{multline*}
	  \int_{x^j\in\mathcal{X}^j} \delta^j(z^j | x^j,x_0^j)[q(\zeta^1,\dots , \zeta^{j -1}, dx^j, z^{j +1}, \dots , z^d)- 
\tilde{q}(\zeta^1,\dots , \zeta^{j -1}, dx^j, z^{j +1}, \dots , z^d) ]
+ \\ q^j(z^j | x^j_0) 	q[\zeta^1,\dots , \zeta^{j -1}, \emptyset, z^{j+1}, \dots , z^d]-\tilde{q}[\zeta^1,\dots , \zeta^{j -1}, \emptyset, z^{j+1}, \dots , z^d].	
\end{multline*}
Now, \eqref{eq : recurrence dans lemme ajoute} is a consequence of the last equality and definition \eqref{eq: def k z1 zd}.
\end{proof}
\begin{proof}[Proof of \eqref{eq: main with f divergence}]
From the definition of the $f_l$-divergence, we have
\begin{equation} \label{eq: def D M Mtilde in proof}
D_{f_l}(M\|\tilde{M})=\int_{\mathcal{Z}^1 \times \dots\times \mathcal{Z}^d}
\frac{\abs{q(\bm{z}) - \tilde{q}(\bm{z})}^{l} }{\tilde{q}(\bm{z})^{l-1}} d \bm{\mu}(\bm{z}).
\end{equation}
Since $t \mapsto t^{1-l}$ is convex, we deduce from Jensen's inequality that
\begin{equation*}
	\tilde{q}(\bm{z})^{1-l}=\left(\int_{\bm{\mathcal{X}}} 
	\prod_{j=1}^d q^j(z^j \mid x_0^j) \tilde{P}(d\bm{x_0})\right)^{1-l}
	\le \int_{\bm{\mathcal{X}}} 
	\prod_{j=1}^d q^j(z^j \mid x_0^j)^{1-l} \tilde{P}(d\bm{x_0}).
\end{equation*}
Inserting this control in \eqref{eq: def D M Mtilde in proof}, and using Fubini-Tonelli theorem we deduce that
$D_{f_l}(M\|\tilde{M})=\int_{\bm{\mathcal{X}}} W(\bm{x_0})^l  \tilde{P}(d\bm{x_0})$ where
\begin{equation*}
W(\bm{x_0})=\Big(\int_{\mathcal{Z}^1 \times \dots\times \mathcal{Z}^d}
\frac{\abs{q(\bm{z}) - \tilde{q}(\bm{z})}^{l}}{
	\prod_{j=1}^d q^j(z^j\mid x_0^j)^{l-1}}
d\bm{\mu}(\bm{z}) \Big)^{1/l}.
\end{equation*}
Now, using that $\tilde{P}$ is a probability, the control \eqref{eq: main with f divergence} follows if we can prove that $W(\bm{x_0})$ is upper bounded independently of $\bm{x_0}$ by the RHS of \eqref{eq: main with f divergence}.
From the definition of $k[\eta^1,\dots,\eta^d]$ we have 
$\abs{q(\bm{z}) - \tilde{q}(\bm{z})}=k[z^1,\dots,z^d]$, and we use Lemma \ref{L: lemma ajoute recurrence k} to get,
	\begin{equation*}
	W(\bm{x_0})
	=\Big(\int_{\mathcal{Z}^1 \times \dots\times \mathcal{Z}^d}
			\abs[\Big]{ 	\sum_{(\zeta^1, \dots , \zeta^d)  \in \prod_{j=1}^d\{ \emptyset, {dx^j} \} } 
	\frac{ k[\zeta^1,\dots,\zeta^d] }
	{ \prod_{j=1}^d q^j(z^j\mid x_0^j)^{1-1/l}}}^l
	d\bm{\mu}(\bm{z}) \Big)^{1/l}.
\end{equation*}
It follows from the triangular inequality that 
\begin{equation} \label{eq: W as sum on all marginal}
	W(\bm{x_0})
\le 
	\sum_{ (\zeta^1, \dots , \zeta^d)  \in \prod_{j=1}^d\{ \emptyset, dx^j \} }    W_{\zeta^1,\dots,\zeta^d}(\bm{x_0}),
	\end{equation}
with 
\begin{equation*} W_{\zeta^1,\dots\zeta^d}(\bm{x_0}):= 
\Big(\int_{\mathcal{Z}^1 \times \dots\times \mathcal{Z}^d}
	\abs[\Big]{ \frac{k[\zeta^1,\dots,\zeta^d]}{
		\prod_{j=1}^d q^j(z^j\mid x_0^j)^{1-1/l}}}^l
	d\bm{\mu}(\bm{z}) \Big)^{1/l}.
\end{equation*}
We remind the definition \eqref{eq: def k z1 zd} of $k[\zeta^1,\dots,\zeta^d]$ and use generalized Minkowski's inequality to get
\begin{multline*} W_{\zeta^1,\dots\zeta^d}(\bm{x_0})\le 
		\int_{\prod\limits_{j: \zeta^j = dx^j}\mathcal{X}^j}
		\Big(	
	\int_{\mathcal{Z}^1 \times \dots\times \mathcal{Z}^d}
	\prod_{j : \zeta^j=dx^j} \abs[\big]{\frac{\delta(z^j\mid x^j,x_0^j)}{q^j(z^j \mid x^j_0)^{1-1/l}}}^l
	\times \\
	\prod_{j : \zeta^j=\emptyset} q(z^j \mid x^j_0)
	d\bm{\mu}(\bm{z}) 
	\Big)^{1/l}
	 \abs{q(\zeta^1, \dots , \zeta^d) - \tilde{q}(\zeta^1, \dots , \zeta^d)}.
\end{multline*}
Recall that the reference measure is in a product form, i.e. $\bm{\mu}(d \bm{z})=\prod_{j=1}^d
\mu^j(dz^j)$. Moreover, notice that $\zeta^j$ can only be $\emptyset$ or $ dx^j$ and that $\int_{\mathcal{Z}^j}q(z^j \mid x^j_0) \mu^j(dz^j)=1$ and $\int_{\mathcal{Z}^j} \abs[\big]{\frac{\delta(z^j\mid x^j,x_0^j)}{q^j(z^j \mid x^j_0)^{1-1/l}}}^l \mu^j(dz^j)=D_{f_l}
(Q^j(\cdot\mid X^j=x^j)\|Q^j(\cdot\mid X^j=x^j_0))$. From here we deduce
\begin{align*}
	W_{\zeta^1,\dots\zeta^d}(\bm{x_0})&
\begin{multlined}[t] \le 
	\int_{\prod\limits_{j: \zeta^j = dx^j}\mathcal{X}^j}
	\Big( \prod_{j: \zeta^j = dx^j}		
	D_{f_l}
	(Q^j(\cdot\mid X^j=x^j)\|Q^j(\cdot\mid X^j=x^j_0))
	\Big)^{1/l}\\
	\times \abs{q(\zeta^1, \dots , \zeta^d) - \tilde{q}(\zeta^1, \dots , \zeta^d)}
\end{multlined}
\\
& \le \int_{\prod\limits_{j: \zeta^j = dx^j}\mathcal{X}^j}
\Big( \prod_{j: \zeta^j = dx^j}	\varepsilon_j \Big)
\abs{q(\zeta^1, \dots , \zeta^d) - \tilde{q}(\zeta^1, \dots , \zeta^d)},
\end{align*}
where we used \eqref{eq : divergence CLDP}. Since
$\int_{\prod\limits_{j: \zeta^j = dx^j}\mathcal{X}^j}
\abs{q(\zeta^1, \dots , \zeta^d) - \tilde{q}(\zeta^1, \dots , \zeta^d)}$ is the total variance distance between the laws of $(X^j)_{j: \zeta^j = dx^j}$ under $P$ and $\tilde{P}$ we deduce
\begin{equation} \label{eq: control W zeta}
W_{\zeta^1,\dots\zeta^d}(\bm{x_0}) \le 	\big( \prod_{j: \zeta^j = dx^j}	\varepsilon_j \big) \times 
d_{TV}\big( L_{((X^j)_{j: \zeta^j = dx^j})},L_{((\tilde{X}^j)_{j: \zeta^j = dx^j})}\big)
\end{equation}
Collecting \eqref{eq: W as sum on all marginal}, \eqref{eq: control W zeta}, we deduce 
\begin{equation*}
	W(\bm{x_0}) \le 	\sum_{ (\zeta^1, \dots , \zeta^d)  \in \prod_{j=1}^d\{ \emptyset, dx^j \} } 	\big( \prod_{j: \zeta^j = dx^j}	\varepsilon_j \big) \times 
	d_{TV}\big( L_{((X^j)_{j: \zeta^j = dx^j})},L_{((\tilde{X}^j)_{j: \zeta^j = dx^j})}\big).
\end{equation*}
After reorganizing the sum in the equation above, we recognize the RHS of \eqref{eq: main with f divergence}, and the Proposition \ref{P: ajout main en divergence} is proved.
%
%
\end{proof}
}
\subsection{Locally private {\rev joint moment} estimation}
In this section we prove all the technical results related to the estimation of {\rev joint moment}
under local differential privacy constraints. 

\subsubsection{Proof of Theorem \ref{th: upper bound joint moment covariance}}

\begin{proof} 
	$\bullet$ We first prove the result for $\hat{\gamma}_n$ as stated in \eqref{eq: borne sup gamma}. This is based on a bias-variance decomposition for the $L^2$ risk. Let us denote the bias term by
	\begin{equation}\label{eq: def biais cov}
		{\rev b_{T^{(1)},T^{(2)},\dots,T^{(d)}}:=\E[\prod_{j=1}^dZ^j_i]-\E[\prod_{j=1}^dX^j_i]=
			\E[\prod_{j=1}^dZ^j]-\E[\prod_{j=1}^dX^j].}		
	\end{equation}	
We first state as lemma a control on the magnitude of this bias term, its proof can be found at the end of this section.
\begin{lemma}\label{l: biais b T1 T2}
	We have {\rev
	\begin{equation*}
		\abs*{ b_{T^{(1)},T^{(2)},\dots,T^{(d)}}}
			\le \sum_{l=1}^d \Big\{ (T^{(l)})^{-k_l(1-\frac{d}{\overline{k}})}
			\Big( \prod_{\substack{j=1\\j\neq l}}^d \norm{X^j}_{k_j} \Big) \norm{X^l}_{k_l}^{1+k_l(1-\frac{d}{\overline{k}})} \Big\}.
	\end{equation*}}
\end{lemma}
\noindent To get \eqref{eq: borne sup gamma}, we write the bias variance decomposition,
$\E\left[( \hat{\gamma}_n-\gamma )^2 \right]= {\rev (b_{T^{(1)},\dots,T^{(d)}})^2} + \text{var}(\hat{\gamma}_n) $. By independence and {\rev H\"older's inequality,}
\begin{align} \label{eq: var basic gamma}
	{\rev 
	\text{var}(\hat{\gamma}_n)=n^{-2} \sum_{i=1}^{n} \text{var}(\prod_{j=1}^dZ^j_i) \le 
	n^{-1}\E(\abs{\prod_{j=1}^dZ^j}^2) \le n^{-1} \prod_{j=1}^d \E(\abs{Z^1}^{2d})^{\frac{1}{d}}.}
\end{align}
We have for {\rev $j\in \{1,\dots,d\}$,} 
\begin{multline*}
	\E(\abs{Z^j}^{\rev 2d})=\E(\abs{\tronc{X^j}{T^{(j)}}+\mathcal{E}^j}^{2d})\le
{\rev 2^{2d-1}}\E[\tronc{X^j}{T^{(j)}}^{\rev 2d} ]+ {\rev 2^{2d-1}}\E[\abs{\mathcal{E}^j}^{\rev 2d}]
\\ \le
 {\rev 2^{2d-1}}
 |T^{(j)}|^{\rev 2d} \left( 1+ \left(\frac{2}{\alpha_j}\right)^{\rev 2d} \E\left[\abs{\mathcal{L}(1)}^{\rev 2d}\right]  \right), 
 \end{multline*} as
$\mathcal{E}^j$ is equal in law to a $\frac{2T^{(j)}}{\alpha_j} \times \mathcal{L}(1)$ variable.
We deduce {\rev $\E(\abs{Z^j}^{2d}) \le 2^{2d-1}\abs{T^{(j)}}^{2d}(1+\frac{2^{2d}d!}{\alpha^{2d}_j}) \le C\abs{T^{(j)}}^{2d}/\alpha^{2d}_j  $ for some constant $C$, using $\alpha_j \le 1$.} From \eqref{eq: var basic gamma}, it entails 
 \begin{equation}\label{eq: upper bound var est gamma} 
 \text{var}(\hat{\gamma}_n)\le C n^{-1} 
 {\rev \frac{\prod_{j=1}^d\abs{T^{(j)}}^2}{\prod_{j=1}^d\alpha^2_j},}
 \end{equation}
 for some  {\rev constant} $C>0$. In turn, recalling Lemma \ref{l: biais b T1 T2}, we get \begin{equation} \label{eq: bias T1 T2 Var est cov}
 	\E\left[ (\hat{\gamma}_n-\gamma)^2 \right] \le c 
 	\left(
 	{\rev \sum_{j=1}^d 
 	(T^{(j)})^{-2k_j(1-\frac{2}{\overline{k}})} 
 	 +n^{-1}\frac{\prod_{j=1}^d\abs{T^{(j)}}^2}{\prod_{j=1}^d\alpha^2_j}}
 \right)
 .
\end{equation}
{\rev The calibration given in the statement of the theorem
is for all $j\in\{1,\dots,d\} $, 	$(T^{(j)})^{k_j}=(n\prod_{j=1}^d\alpha_j^2)^{1/2}$,} which is such that {\revd all the terms} in the right hand side of \eqref{eq: bias T1 T2 Var est cov} {\revd equilibrate} and 
$	\E\left[ (\hat{\gamma}_n-\gamma)^2\right]\le c 
{\rev (n\prod_{j=1}^d\alpha^2_j)^{-(1-\frac{d}{\overline{k}})}
 =c(n\prod_{j=1}^d\alpha^2_j)^{-\frac{\overline{k}-d}{\overline{k}}}}$. {\rev The proof of \eqref{eq: borne sup gamma} is complete.
 }
\end{proof}

\begin{proof}[Proof of Lemma \ref{l: biais b T1 T2}] 
	From \eqref{eq: def Z covariance} and \eqref{eq: def biais cov}, we have
	\begin{align}\nonumber
		b_{{\rev T^{(1)},\dots,T^{(d)}}} & =\E \left[  
		{\rev \prod_{j=1}^d \left(\tronc{X^j}{T^{(j)}} + \mathcal{E}^j \right)} 
		\right]
		-\E \left[  {\rev  \prod_{j=1}^d X^j }  \right]
		\\
		&=  \label{eq: biais second expression est cov}
		\E
		\left[ {\rev \prod_{j=1}^d \tronc{X^j}{T^{(j)}} }   \right]
		-\E \left[   {\rev  \prod_{j=1}^d X^j }   \right],
	\end{align}
	where we used that {\rev the variables $\mathcal{E}^j$, $j=1,\dots,d$} are centered and independent of {\rev the variables $X^j$ $j=1,\dots,d$.} 
	We set $\Delta^{(j)}=\tronc{X^j}{T^{(j)}}-X^j$ for {\rev $j=1,\dots,d$} and get,
	{\rev 
		\begin{equation*} 
			b_{{\rev T^{(1)},\dots,T^{(d)}}}=
			\sum_{l=1}^d \E\left[
			\big(\prod_{j=1}^{l-1} X^j\big) \Delta^{(l)} \big(
			\prod_{j=l+1}^d\tronc{X^j}{T^{(j)}}
			\big)
			\right]
		\end{equation*}
		We remark that $\Delta^{(l)}=\Delta^{(l)} \one_{ \{ \abs{X^l} \ge T^{(l)}\}}$ and deduce
		\begin{equation*}
			\abs{b_{{\rev T^{(1)},\dots,T^{(d)}}}}\le
			\sum_{l=1}^d \E\left[
			\big(\prod_{j=1}^{l-1} \abs{X^j}\big) \abs{\Delta^{(l)}} \one_{ \{ \abs{X^l} \ge T^{(l)}\}} \big(
			\prod_{j=l+1}^d|\tronc{X^j}{T^{(j)}}|
			\big)
			\right]
		\end{equation*}
		We assess the magnitude of each term in the sum above.
		Using that $\sum_{j=1}^d \frac{1}{k_j}<1$, we define $r>1$ by the identity $
		{\revd \sum_{j=1}^{d}} \frac{1}{k_j} + \frac{1}{r}=1$. We now apply for the term corresponding to index $l$ in the sum above the H\"older inequality 
		{\revd to a product of $d+1$ quantities} with coefficients $\frac{1}{k_1}+\dots+\frac{1}{k_{l-1}}+\frac{1}{k_{l}}+\frac{1}{r}+\frac{1}{k_{l+1}}+\dots\frac{1}{k_d}=1$. This yields,
		\begin{equation*}
			\abs{b_{{\rev T^{(1)},\dots,T^{(d)}}}}\le
			\sum_{l=1}^d 
			\big(\prod_{j=1}^{l-1} \norm{X^j}_{k_j} \big) \norm{\Delta^{(l)}}_{k_l} \norm{\one_{ \{ \abs{X^l} \ge T^{(l)}\}}}_r \big(
			\prod_{j=l+1}^d\norm{\tronc{X^j}{T^{(j)}}}_{k_j}.
			\big)
		\end{equation*}
		Using that $ \abs{\Delta^{(l)}} \le \abs{X^l}$ and $\abs{\tronc{X^j}{T^{(j)}}} \le \abs{X^j}$, we deduce
		\begin{equation*}
			\abs{b_{{\rev T^{(1)},\dots,T^{(d)}}}}\le
			\sum_{l=1}^d 
			\big(\prod_{j=1}^{d} \norm{X^j}_{k_j} \big)  \norm{\one_{ \{ \abs{X^l} \ge T^{(l)}\}}}_r.
		\end{equation*}
		By Markov inequality $ \norm{\one_{ \{ \abs{X^l} \ge T^{(l)}\}}}_r = \mathbb{P}\left( \abs{X^l} \ge T^{(l)} \right)^{1/r} \le \Big( \frac{ \norm{X^l}_{k_l}^{k_l } }{|T^{(l)}|^{k_l }} \big)^{1/r}$ and we obtain
		\begin{equation*}
			\abs{b_{{\rev T^{(1)},\dots,T^{(d)}}}}\le
			\sum_{l=1}^d 
			\big(\prod_{\substack{j=1\\j \neq l}}^{d} \norm{X^j}_{k_j} \big) \norm{X^l}^{1+\frac{k_l}{r}}_{k_l} 
			|T^{(l)}|^{-k_l/r}.
		\end{equation*}
		As $\frac{1}{r}=1-\sum_{j=1}^d \frac{1}{k_j}=1-\frac{d}{\overline{k}}$, we deduce
		the lemma.
	}
\end{proof}

\subsubsection{Proof of Corollary \ref{C: upper bound covariance}}
\begin{proof}
{\rev 	By direct application of Theorem \ref{th: upper bound joint moment covariance}, we have
	\begin{equation} \label{eq: borne sup gamma dim 2}
			\E \left[ \abs*{\hat{\gamma}_n - \gamma}^2\right] \le c (n\alpha^2_1\alpha^2_2)^{-\frac{\overline{k}-2}{\overline{k}}}.
	\end{equation}	
We now deduce the result }\eqref{eq: borne sup theta} on $\hat{\theta}_n$. An additional error appears in the estimation of $\theta$ due to the inference of both means $m_1$, $m_2$. We will see that these additional errors are at most of the same magnitude as the estimation error of the cross term $\gamma=E[X^1X^2]$. We need to recall the bias-variance decomposition given in the proof of Corollary 1 in \cite{Martin} :
\begin{equation} \label{eq: biais variance Duchi}
\E\left[(\hat{m}_n^{(j)}-m^{(j)})^2\right] \le 
c\left( (T^{(j)})^{-(k_j-1)} + n^{-1}\frac{\abs{T^{(j)}}^2}{\alpha^2_j} \right), \quad \text{for $j \in \{1,2\}$},
\end{equation}  
where the constant $c$ does not depend on $n\ge 1$, $\alpha_j\in(0,1]$. Let us emphasize that the optimal trade-off in \eqref{eq: biais variance Duchi} is given by the choices of $T^{(j)}$ appearing in the statement of Theorem \ref{th: recall Duchi est moyenne}. However, our choice  $T^{(j)}=(n\alpha_1^2 \alpha_2^2)^{\frac{1}{2k_j}}$ is tailored to get the optimal trade-off while estimating $\gamma$ with $\alpha_1, \alpha_2 < 1$, and yields to a $\mathbf{L}^2$-risk for the estimation of the means which is suboptimal in $\alpha_j$. {\modar Replacing in \eqref{eq: biais variance Duchi}, we get,}
\begin{multline} \label{eq: control moyenne sous opt}
	\E\left[(\hat{m}_n^{(j)}-m^{(j)})^2\right] \le c\left((n\alpha_1^2\alpha_2^2)^{-\frac{k_j-1}{k_j}}  +
	n^{-1} \frac{(n\alpha_1^2\alpha_2^2)^2}{\alpha_j^2}
	  \right) \le
	c (n\alpha^2_1\alpha_2^2)^{-\frac{k_j-1}{k_j}}\left(1+\alpha_{3-j}^2\right)
	\\ \le  c (n\alpha^2_1\alpha_2^2)^{-\frac{k_j-1}{k_j}} \quad \text{for $j \in \{1,2\}$}.
\end{multline}  
Now, we split
\begin{equation} \label{eq: split gamma theta est cov}
	\hat{\theta}_n - \theta = \hat{\gamma}_n -\gamma
-\left[ (\hat{m}^{(1)}_n-m^{(1)}) m^{(2)} + \hat{m}^{(1)}_n(\hat{m}^{(2)}_n-m^{(2)}) 
 \right],
\end{equation}
and denote by $\sum_{l=1}^2 e_n^{(l)}$ the two terms in the bracket of the above equation. By  \eqref{eq: control moyenne sous opt}, we
have, 
{\modar 
\begin{multline}
	\label{eq: controle e1n est cov}
	\E\left[ 
	\abs{e_n^{(1)}}^2
	\right]= 
	\E \left[ (\hat{m}^{(1)}_n-m^{(1)})^2 	\right] 
	 \abs{m^{(2)}}^2 \le c (n\alpha^2_1\alpha^2_2)^{-\frac{k_1-1}{k_1}}
	 \\\le c(n\alpha^2_1\alpha^2_2)^{-[1-\frac{1}{k_1}-\frac{1}{k_2}]}
=c(n\alpha^2_1\alpha^2_2)^{-\frac{\overline{k}-2}{\overline{k}}}
\end{multline}}
where we used $n\alpha^2_1\alpha^2_2 \ge 1$ and $X^2 \in \mathbf{L}^{k_2} \subset \mathbf{L}^{1}$. 
The control of the {\modar second} term necessitates more care.
{\modar Using \eqref{eq: def Z covariance}, we have, recalling the definition $\hat{m}^{(1)}_n=n^{-1} \sum_{i=1}^{n} Z^1_i$, as given in \eqref{eq: def m hat est cov},}
\begin{align*}
	\E\left[ 
\abs{e_n^{(2)}}^2
\right]& = 	\E\left[ \abs{\hat{m}^{(1)}_n}^2 (\hat{m}_n^{(2)}-m^{(2)})^2\right]\\
&\le
	2 \E\left[ \left(\frac{1}{n}\sum_{i=1}^n \tronc{X_i^1}{T^{(1)}} \right)^2 (\hat{m}_n^{(2)}-m^{(2)})^2\right]	+
	2\E\left[ \left(\frac{1}{n}\sum_{i=1}^n \mathcal{E}^1_i \right)^2  (\hat{m}_n^{(2)}-m^{(2)})^2\right]
	\\ 
 & \le  	2 \abs{T^{(1)}}^2 \E\left[(\hat{m}_n^{(2)}-m^{(2)})^2\right]
	+2 \E\left[ \left(\frac{1}{n}\sum_{i=1}^n\mathcal{E}^1_i \right)^2 \right] \E\left[(\hat{m}_n^{(2)}-m^{(2)})^2\right],
\end{align*}
where we used $ \abs*{\tronc{X_i^1}{T^{(1)}}} \le T^{(1)}$ for the first expectation, and the independence of $(\mathcal{E}^1_i)_i$ and $\hat{m}_n^{(2)}$ for the second one.
Recalling $T^{(1)}=(n\alpha^2_1\alpha^2_2)^{\frac{1}{2k_1}}$ and that $\mathcal{E}_i^1$ are iid centered variable with variance $\frac{8\abs{T^{(1)}}^2}{\alpha^2_1}$, we get
	\begin{align} \nonumber
		\E\left[ 
	\abs{e_n^{(2)}}^2\right] &\le c\left( \abs{T^{(1)}}^2 + \frac{{\abs{T^{(1)}}^2}}{n\alpha^2_1}\right) \E\left[(\hat{m}_n^{(2)}-m^{(2)})^2\right]
\\ \nonumber
	&\le c\left( (n\alpha^2_1\alpha^2_2)^{\frac{1}{k_1}} + 
\alpha^2_1(n\alpha^2_1\alpha_2^2)^{\frac{1}{k_1}-1}	\right)
\E\left[(\hat{m}_n^{(2)}-m^{(2)})^2\right]
\\ \nonumber
& \le c\left( (n\alpha^2_1\alpha_2^2)^{\frac{1}{k_1}} + 
\alpha^2_1(\alpha^2_1\alpha_2^2)^{\frac{1}{k_1}-1}	\right)
	(n\alpha^2_1\alpha_2^2)^{-\frac{k_2-1}{k_2}} 
	\\& \label{eq: controle e2n est cov}
	\le c (n\alpha^2_1\alpha_2^2)^{\frac{1}{k_1}} (n\alpha^2_1\alpha_2^2)^{-\frac{k_2-1}{k_2}}  =c(n\alpha^2_1\alpha_2^2)^{-\frac{\overline{k}-2}{\overline{k}}},
	\end{align}
where we used $\alpha_1^2 \le 1$, $n\alpha^2_1\alpha_2^2 \ge 1$.

\noindent {\rev Collecting \eqref{eq: borne sup gamma dim 2} with} \eqref{eq: split gamma theta est cov}-- \eqref{eq: controle e2n est cov}, we deduce \eqref{eq: borne sup theta}.
\end{proof}
\subsubsection{Proof adaptive procedure}
Before proving Theorem \ref{th: est gamma adaptive}, we introduce several notations and state some auxiliary lemmas.
We set for ${\rev\bm{T}} \in \mathcal{T}$,
\begin{align} \label{eq: def D est cov}
	&\mathbb{D}_{\rev \bm{T}}=\left( \sup_{{\rev\bm{T}}^\prime \in \mathcal{T}} \abs*{\E\left[\hat{\gamma}^{({\rev\bm{T}},{\rev\bm{T}}^\prime)}_n-\hat{\gamma}^{({\rev\bm{T}}^\prime)}_n\right]}\right) \vee \abs*{
	\E\left[\hat{\gamma}^{({\rev\bm{T}})}_n\right]-\gamma},
\\ \label{eq: def b barre est adapte} & \overline{\bm{b}}_{{\rev\bm{T}}}= \overline{\bm{b}}_{ {\rev (T^{(1)},\dots,T^{(d)})}}= 
{\rev 	\sum_{l=1}^d \Big\{ |T^{(l)}|^{-k_l(1-\frac{d}{\overline{k}})}  \big(\prod_{\substack{j=1\\j \neq l}}^{d} \norm{X^j}_{k_j} \big) \norm{X^l}^{1+k_l(1-\frac{d}{\overline{k}})}_{k_l} \Big\}
}.
\end{align}
The quantity $\overline{\bm{b}}_{\rev\bm{T}}$ is some upper bound for the bias term according to Lemma \ref{l: biais b T1 T2}. We show in the next lemma that it also controls $\mathbb{D}_{\rev{\bm{T}}}$.
\begin{lemma}\label{l: control DT est cov}
	We have $\mathbb{D}_{\rev\bm{T}} \le 2 \overline{\bm{b}}_{\rev\bm{T}}$.
\end{lemma}
\begin{proof}
	By Lemma \ref{l: biais b T1 T2} we have $\abs*{\E\left[\hat{\gamma}^{({\rev\bm{T}})}_n\right]-\gamma} {\rev {}= \abs{
b_{T^{(1)},\dots,T^{(d)}}} } \le \overline{\bm{b}}_{\rev\bm{T}}$ for ${\rev\bm{T}}\in \mathcal{T}$.
 For $({\rev\bm{T}},{\rev\bm{T}}^\prime) \in \mathcal{T}^2$, 
	recalling \eqref{eq : gamma hat T adapt}--\eqref{eq : gamma hat T Tprime adapt},
	we have
	\begin{align*}
		\E\left[ \hat{\gamma}^{({\rev\bm{T}},{\rev\bm{T}}^\prime)}_n-\hat{\gamma}^{({\rev\bm{T}}^\prime)}_n
		\right]
		&=
		\E\left[{\rev \prod_{j=1}^d Z^{j,T^{(j)}\wedge T^{\prime(j)}}_1 } \right] 
		-\E\left[{\rev \prod_{j=1}^d Z^{j,T^{\prime(j)}}_1}   \right] \\
		&=\E\left[ {\rev \prod_{j=1}^d \tronc{X^j_1}{T^{(j)}\wedge T^{\prime(j)}} }  \right] 
		-\E\left[ {\rev  \prod_{j=1}^d \tronc{X^j_1}{T^{\prime(j)}}}  \right]
		\\&=:{\rev \bm{b}_{ \bm{T},\bm{T}^\prime} },
	\end{align*}
	where we used definitions \eqref{eq: def Zi adpat} and the centering of the Laplace variables in the second line.
	Now we show {\rev $\abs*{ \bm{b}_{\bm{T} ,\bm{T}'}} \le 2 
		\overline{\bm{b}}_{\bm{T}}$.}
{\rev	We write $ \bm{b}_{\bm{T} ,\bm{T}'} = \sum_{l=1}^d \bm{b}^{(l)}_{\bm{T},\bm{T}'}$ where
	\begin{equation*}
		\bm{b}^{(l)}_{\bm{T}, \bm{T}'}=
		\E \left[
		\Big(\prod_{j=1}^{l-1} \tronc{X^j_1}{T^{\prime(j)}} \Big)
		\Big(\tronc{X^l_1}{T^{(l)}\wedge T^{\prime(l)}}-\tronc{X^l_1}{T^{\prime(l)}} \Big)
		\Big(\prod_{j=l+1}^{d} 
		\tronc{X^j_1}{T^{(j)}\wedge T^{\prime(j)}} \Big)
				\right].
	\end{equation*} 
Let us study 	$\bm{b}^{(l)}_{\bm{T} \wedge \bm{T}'}$ :
\begin{itemize}
	\item {\em Case 1 : $T^{(l)} \ge T^{\prime(l)}$.} Then, $\tronc{X^l_1}{T^{(l)}\wedge T^{\prime(l)}}=\tronc{X^l_1}{T^{\prime(l)}} $ and $	\bm{b}^{(l)}_{\bm{T} \wedge \bm{T}'}=0$.
	\item {\em Case 2 : $T^{(l)} \le T^{\prime(l)}$.} Then, $\tronc{X^l_1}{T^{(l)}\wedge T^{\prime(l)}}=\tronc{X^l_1}{T^{(l)}}$ and we write
	\begin{multline*}
			\bm{b}^{(l)}_{\bm{T} , \bm{T}'}=
		\E \left[
		\Big(\prod_{j=1}^{l-1} \tronc{X^j_1}{T^{\prime(j)}} \Big)
		\Big(\tronc{X^l_1}{T^{(l)}}-X^l_1 \Big)
		\Big(\prod_{j=l+1}^{d} 
		\tronc{X^j_1}{T^{(j)}\wedge T^{\prime(j)}} \Big)
		\right]+\\
				\E \left[
		\Big(\prod_{j=1}^{l-1} \tronc{X^j_1}{T^{\prime(j)}} \Big)
		\Big(X^l_1-\tronc{X^l_1}{T^{\prime(l)}} \Big)
		\Big(\prod_{j=l+1}^{d} 
		\tronc{X^j_1}{T^{(j)}\wedge T^{\prime(j)}} \Big)
		\right].
	\end{multline*}
Now, we use  $\abs{\tronc{X^l_1}{T^{(l)}} - X^l_1}= \abs{\tronc{X^l_1}{T^{(l)}}-X^l_1} \one_{ \{ \abs{X^l} \ge T^{(l)}\}} 
\le \abs{X^l_1} \one_{ \{ \abs{X^l} \ge T^{(l)}\}}$. Similarly,  $\abs{X^l_1-\tronc{X^l_1}{T^{\prime(l)}}}\le 
\abs{X^l_1} \one_{ \{ \abs{X^l} \ge T^{\prime(l)}\}}$ and $	\abs{\tronc{X^j_1}{T^{(j)}\wedge T^{\prime(j)}}} \le 
\abs{	\tronc{X^j_1}{ T^{\prime(j)}} } \le \abs{X^j_1}$. We
deduce
	\begin{multline*}
	\abs{\bm{b}^{(l)}_{\bm{T} ,\bm{T}'}}\le
	\E \left[
	\Big(\prod_{j=1}^{l} \abs{X^j_1} \Big)
	\Big(\abs{X^l_1}  \one_{ \{ \abs{X^l} \ge T^{(l)}\}}  \Big)
	\Big(\prod_{j=l+1}^{d} 
	\abs{X^j_1}\Big)
	\right]+\\
	\E \left[
	\Big(\prod_{j=1}^{l} \abs{X^j_1} \Big)
	\Big(\abs{X^l_1}  \one_{ \{ \abs{X^l} \ge T^{\prime(l)}\}}  \Big)
	\Big(\prod_{j=l+1}^{d} 
	\abs{X^j_1}\Big)
	\right].
\end{multline*}
Now, following the proof of Lemma \ref{l: biais b T1 T2}, we can show
\begin{align*}
		\abs{\bm{b}^{(l)}_{\bm{T}, \bm{T}'}}
&		\le 
\left[
 (T^{(l)})^{-k_l(1-\frac{d}{\overline{k}})} + (T^{\prime(l)})^{-k_l(1-\frac{d}{\overline{k}})}\right]
\Big( \prod_{\substack{j=1\\j\neq l}}^d \norm{X^j}_{k_j} \Big) \norm{X^l}_{k_l}^{1+k_l(1-\frac{d}{\overline{k}})},\\
&
		\le 
2 (T^{(l)})^{-k_l(1-\frac{d}{\overline{k}})}
\Big( \prod_{\substack{j=1\\j\neq l}}^d \norm{X^j}_{k_j} \Big) \norm{X^l}_{k_l}^{1+k_l(1-\frac{d}{\overline{k}})},
\end{align*}
where in the second line we used $T^{\prime(l)}\ge T^{(l)}$.
\end{itemize}
Summing in $l$ the controls we have just obtained for $\abs{\bm{b}^{(l)}_{\bm{T} , \bm{T}'}}$, we deduce  
 $\abs*{ \bm{b}_{\bm{T} , \bm{T}'}} \le 2 
\overline{\bm{b}}_{\bm{T}}$, from the definition \eqref{eq: def b barre est adapte}.
This proves the lemma.
}
\end{proof}
The following proposition shows that $\mathbb{B}_T$, defined in \eqref{eq: def est biais adaptatif est cov}, can be compared to the bias and heuristically justifies the choice \eqref{eq: def That est cov}.
\begin{proposition} \label{P: maj BBT est cov}	
	Assume that {\rev $\sum_{j=1}^d \frac{1}{k_j}<1$,} $\beta_n^j=\frac{\alpha_j}{\lfloor \log_2(n) \rfloor}$, for {\rev $j=1,\dots,d$} and $\kappa_n=c_0\log_2(n)$ for some $c_0>0$.
	Then, if $c_0$ is large enough, there exists $c>0$, $\overline{c}>0$, such that for all ${\rev\bm{T}} \in \mathcal{T}$, $\forall n \ge 1$, {\rev$(n \prod_{j=1}^d\alpha^2_j)/\log(n)^{2d+1} \ge 1 $,}
	\begin{equation*}
		\E\left[\mathbb{B}_{\rev\bm{T}}\right] \le c \overline{\bm{b}}^2_{\rev\bm{T}} + \frac{c}{n^{\overline{c}}{\rev \prod_{j=1}^d\alpha^2_j}}.
	\end{equation*}
The constant $\overline{c}$ can be arbitrarily large if $c_0$ is chosen large enough.	
\end{proposition}
\begin{proof}
For $({\rev\bm{T}},{\rev\bm{T}}^\prime) \in \mathcal{T}^2$, we write,
\begin{align*}
	\abs{\hat{\gamma}^{({\rev\bm{T}},{\rev\bm{T}}^\prime)}_n-\hat{\gamma}^{({\rev\bm{T}}^\prime)}_n}^2 \le 8 \abs*{\hat{\gamma}^{({\rev\bm{T}},{\rev\bm{T}}^\prime)}_n - \E\left[ \hat{\gamma}^{({\rev\bm{T}},{\rev\bm{T}}^\prime)}_n  \right] }^2 +
	 8 \abs*{\hat{\gamma}^{({\rev\bm{T}}^\prime)}_n - \E\left[ \hat{\gamma}^{({\rev\bm{T}}^\prime)}_n  \right] }^2
	 +8 \abs*{\E\left[\hat{\gamma}^{({\rev\bm{T}}^\prime)}_n\right] - \E\left[ \hat{\gamma}^{({\rev\bm{T}},{\rev\bm{T}}^\prime)}_n  \right] }^2	,
\end{align*}
and deduce from \eqref{eq: def est biais adaptatif est cov} and \eqref{eq: def D est cov} that
\begin{align}
	\nonumber
	\mathbb{B}_{T} &\le 
	8 \sum_{{\rev\bm{T}}^\prime \in \mathcal{T}}
		\left\{   \left( \abs*{\hat{\gamma}^{({\rev\bm{T}},{\rev\bm{T}}^\prime)}_n - \E[\hat{\gamma}^{({\rev\bm{T}},{\rev\bm{T}}^\prime)}_n ] }^2 -
	\frac{1}{16}\mathbb{V}_{{\rev\bm{T}}^\prime}\right)_+ \right\}
\\	& \nonumber \qquad\qquad\qquad+
	8 \sum_{{\rev\bm{T}}^\prime \in \mathcal{T}}
	\left\{   \left( \abs*{\hat{\gamma}^{({\rev\bm{T}}^\prime)}_n - \E[\hat{\gamma}^{({\rev\bm{T}}^\prime)}_n ] }^2 -
	\frac{1}{16}\mathbb{V}_{{\rev\bm{T}}^\prime}\right)_+ \right\}
	+ 8 \mathbb{D}_{\rev\bm{T}}^2
	\\& \label{eq: maj B B1 B2 DT}
	=:8\left[ \mathbb{B}_{{\rev\bm{T}}}^{(1)}+ \mathbb{B}_{{\rev\bm{T}}}^{(2)} + \mathbb{D}_{\rev\bm{T}}^2\right].
\end{align}
Using Lemma \ref{l: control DT est cov}, we see that the proposition will be proved as soon as we show, 
\begin{equation} \label{eq: control Bl preuve control D}
	\E\left[\mathbb{B}^{(l)}_{\rev\bm{T}}\right] \le \frac{c}{n^{\overline{c}}{\rev \prod_{j=1}^d \alpha_j^2}},  \text{ for } {\rev\bm{T}} \in \mathcal{T}, \text{ and }l=1,2.
\end{equation}
First, we focus on
$\E\left[\mathbb{B}^{(2)}_{\rev\bm{T}}\right]$. For ${\rev\bm{T}}^\prime \in \mathcal{T}$, we denote by $g_{{\rev\bm{T}}^\prime}:\mathbb{R}^{\rev d}\times\mathbb{R}^{\rev d}$ the function 
defined as {\rev $g_{{\bm{T}}^\prime}(\bm{x},\bm{e})=g_{{\rev\bm{T}}^\prime}(\bm{x},e_1,\dots,e_d)=
	\prod_{j=1}^d(\tronc{x_j}{T^{\prime{(j)}}}+ e_j)$,} which is   
such that
\begin{equation*}
	{\rev g_{{\bm{T}}^\prime}(\bm{X}_i,{\rev  \mathcal{E}^{{(1)},T^{\prime{(1)}}}_i, \dots , \mathcal{E}^{{(d)},T^{\prime{(d)}}}_i} )=
	 \prod_{j=1}^d \big( \tronc{X_i^1}{T^{\prime{(j)}}}+\mathcal{E}^{(j),T^{\prime{(j)}}}_i \big)
		= \prod_{j=1}^d 	Z^{j,T^{\prime{(j)}}}_i,}
\end{equation*}
	by \eqref{eq: def Zi adpat}. Thus, 
\begin{equation*} 
	\hat{\gamma}^{({\rev\bm{T}}^\prime)}_n-\E\left[\hat{\gamma}^{({\rev\bm{T}}^\prime)}_n\right]
	= \frac{1}{n}\sum_{i=1}^n \left\{g_{{\rev\bm{T}}^\prime}(\bm{X}_i,{\rev  \mathcal{E}^{{(1)},T^{\prime{(1)}}}_i, \dots, \mathcal{E}^{{(d)},T^{\prime{(d)}}}_i} )- 
	\E\left[
g_{{\rev\bm{T}}^\prime}(\bm{X}_i, {\rev \mathcal{E}^{{(1)},T^{\prime{(1)}}}_i,\dots, \mathcal{E}^{{(d)},T^{\prime{(d)}}}_i} )	\right]
\right\},\end{equation*}
recalling \eqref{eq : gamma hat T adapt}. We intend to apply Bernstein's inequality and introduce a set on which the random variables we consider are bounded. Let 
\begin{equation*}
	\Omega_n=\left\{ \omega \in \Omega \mid \forall j\in {\rev \{1,\dots,d\},} ~\forall T \in \mathcal{T}^{(j)}, ~\forall i \in \{1,\dots,n\}, 
	\text{we have } \abs{\mathcal{E}_i^{{(j)},T}} \le T \tilde{\kappa}_n^{(j)} 
	\right\}
\end{equation*}
where $\tilde{\kappa}_n^{(j)}=\frac{(2c_0+{\rev 8d})\log(n)}{\beta_n^j}$ for {\rev $j=1,\dots,d$,} and $c_0$ is the constant given in the statement of the proposition.
We introduce the following lemma, its proof is postponed until after the current proposition.
\begin{lemma}\label{l: majo Omega comp est cov}
	We have $\P(\Omega_n^c) \le 2 \frac{\lfloor log_2(n) \rfloor}{n^{{\rev 4d}+c_0}}$. 
\end{lemma}
We introduce a bounded modification of $g_{T^\prime}$ defined as
\begin{equation*}
{\rev 	\tilde{g}_{{\rev\bm{T}}^\prime}(\bm{X}_i, \mathcal{E}^{1,T^{\prime{(1)}}}_i, \dots,\mathcal{E}^{d,T^{\prime{(d)}}}_i )=
	\prod_{j=1}^d 
\left(\tronc{X_i^j}{T^{\prime{(j)}}}+\tronc{\mathcal{E}^{(j),T^{\prime{(j)}}}_i}{T^{\prime{(j)}}\tilde{\kappa}_n^{(j)}}\right).}
\end{equation*}
We have $\norm{\tilde{g}_{{\rev\bm{T}}^\prime}}_\infty \le 
{\rev \prod_{j=1}^d \big( T^{\prime(j)}(1+\tilde{\kappa}_n^{(j)})\big)} =:M_{{\rev\bm{T}}^\prime}$, and with computations analogous to the ones giving \eqref{eq: upper bound var est gamma}, we can prove
\begin{equation*}
	{\rev \mathop{var}\left(  \tilde{g}_{{\rev\bm{T}}^\prime}(\bm{X}_i, \mathcal{E}^{1,T^{\prime{(1)}}}_i, \dots,\mathcal{E}^{j,T^{\prime{(j)}}}_i )    \right) \le C \frac{\prod_{j=1}^d \abs{T^{\prime(j)}}^2 }{\prod_{j=1}^d\abs{\beta_n^j}^2}}=:v_{{\rev\bm{T}}^\prime},
\end{equation*}
for some universal constant $C$. 
We recall the Bernstein's inequality (see e.g. \cite{Boucheron_et_alBook}) : for $({G}_i)_{i=1,\dots,n}$ a iid sequence with
 $\norm{G_i}_\infty \le M$ and $\mathop{var}({G}_i)\le v$,
\begin{align*}
	\mathbb{P}\left(
	\abs{\frac{1}{n}\sum_{i=1}^n {G}_i - \E[{G}_i] }\ge t
	\right) \le 2 \exp\left(-\frac{n t^2}{2(v+Mt)}\right)
\le 2 \exp\left(-\frac{n t^2}{4v}\right)+2 \exp\left(-\frac{n t}{4M}\right).
\end{align*}
As on $\Omega_n$, the random variables ${g}_{{\rev\bm{T}}^\prime}(\bm{X}_i, \mathcal{E}^{1,T^{\prime{(1)}}}_i, {\rev \dots, \mathcal{E}^{d,T^{\prime{(d)}}}_i })$ and $\tilde{g}_{{\rev\bm{T}}^\prime}(\bm{X}_i, \mathcal{E}^{1,T^{\prime{(1)}}}_i, {\rev \dots,\mathcal{E}^{d,T^{\prime{(d)}}}_i })$ are almost surely equal,
we deduce for $t \ge 0$,
\begin{multline} \label{eq : Berstein applique gamma T}
	\P \left(  \left\{\abs*{\hat{\gamma}^{({\rev\bm{T}}^\prime)}_n - \E[\hat{\gamma}^{({\rev\bm{T}}^\prime)}_n ] }^2 
	-\frac{1}{16} \mathbb{V}_{{\rev\bm{T}}^\prime} \ge t ; ~ \Omega_n \right\}
	\right) \\
	 =	\P \left(  \left\{ \abs{n^{-1} \sum_{i=1}^n \tilde{g}_{{\rev\bm{T}}^\prime}(\bm{X}_i, \mathcal{E}^{1,T^{\prime{(1)}}}_i,
	 	{\rev \dots, \mathcal{E}^{d,T^{\prime{(d)}}}_i} ) - \E[\tilde{g}_{T^\prime}(\bm{X}_i, \mathcal{E}^{1,T^{\prime{(1)}}}_i,
	 	{\rev \dots, \mathcal{E}^{d,T^{\prime{(d)}}}_i} )]}  \ge  \sqrt{\frac{1}{16} \mathbb{V}_{T^\prime}+t}
	;~ \Omega_n\right\}\right)
	 \\ \le 2 \exp\left(-\frac{n \left( \frac{1}{16} \mathbb{V}_{{\rev\bm{T}}^\prime}+t\right) }{4v_{{\rev\bm{T}}^\prime}}\right)+2 \exp\left(-\frac{n \sqrt{\frac{1}{16} \mathbb{V}_{{\rev\bm{T}}^\prime}+t}}{4M_{{\rev\bm{T}}^\prime}}\right)
	  \\ \le 2 \exp\left(-\frac{n \mathbb{V}_{{\rev\bm{T}}^\prime}}{64 v_{{\rev\bm{T}}^\prime}} \right)	  
	  \exp \left( -\frac{nt}{4v_{{\rev\bm{T}}^\prime}}\right)
	  +2 \exp\left(\frac{-n\sqrt{\mathbb{V}_{{\rev\bm{T}}^\prime}}}{32 M_{{\rev\bm{T}}^\prime} }  \right)
	  \exp\left(\frac{-n\sqrt{t}}{8M_{{\rev\bm{T}}^\prime}} \right).
\end{multline}
By \eqref{eq: def penalisation est cov}, we have $\frac{n\mathbb{V}_{{\rev\bm{T}}^\prime}}{v_{{\rev\bm{T}}^\prime}}=\frac{\kappa_n}{C}=\frac{c_0 \ln(n)}{C}$ 
for some universal constant $C$, and 
{\rev $\frac{n\sqrt{\mathbb{V}_{{\rev\bm{T}}^\prime}}}{M_{{\rev\bm{T}}^\prime}}=\frac{\sqrt{n}\sqrt{\kappa_n}}{\prod_{j=1}^d\beta_n^j \prod_{j=1}^d(1+\tilde{k}_n^{(j)})}
$} 
is equal to $\frac{ \sqrt{n} \sqrt{c_0 \log(n)} \lfloor\log_2(n)\rfloor^{\rev 2d} }{{\rev \prod_{j=1}^d \alpha_j} {\rev \prod_{j=1}^d (1+(2c_0+8d)\frac{\log(n)\lfloor\log_2(n)\rfloor}{\alpha_j})}} \ge \frac{C'{\modar{\sqrt{n}}}}{(c_0+{\rev 4d})^{\rev 2d-1/2} \log(n)^{\rev 2d -1/2}}$ for some constant $C'$. Hence,
\begin{multline*}
\P \left(  \left\{\abs*{\hat{\gamma}^{({\rev\bm{T}}^\prime)}_n - \E[\hat{\gamma}^{({\rev\bm{T}}^\prime)}_n ] }^2 
-\frac{1}{16} \mathbb{V}_{{\rev\bm{T}}^\prime} \ge t ; ~ \Omega_n \right\}
\right) \\ \le 2 \exp \left( -\frac{c_0 \ln(n)}{C64}\right)	  
\exp \left( -\frac{nt}{4v_{{\rev\bm{T}}^\prime}}\right)
+2 \exp\left( -\frac{C'{\modar{\sqrt{n}}}}{{\rev (c_0+2d)^{2d-1/2} \log(n)^{2d-1/2}}}   \right) 
\exp\left(\frac{-n\sqrt{t}}{8M_{{\rev\bm{T}}^\prime}} \right).
\end{multline*}
By choosing $c_0$ large enough, we have
\begin{equation*}
	\P \left(  \left\{\abs*{\hat{\gamma}^{({\rev\bm{T}}^\prime)}_n - \E[\hat{\gamma}^{({\rev\bm{T}}^\prime)}_n ] }^2 
	-\frac{1}{16} \mathbb{V}_{{\rev\bm{T}}^\prime} \ge t ; ~ \Omega_n \right\}
	\right) \\ \le C n^{-\overline{c}} 	\left( \exp \left( -\frac{nt}{4v_{{\rev\bm{T}}^\prime}}\right) + \exp\left(\frac{-n\sqrt{t}}{8M_{{\rev\bm{T}}^\prime}} \right)\right),
\end{equation*}
where $\overline{c}>0$ is any arbitrary constant. Since the previous control is valid for any ${\rev\bm{T}}^\prime \in \mathcal{T}$, we deduce that,
\begin{multline*}
	\P \left(  \left\{ \sup_{{\rev\bm{T}}^\prime\in\mathcal{T}} \left( \abs*{\hat{\gamma}^{({\rev\bm{T}}^\prime)}_n - \E[\hat{\gamma}^{({\rev\bm{T}}^\prime)}_n ] }^2 
	-\frac{1}{16} \mathbb{V}_{{\rev\bm{T}}^\prime} \right)_+ \ge t ; ~ \Omega_n \right\}
	\right) \\ \le Cn^{-\overline{c}} \sum_{{\rev\bm{T}}^\prime \in \mathcal{T}}  	\left( \exp \left( -\frac{nt}{4v_{{\rev\bm{T}}^\prime}}\right) + \exp\left(\frac{-n\sqrt{t}}{8M_{{\rev\bm{T}}^\prime}} \right)\right).
\end{multline*}
Integrating with respect to $t$ on $[0,\infty)$, we get
\begin{equation*}
	\E \left[ 
	 \sup_{{\rev\bm{T}}^\prime\in\mathcal{T}} \left( \abs*{\hat{\gamma}^{(T^\prime)}_n - \E[\hat{\gamma}^{({\rev\bm{T}}^\prime)}_n ] }^2 
	-\frac{1}{16} \mathbb{V}_{{\rev\bm{T}}^\prime} \right)_+  
	\one_{\Omega_n}
	\right] \\ \le C n^{-\overline{c}}  \sum_{{\rev\bm{T}}^\prime \in \mathcal{T}} 	\left( 
	\frac{v_{{\rev\bm{T}}^\prime}}{n}+ \frac{M_{{\rev\bm{T}}^\prime}^2}{n^2}\right).
\end{equation*}
Using \eqref{eq: def ens T possible}, 
$v_{{\rev\bm{T}}^\prime}= C{\rev  \frac{\prod_{j=1}^d T^{\prime(j)}}{\prod_{j=1}^d\abs{\beta_n^j}^2}} \le C \frac{n^{{\rev2d}}{\modar {\log_2(n)}^{{\rev2d}}}}{{\rev \prod_{j=1}^d \alpha_j^2}}$ and $M_{{\rev\bm{T}}^\prime}= {\rev \prod_{j=1}^d \big(T^{\prime(j)}(1+\tilde{\kappa}_n^{(j)})\big)}$ is upper bounded by
{\rev $ n^{d} \prod_{j=1}^d(1+\frac{C\log(n)^2}{\alpha_j})$,} we deduce
\begin{equation} \label{eq: borne sup T gamma V part 1}
	\E \left[ 
	\sup_{{\rev\bm{T}}^\prime\in\mathcal{T}} \left( \abs*{\hat{\gamma}^{({\rev\bm{T}}^\prime)}_n - \E[\hat{\gamma}^{({\rev\bm{T}}^\prime)}_n ] }^2 
	-\frac{1}{16} \mathbb{V}_{{\rev\bm{T}}^\prime} \right)_+  
	\one_{\Omega_n}
	\right] \le C n^{-\overline{c}}  {\rev \frac{n^{2d-1} \log(n)^{2d+1}}{\prod_{j=1}^d\alpha^2_j}}.
\end{equation}

We now study the contribution coming from the event $\Omega_n^c$. 
We have,  using the simple inequality $(a-b)_+ \le (a)_+$ 
and Jensen's inequality
\begin{equation*}
	\E \left[ 
\sup_{{\rev\bm{T}}^\prime\in\mathcal{T}} \left( \abs*{\hat{\gamma}^{({\rev\bm{T}}^\prime)}_n - \E[\hat{\gamma}^{({\rev\bm{T}}^\prime)}_n ] }^2 
-\frac{1}{16} \mathbb{V}_{{\rev\bm{T}}^\prime} \right)_+  
\one_{\Omega_n^c}\right] \le2 \sum_{{\rev\bm{T}}^\prime \in \mathcal{T}}
	\E \left[ 
\abs*{\hat{\gamma}^{({\rev\bm{T}}^\prime)}_n}^2 \one_{\Omega_n^c}\right].
\end{equation*}
By using Cauchy--Schwarz's inequality, it comes
\begin{align*}
	\E \left[ 
	\sup_{T^\prime\in\mathcal{T}} \left( \abs*{\hat{\gamma}^{(T^\prime)}_n - \E[\hat{\gamma}^{(T^\prime)}_n ] }^2 
	-\frac{1}{16} \mathbb{V}_{T^\prime} \right)_+  
	\one_{\Omega_n^c}\right] &\le 2\sum_{T^\prime \in \mathcal{T}}
	\E \left[ 
	\abs*{\hat{\gamma}^{(T^\prime)}_n}^4 \right]^{1/2} \P(\Omega_n^c)^{1/2}
	\\&\le 2
	\sum_{T^\prime \in \mathcal{T}}
	 \left[ n^{-1}
	\sum_{i=1}^n \E({\rev \abs{\prod_{j=1}^dZ^{j,T^{\prime(j)}}}^4}) \right]^{1/2} \P(\Omega_n^c)^{1/2},
\end{align*}
where we used again Jensen's inequality. Since $\E(\abs{Z^{l,T^{\prime(l)}}})^q \le C_q \frac{\abs{T^{\prime(l)}}^q}{\beta_n^q}$ for all $q\ge 1$ and $l\in\{1,\dots,d\}$, we deduce,

\begin{align} \nonumber
	\E \left[ 
	\sup_{{\rev\bm{T}}^\prime\in\mathcal{T}} 
	\left( \abs*{\hat{\gamma}^{({\rev\bm{T}}^\prime)}_n - \E[\hat{\gamma}^{({\rev\bm{T}}^\prime)}_n ] }^2 
	-\frac{1}{16} \mathbb{V}_{{\rev\bm{T}}^\prime} \right)_+  
	\one_{\Omega_n^c}\right] &\le
	C  {\rev \Big( \sum_{{\rev\bm{T}}^\prime \in \mathcal{T}}
 \frac{\prod_{j=1}^d \abs{T^{\prime(j)}}^2 }{ \prod_{j=1}^d \beta_n^j } \Big)} \P(\Omega_n^c)^{1/2}
\\ & \le \nonumber
C \mathop{card}(\mathcal{T}) \times \frac{n^{\rev 2d}}{{\rev \prod_{j=1}^d \beta_n^j}} \times  \frac{\sqrt{\log(n)}}{n^{{\rev 2d}+c_0/2}}
\\& \le \label{eq: borne sup T gamma V part 2}
C  \frac{\log(n)^{{\rev 3d+1/2}}}{{\rev \prod_{j=1}^d \alpha^2_j}~n^{c_0/2}}
\end{align}
	by Lemma \ref{l: majo Omega comp est cov}. Collecting \eqref{eq: borne sup T gamma V part 1} and 
\eqref{eq: borne sup T gamma V part 2}, with the fact that $c_0$ can be chosen arbitrarily large, and $\alpha\le 1$, we have
\begin{equation*}
	\E \left[ 
	\sup_{{\rev\bm{T}}^\prime\in\mathcal{T}} \left( \abs*{\hat{\gamma}^{({\rev\bm{T}}^\prime)}_n - \E[\hat{\gamma}^{({\rev\bm{T}}^\prime)}_n ] }^2 
	-\frac{1}{16} \mathbb{V}_{T^\prime} \right)_+  \right] 
	\le \frac{C}{{\rev \prod_{j=1}^d\alpha^2_j }~ n^{\overline{c}}}.
\end{equation*}
		This is exactly \eqref{eq: control Bl preuve control D} with $l=2$. The proof of $\eqref{eq: control Bl preuve control D}$ with $l=1$ is obtained similarly, remarking that the application of the Bernstein's inequality, with the same constants $M_{T^\prime}$, $v_{T^\prime}$ still yields to the upper bound \eqref{eq : Berstein applique gamma T} for
				$	\P \left(  \left\{\abs*{\hat{\gamma}^{({\rev\bm{T}},{\rev\bm{T}}^\prime)}_n - \E[\hat{\gamma}^{({\rev\bm{T}},{\rev\bm{T}}^\prime)}_n ] }^2 
		-\frac{1}{16} \mathbb{V}_{{\rev\bm{T}}^\prime} \ge t ; ~ \Omega_n \right\}
		\right) $.	
\end{proof}

\begin{proof}[Proof of Lemma \ref{l: majo Omega comp est cov}]
The set  $\Omega_n^c$ is included in 
\begin{equation*}
\bigcup_{j=1}^{\rev d} \bigcup_{T^{(j)}\in\mathcal{T}^{(j)}} \{ \abs{\mathcal{E}^{(j),T^{(j)}}} \ge T^{(j)} \tilde{\kappa}_n^{(j)}\}.	
\end{equation*} 
But $\mathcal{E}^{(j),T^{(j)}}/T^{(j)}$ is distributed as a Laplace $\frac{2}{\beta_n^j}\times\mathcal{L}(1)$ variable, and we deduce 
\begin{multline*}
	\P\left(\Omega_n^c\right) \le {\rev \sum_{j=1}^d \mathop{card}\left(\mathcal{T}^{(j)}\right) \P\left( \frac{2}{\beta_n^j}\abs*{\mathcal{L}(1)} \ge \tilde{\kappa}_n^{(j)}\right)} 
	 \le 2 \lfloor \log_2(n) \rfloor  e^{-\tilde{\kappa}_n\beta_n^{{\rev(j)}}/2} \\= 2 \lfloor \log_2(n) \rfloor  e^{-\ln(n)(2c_0+{\rev 8d})/2}=2 \lfloor \log_2(n) \rfloor  n^{-{\rev 4d}-c_0}, 
\end{multline*}   
as we wanted. 
\end{proof}

\begin{proposition}\label{prop: est gamma adaptive sur classe}
	Assume that {\rev $k_1^{-1}+\dots+k_d^{-1}<1$,} and $\beta_n^j=\frac{\alpha_j}{\lfloor \log_2(n) \rfloor}$, $\kappa_n=c_0 \log (n)$ for some $c_0>0$. {\modar If $c_0$ is large enough, there} exist $c>0$, $\overline{c}_0>0$, such that
	\begin{equation*}
		\E\left[(\hat{\gamma}^{\widehat{T}}_n - \gamma )^2\right]
		\le c \inf_{{\rev\bm{T}}\in \mathcal{T}} [\overline{\bm{b}}_{\rev\bm{T}}^2 + \mathbb{V}_{\rev\bm{T}} ]
		+ \frac{c}{{\rev \prod_{j=1}^d\alpha^2_j ~ n^{\overline{c}_0}}},
	\end{equation*}
	for all $n \ge 1$, $\alpha_j\le 1$, $(n{\rev \prod_{j=1}^d\alpha^2_j})/(\log(n))^{\rev 2d+1} \ge 1 $. Moreover, the constant $\overline{c}_0$ can be chosen arbitrarily large by choosing $c_0$ large enough.
\end{proposition}
\begin{proof} Let ${\rev\bm{T}}\in\mathcal{T}$, we have
\begin{equation*}
\abs{\hat{\gamma}^{(\widehat{{\rev\bm{T}}})}-\gamma}\le
\abs{\hat{\gamma}^{(\widehat{{\rev\bm{T}}})}-\hat{\gamma}^{(\widehat{{\rev\bm{T}}},{\rev\bm{T}})}}
+
\abs{\hat{\gamma}^{(\widehat{{\rev\bm{T}}},{\rev\bm{T}})}-\hat{\gamma}^{({\rev\bm{T}})}}+
\abs{\hat{\gamma}^{({\rev\bm{T}})}-\gamma}.
\end{equation*}
From the definition \eqref{eq: def D est cov}, we have 
$\abs{\hat{\gamma}^{({\rev\bm{T}})}-\gamma} \le  \abs{\hat{\gamma}^{({\rev\bm{T}})}-\E[\hat{\gamma}^{({\rev\bm{T}})}]} +
\abs{\E[\hat{\gamma}^{({\rev\bm{T}})}-\gamma]}  \le \abs{\hat{\gamma}^{({\rev\bm{T}})}-\E[\hat{\gamma}^{({\rev\bm{T}})}]}+\mathbb{D}_{\rev\bm{T}}$, and it follows
\begin{equation*}
	\abs{\hat{\gamma}^{(\widehat{{\rev\bm{T}}})}-\gamma}\le
	\abs{\hat{\gamma}^{(\widehat{{\rev\bm{T}}})}-\hat{\gamma}^{(\widehat{{\rev\bm{T}}},{\rev\bm{T}})}}
	+
	\abs{\hat{\gamma}^{(\widehat{{\rev\bm{T}}},{\rev\bm{T}})}-\hat{\gamma}^{({\rev\bm{T}})}}+
\abs{\hat{\gamma}^{({\rev\bm{T}})}-\E[\hat{\gamma}^{({\rev\bm{T}})}]}+\mathbb{D}_{\rev\bm{T}}.
\end{equation*}	
By \eqref{eq: def est biais adaptatif est cov}, we have 
$\abs{\hat{\gamma}^{(\widehat{{\rev\bm{T}}},{\rev\bm{T}})}-\hat{\gamma}^{({\rev\bm{T}})}}^2 \le \mathbb{B}_{\widehat{{\rev\bm{T}}}}+\mathbb{V}_{\widehat{{\rev\bm{T}}}}$,
and recalling $\hat{\gamma}^{(\widehat{{\rev\bm{T}}},{\rev\bm{T}})}=\hat{\gamma}^{({\rev\bm{T}},\widehat{{\rev\bm{T}}})}$, we also get
$\abs{\hat{\gamma}^{(\widehat{{\rev\bm{T}}})}-\hat{\gamma}^{(\widehat{{\rev\bm{T}}},{\rev\bm{T}})}}^2 \le \mathbb{B}_{{{\rev\bm{T}}}}+\mathbb{V}_{{{\rev\bm{T}}}}$. Thus, 
\begin{equation*}
	\abs{\hat{\gamma}^{(\widehat{\rev\bm{T}})}-\gamma}^2\le16 \left[ 
\mathbb{B}_{\widehat{{\rev\bm{T}}}}+\mathbb{V}_{\widehat{{\rev\bm{T}}}}
	+
	\mathbb{B}_{{{\rev\bm{T}}}}+\mathbb{V}_{{{\rev\bm{T}}}}+
	\abs{\hat{\gamma}^{({\rev\bm{T}})}-\E[\hat{\gamma}^{({\rev\bm{T}})}]}^2+\mathbb{D}_{{\rev\bm{T}}}^{{\revd2}}\right].
\end{equation*}		
Using \eqref{eq: def That est cov}, we deduce
\begin{equation*}
	\abs{\hat{\gamma}^{(\widehat{{\rev\bm{T}}})}-\gamma}^2\le16 \left[ 
	2\mathbb{B}_{{{\rev\bm{T}}}}+2\mathbb{V}_{{{\rev\bm{T}}}}+
	\abs{\hat{\gamma}^{({\rev\bm{T}})}-\E[\hat{\gamma}^{({\rev\bm{T}})}]}^2+\mathbb{D}_{{\rev\bm{T}}}^{{\revd2}}\right].
\end{equation*}		
From the study of the variance of $\hat{\gamma}^{({\rev\bm{T}})}$ as in \eqref{eq: upper bound var est gamma}, we have $\E[\abs{\hat{\gamma}^{({\rev\bm{T}})}-\E[\hat{\gamma}^{({\rev\bm{T}})}]}^2] \le
 C n^{-1} {\rev  \frac{\prod_{j=1}^d\abs{T^{(j)}}^2}{\prod_{j=1}^d\abs{\beta_n^j}^2}}$, which is smaller than $\mathbb{V}_T$, if $c_0$ in \eqref{eq: def penalisation est cov} is large enough. Thus, we deduce
 \begin{equation*}
 	\E\left[\abs{\hat{\gamma}^{(\widehat{{\rev\bm{T}}})}-\gamma}^2\right]\le16 \left[ 
 	2\E[\mathbb{B}_{{{\rev\bm{T}}}}]+3\mathbb{V}_{{{\rev\bm{T}}}}+\mathbb{D}_{{\rev\bm{T}}}^{{\revd2}}\right].
 \end{equation*}		
Now, Lemma  \ref{l: control DT est cov} and Proposition \ref{P: maj BBT est cov} yield to
  \begin{equation*}
 	\E\left[\abs{\hat{\gamma}^{(\widehat{{\rev\bm{T}}})}-\gamma}^2\right]\le c \left[ 
 \overline{\bm{b}}_{\rev\bm{T}}^2 + \mathbb{V}_{\rev\bm{T}}\right] + \frac{c}{{\rev \prod_{j=1}^d \alpha^2_j}~ n^{\overline{c}_0}},
 \end{equation*}		
for any ${\rev\bm{T}}\in \mathcal{T}$. This proves the proposition.
\end{proof}
We end this section by providing the proof of Theorem \ref{th: est gamma adaptive}
\begin{proof}[Proof of Theorem \ref{th: est gamma adaptive}]
By {\rev Proposition  \ref{prop: est gamma adaptive sur classe},} it is sufficient to evaluate $\inf_{{\rev\bm{T}} \in \mathcal{T}} \left[ 
\overline{\bm{b}}_{\rev\bm{T}}^2 + \mathbb{V}_{\rev\bm{T}}\right]$ which is, up to a constant, the infimum over ${\rev\bm{T}}\in\mathcal{T}$ of
\begin{align} \nonumber
	{\rev \sum_{j=1}^d	(T^{(j)})^{-2k_j(1-\frac{d}{\overline{k}})} 
	+n^{-1} \kappa_n\frac{\prod_{j=1}^d\abs{T^{(j)}}^2 }{\prod_{j=1}^d\abs{\beta_n^j}^2}}
	\\ \label{eq: a optimiser dans Tcal est cov} \le 
	{\rev \sum_{j=1}^d	(T^{(j)})^{-2k_j(1-\frac{d}{\overline{k}})} 
	+Cn^{-1}  \log(n)^{2d+1} \frac{\prod_{j=1}^d\abs{T^{(j)}}^2 }{\prod_{j=1}^d\alpha^2_j}},	
\end{align}
	for some $C>0$. If we set $T^{*(j)}=\left(\frac{n{\rev \prod_{j=1}^d\alpha_j^2}}{\log(n)^{\rev 2d+1}}\right)^{1/(2k_{\rev j})}$ for {\rev 
		$j\in\{1,\dots,d\}$,} the above quantity is smaller than some constant time
	$$
	\left(\frac{\log(n)^{\rev 2d+1}}{n{\rev \prod_{j=1}^d\alpha_j^2}}\right)^{\frac{\overline{k}-{\rev d}}{\overline{k}}},
	$$
	which is the expected rate. It remains to check that the same rate can be obtained by restricting $T$ in $\mathcal{T}$. As $(n{\rev \prod_{j=1}^d\alpha_j^2})/(\log(n))^{\rev 2d+1} \ge 1 $, and $\alpha_j \le 1$, we see that for $n\ge 3$, {\rev $T^{*(j)}\in [1,n]$ 
		for $j \in \{1,\dots,d\}$}. 
	By the definition \eqref{eq: def ens T possible} of $\mathcal{T}^{(j)}$, we see that $\mathcal{T}^{(j)}$ is a grid of $[1,n]$ such that for any $t^*\in [1,n]$, there exists $t^{(j)} \in \mathcal{T}^{(j)}$ with $\frac{1}{2}t^* \le t^{(j)}\le 2t^*$. Hence, by replacing the $T^{*(j)}$ by their closest values in $\mathcal{T}^{j}$ we only {\modar increase} the value of \eqref{eq: a optimiser dans Tcal est cov} by a multiplicative constant. We deduce
	$$
\inf_{T \in \mathcal{{\rev\bm{T}}}} \left[ 
\overline{\bm{b}}_{\rev\bm{T}}^2 + \mathbb{V}_{\rev\bm{T}}\right] \le c\left(\frac{\log(n)^{\rev 2d+1}}{n{\rev \prod_{j=1}^d\alpha_j^2}}\right)^{\frac{\overline{k}-{\rev d}}{\overline{k}}},
	$$
	and Theorem \ref{th: est gamma adaptive} follows from {\rev  Proposition \ref{prop: est gamma adaptive sur classe}.}
\end{proof}

\subsection{Proof locally private density estimation}\label{s: Proof locally density}
This section is devoted to the proof of the results stated in Section \ref{s: density}, about the density estimation under $\alpha$-CLDP constraints.

{\rev 
	We start by proving that $Z_i^j$ defined according to 
\eqref{eq: def Zij} are $\alpha_j$ local differential private view of the observation $X_i^j$, as stated in Lemma \ref{l: privacy procedure}.
\begin{proof}[Proof of Lemma \ref{l: privacy procedure}]
		The density of $\mathcal{E}_i^j$ at the point $x\in \R$ is given by the value $\frac{1}{2 \kappa} \alpha_j h \exp(- \frac{1}{2 \kappa} \alpha_j h |x|)$. Then, the reverse triangle inequality and the fact the infinity norm of $K$ is bounded by $\kappa$ provide 
		\begin{align*}
			\sup_{z \in \mathcal{Z}} \frac{q^j(z |X_i^j = x)}{q^j(z |X_i^j = x')} & \le \sup_{z \in \mathcal{Z}} \exp \Big(- \frac{1}{2 \kappa} \alpha_j h \abs{z - \frac{1}{h} K(\frac{x - x_0^j}{h})} + \frac{1}{2 \kappa} \alpha_j h \abs{z - \frac{1}{h} K(\frac{x' - x_0^j}{h})} \Big) \\
			& \le \exp \Big(\frac{1}{2 \kappa} \alpha_j h \frac{1}{h} \abs{K(\frac{x - x_0^j}{h}) - K(\frac{x' - x_0^j}{h})} \Big) \\
			& \le \exp(\alpha_j).
		\end{align*}
	\end{proof}
}

\begin{proof}[Proof of Theorem \ref{th: upper density}.]
The proof is based on the usual bias-variance decomposition. We have indeed
$\E[|\hat{\pi}^Z_h(\bm{x_0}) - \pi(\bm{x_0})|^2] = |\E[\hat{\pi}^Z_h(\bm{x_0})] - \pi(\bm{x_0})|^2 + var(\hat{\pi}^Z_h(\bm{x_0})).$
One can easily bound the bias part (see for example Proposition 1.2 in \cite{Tsy}), obtaining for any $\bm{x_0} \in \R^d$ and any $h > 0$
\begin{equation}{\label{eq: bias density}}
|\E[\hat{\pi}^Z_h(\bm{x_0})] - \pi(\bm{x_0})|^2 \le c h^{2 \beta},
\end{equation}
for some $c > 0$. Regarding the variance of $\hat{\pi}^Z_h(\bm{x_0})$, we use its explicit form and the fact that the vectors $\bm{X}_1, \dots , \bm{X}_n$ and the Laplace random variables are independent to get
\begin{align*}
var(\hat{\pi}^Z_h(\bm{x_0}))
& = \frac{1}{n^2} \sum_{i = 1}^n var \Big( \prod_{j = 1}^d ( \frac{1}{h} K(\frac{X_i^j - x_0^j}{h}) + \mathcal{E}_i^j ) \Big) \\
& = \frac{1}{n^2} \sum_{i = 1}^n var \Big( \sum_{I_k} \prod_{j \in I_k}  \frac{1}{h} K(\frac{X_i^j - x_0^j}{h}) \prod_{j \in (I_k)^c} \mathcal{E}_i^j  \Big) 
\end{align*}
where $I_k$ is a set of index such that $|I_k| = k$, for $k \in \{1, \dots , d \}$.
Then, it is well known that $var(\frac{1}{h^d} \prod_{j=1}^d K(\frac{X_i^j - x_0^j}{h})) \le \frac{c}{h^d}$
for some positive $c$ (one can easily see that by adapting Proposition 1.1 in \cite{Tsy} to the multidimensional context). 
Moreover, by construction, $\mathcal{E}_i^j$ are iid $\sim \mathcal{L}(\frac{2 \kappa}{\alpha_j h})$, which guarantees that
$var(\prod_{j=1}^d \mathcal{E}_i^j) \le \frac{c}{ \prod_{j = 1}^d (\alpha_j h)^{2}}.$
One can then readily check that, for any set of index $I_k$ such that $|I_k| = k$, it is 
$$var(\prod_{j \in I_k} \frac{1}{h} K(\frac{X_i^j - x_0^j}{h}) \prod_{j \in (I_k)^c}\mathcal{E}_i^j) \le \frac{c}{h^k} \frac{1}{ h^{2(d - k)}} \prod_{j \in (I_k)^c} \frac{1}{\alpha_{j}^2}.$$
 It implies 
\begin{align}{\label{eq: variance with sum}}
var(\hat{\pi}^Z_h(\bm{x_0})) \le \frac{1}{n} \sum_{k =0}^d \frac{c}{h^k} \frac{1}{ h^{2(d - k)}} \sum_{I_k} \prod_{j \in (I_k)^c} \frac{1}{\alpha_{j}^2} = \frac{c}{n h^{2d} \prod_{j = 1}^d \alpha_j^2} \sum_{k =0}^d  \sum_{I_k}\prod_{j \in I_k}(h \alpha_j^2).
\end{align}
We now recall that $h$ is a bandwidth we have assumed being smaller than $1$. We have also required $\alpha_j \le 1$ for any $j \in \{1, \dots , d \}$, which yields $h \alpha_j^2 < 1$ for any $j \in \{1, \dots, d \}$. Then, the largest term in the sum above is the one for which $k = 0$, which means that $I_k = \emptyset$. We derive 
\begin{equation}{\label{eq: var}}
var(\hat{\pi}^Z_h(\bm{x_0})) \le \frac{c}{n h^{2d} \prod_{j = 1}^d \alpha_j^2}.
\end{equation}
We then look for the choice of $h$ that realizes the trade-off between the bound on the variance here above and the bound on the bias term gathered in \eqref{eq: bias density}. This is achieved by the rate optimal bandwidth 
$h^* := (\frac{1}{n \, \prod_{j = 1}^d \alpha_j^{2}})^{\frac{1}{2(\beta + d)}}.$
We remark that, as by hypothesis it is $n \prod_{j = 1}^d \alpha_j^{2} \rightarrow \infty$ for $n \rightarrow \infty$, it clearly follows $h^* = h^*_n \rightarrow 0$ for $n \rightarrow \infty$. 
Replacing it in \eqref{eq: bias density} and \eqref{eq: var} we obtain 
$\E[|\hat{\pi}^Z_h(\bm{x_0}) - \pi(\bm{x_0})|^2] \le c \Big( \frac{1}{n \, \prod_{j = 1}^d \alpha_j^2}\Big)^{\frac{\beta}{\beta + d}},$
 which concludes the proof. 
\end{proof}

\begin{proof}[Proof of Theorem \ref{th: upper density same alpha}]
We observe that, for $\alpha_1 = \dots = \alpha_d = \alpha$, \eqref{eq: variance with sum} translates to $var(\hat{\pi}^Z_h(\bm{x_0})) \le \frac{c}{n (h \alpha)^{2d}} \sum_{k =0}^d(h \alpha^2)^k$. Then, we consider two different cases. \\

\noindent $\bullet$ If $\alpha \ge n^{\frac{1}{2(2 \beta + d)}}$, we choose the optimal bandwidth as in the privacy free-context 
$h^* := (\frac{1}{n})^{\frac{1}{2\beta + d}}.$
We observe we have in this case 
$h \alpha^2 \ge n^{- \frac{1}{2 \beta + d} + \frac{1}{2 \beta + d}} = 1.$
Hence, in \eqref{eq: variance with sum}, the worst term is the one for which $k= d$. It implies the variance of $\hat{\pi}^Z_h(\bm{x_0})$ is bounded by 
\begin{align*}
\frac{c}{n (h \alpha)^{2d}} (h \alpha^2)^d = \frac{c}{n h^d} = c (\frac{1}{n})^{\frac{2 \beta}{2 \beta + d}}.
\end{align*}
We observe that the bandwidth $h^*$ is the one that achieves the balance in the decomposition bias-variance, as it is also $(h^*)^{2 \beta} = (\frac{1}{n})^{\frac{2 \beta}{2 \beta + d}}$. The proof in the first case is then concluded. \\

\noindent $\bullet$ Consider now what happens for $\alpha < n^{\frac{1}{2(2 \beta + d)}}$. In this case, we will see that the optimal choice in terms of convergence rate will consist in taking 
$h^* := (\frac{1}{n \, \alpha^{2d}})^{\frac{1}{2(\beta + d)}}.$
Remark that we have assumed $n \alpha^{2d} \rightarrow \infty$ for $n \rightarrow \infty$ so that $h^*\rightarrow 0$, for $n$ going to $\infty$. 
We observe in this context it is 
\begin{align*}
h^* \alpha^2 = (\frac{1}{n \, \alpha^{2d}})^{\frac{1}{2(\beta + d)}} \alpha^2 = (\frac{1}{n})^{\frac{1}{2(\beta + d)}}\alpha^{\frac{2 \beta + d}{\beta + d}} \le (\frac{1}{n})^{\frac{1}{2(\beta + d)}} n^{\frac{1}{2(\beta + d)}} = 1.
\end{align*}
Thus, the largest term in the sum in \eqref{eq: variance with sum} is for $k = 0$, which implies 
\begin{align*}
    var(\hat{\pi}^Z_h(\bm{x_0})) \le \frac{c}{n (h \alpha)^{2d}} = \frac{c}{n \alpha^{2d}}(n \alpha^{2d})^{\frac{d}{\beta + d}} = c (\frac{1}{n \alpha^{2d}})^{\frac{\beta}{\beta + d}}.
\end{align*}
The bandwidth $h^*$ realizes the balance between the variance and the bias, as $(h^*)^{2 \beta} = (\frac{1}{n \alpha^{2d}})^{\frac{\beta}{\beta + d}}$. The proof is then complete. 

\end{proof}

\subsubsection{Proof adaptive procedure}
We start by proving that $Z_i^j$ defined according to \eqref{eq: def Zi adpat} are $\alpha_j$ local differential private view of the observation $X_i^j$, as stated in Lemma \ref{l: privacy adaptive density}.

\begin{proof}[Proof of Lemma \ref{l: privacy adaptive density}] 
From the definition of Laplace, using the independence of the variables $(Z_i^{j,h})_{h \in H_n}$ and denoting as $q^j((z^{j,h})_{h \in H_n} | X_i^j = x)$ the density of the law of $Z_i^{j,h}$ conditional to $X_i^j = x \in \R$ we obtain 
\begin{align*}
\frac{q^j((z^{j,h})_{h \in H_n} | X_i^j = x)}{q^j((z^{j,h})_{h \in H_n} | X_i^j = x')} & = \frac{\prod_{h \in H_n} \exp(|z^{j,h} - \frac{1}{h} K(\frac{x - x_0^j}{h})| \frac{\beta_n^j\, h}{2 \kappa})}{\prod_{h \in H_n} \exp(|z^{j,h} - \frac{1}{h} K(\frac{x' - x_0^j}{h})| \frac{\beta_n^j \, h}{2 \kappa})} \\
& \le \prod_{h \in H_n} \exp(| K(\frac{x - x_0^j}{h}) - K(\frac{x' - x_0^j}{h})| \frac{\beta_n^j}{2 \kappa}) \\
& \le \prod_{h \in H_n} \exp(\beta_n^j) \\
& = \exp(Card(H_n)\beta_n^j ) \le \exp(\alpha_j),
\end{align*}
being the last a consequence of how we have chosen $\beta_n^j$.
\end{proof}

\noindent Before proving the main theorem, let us introduce the notation $\pi_h^z(\bm{x_0})$ for $\E[\hat{\pi}_h^z(\bm{x_0})]$ and 
\begin{align}{\label{eq: def Dh}}
\mathbb{D}_h &  := \Big( \sup_{\eta \in H_n} |\E[\hat{\pi}_{\eta \land h }^z(\bm{x_0}) - \hat{\pi}_h^z(\bm{x_0})]| \Big) \lor |\E[\hat{\pi}_h^z(\bm{x_0})] - \pi(\bm{x_0})| \\
& = \Big( \sup_{\eta < h} |\E[\hat{\pi}_{\eta }^z(\bm{x_0}) - \hat{\pi}_h^z(\bm{x_0})]| \Big) \lor |\E[\hat{\pi}_h^z(\bm{x_0})] - \pi(\bm{x_0})| \nonumber .
\end{align}
As for \eqref{eq: bias density}, with classical computations as in Proposition 1.2 of \cite{Tsy} it readily follows, for some $c>0$, 
\begin{equation}{\label{eq: bound Dh}}
    \mathbb{D}_h \le c h^{\beta}.
\end{equation}

\noindent The proof of Theorem \ref{th: adaptive density} heavily relies on the following proposition.
\begin{proposition}{\label{prop: adaptive density}}
 Assume that $\pi \in \mathcal{H}(\beta, \mathcal{L})$ for some $\beta$ and $\mathcal{L} \ge 1$. Moreover, $\beta_n^j = \frac{\alpha_j}{ \lfloor \log_2 n \rfloor}$ for any $j \in \{1, \dots , d \}$ and $a_n = {c_0} \log n$ for some ${c_0} > 0$. {\modar If $c_0$ is large enough, there} exist $c> 0$ and $\bar{c} > 0$ such that 
$$\E [(\hat{\pi}_{\hat{h}}^z(\bm{x_0}) - \pi (\bm{x_0}))^2] \le c \inf_{h \in H_n} (\mathbb{V}_h + \mathbb{D}_h^2) + \frac{c}{n^{\bar{c}}\prod_{j=1}^d \alpha_j^2}$$
for all $n \ge 1$, $\alpha_j \le 1$ and $\frac{n \prod_{j = 1}^d \alpha_j^2}{\log n^{1 + 2d}} \ge 1$. Moreover, the constant $\bar{c}$ can be chosen arbitrarily large, taking the constant $c_0$ large enough.   
\end{proposition}

\begin{proof}[Proof of Proposition \ref{prop: adaptive density}]
Let $h \in H_n$. It is 
$$|\hat{\pi}_{\hat{h}}^z(\bm{x_0}) - \pi (\bm{x_0})| \le |\hat{\pi}_{\hat{h}}^z(\bm{x_0}) - \hat{\pi}_{\hat{h},h}^z(\bm{x_0})| + |\hat{\pi}_{\hat{h},h}^z(\bm{x_0}) - \hat{\pi}_{h}^z(\bm{x_0})| + |\hat{\pi}_{h}^z(\bm{x_0}) - \pi (\bm{x_0})|.$$
Following the same computations as in the proof of {\modar Proposition \ref{prop: est gamma adaptive sur classe}}
 it is then easy to check that 
$$\E[|\hat{\pi}_{\hat{h}}^z(\bm{x_0}) - \pi (\bm{x_0})|^2] \le c \Big( \E[\mathbb{B}_h] + \mathbb{V}_h + \mathbb{D}_h^2 \Big).$$
Next, we study in detail $\E[\mathbb{B}_h]$. Splitting $\mathbb{B}_h$ in a way analogous to 
\eqref{eq: maj B B1 B2 DT} in Proposition \ref{P: maj BBT est cov} we have
\begin{align*}
	\mathbb{B}_{h} &\le 
	8 \sum_{\eta \in H_n}
		\left\{   \left( \abs*{\hat{\pi}^z_{h, \eta}(\bm{x_0})  - \E[\hat{\pi}^z_{h, \eta}(\bm{x_0})] }^2 -
	\frac{1}{16}\mathbb{V}_{\eta}\right)_+ \right\}
\\	& \qquad\qquad\qquad+
	8 \sum_{\eta \in H_n}
	\left\{   \left( \abs*{\hat{\pi}^z_{\eta}(\bm{x_0})  - \E[\hat{\pi}^z_{ \eta}(\bm{x_0})] }^2 -
	\frac{1}{16}\mathbb{V}_{\eta}\right)_+ \right\}
	+ 8 \mathbb{D}_h^2
	\\& \label{eq: maj B B1 B2 DT}
	=:8\left[ \mathbb{B}_{h}^{(1)}+ \mathbb{B}_{h}^{(2)} + \mathbb{D}_h^2\right].
\end{align*}
Hence, Proposition \ref{prop: adaptive density} will be proven once we show that 
\begin{equation}{\label{eq: bound Bh 1.25}}
\E[\mathbb{B}_h^{(l)}] \le \frac{c}{n^{\bar{c}}} \frac{1}{\prod_{j=1}^d \alpha_j^2}
\end{equation}
for $h \in H_n$ and $l=1, 2$. We start by considering $\E[\mathbb{B}_h^{(2)}]$. Similarly as in Proposition \ref{prop: adaptive density}, we introduce for any $\eta \in H_n$ the function $g: \R^d \times \R^d \rightarrow \R, $ defined as 
$$g_\eta(x^1, \dots , x^d, e^1, \dots , e^d) = (\frac{1}{\eta} K(\frac{x^1 -x_0^1}{\eta}) + e^1) \times \dots \times (\frac{1}{\eta} K(\frac{x^d -x_0^d}{\eta}) + e^d).$$
It is such that
\begin{align*}
g_\eta({X_i}^1, \dots, X_i^d, \mathcal{E}_i^{1, \eta}, \dots , \mathcal{E}_i^{d, \eta}) & = (\frac{1}{\eta} K(\frac{{X_i}^1 -x_0^1}{\eta}) + \mathcal{E}_i^{1, \eta}) \times\dots \times (\frac{1}{\eta} K(\frac{X_i^d -x_0^d}{\eta}) + \mathcal{E}_i^{d, \eta}) \\
& = Z_i^{1, \eta} \times \dots \times Z_i^{d, \eta}. 
\end{align*}
Hence, we can write 
$$\hat{\pi}^z_{\eta}(\bm{x_0})  - \E[\hat{\pi}^z_{ \eta}(\bm{x_0})] = \frac{1}{n} \sum_{i = 1}^n \left \{ g_\eta(\bm{X}_i, \bm{\mathcal{E}}_i^{\eta}) - \E[g_\eta(\bm{X}_i, \bm{\mathcal{E}}_i^{\eta})] \right  \}.$$
As in the proof of Proposition \ref{P: maj BBT est cov} we want to apply Bernstein's inequality, for which we need the variables to be bounded. For this reason we introduce 
\begin{equation*}
	\tilde{\Omega}_n=\left\{ \omega \in \Omega \mid \forall j\in \{1,\dots,d\}  ~\forall l\in \{1,2\}, ~\forall h \in H_n, ~\forall i \in \{1,\dots,n\}, 
	\text{we have } \abs{\mathcal{E}_i^{j,h}} \le \frac{ \tilde{a}_n^j }{h}
	\right\},
\end{equation*}
where $\tilde{a}_n^j := \frac{\log n}{\beta_n^j} 2\kappa (c_0 + 4d)$, $\kappa$ is as in \eqref{eq: property kernel} and $c_0$ is the constant given in the statement. Then, we can modify $g_\eta$. We set
\begin{align*}
\tilde{g_\eta}({X_i}^1, \dots , X_i^d, \mathcal{E}_i^{1, \eta}, \dots , \mathcal{E}_i^{d, \eta}) : = (\frac{1}{\eta} K(\frac{{X_i}^1 -x_0^1}{\eta}) + [\mathcal{E}_i^{1, \eta}]_{\frac{\Tilde{a}_n^1}{\eta}}) \times \dots \times (\frac{1}{\eta} K(\frac{X_i^d -x_0^d}{\eta}) + [\mathcal{E}_i^{d, \eta}]_{\frac{\Tilde{a}_n^d}{\eta}}),
\end{align*}
where we have used the same notation as in {\modar Section \ref{Sss:joint moment estimator}} to denote the truncation of the Laplace random variables. \\
Following the proof of Lemma \ref{l: majo Omega comp est cov} it is easy to check that 
\begin{equation}{\label{eq: bound P tilde Omega c 1.5}}
\mathbb{P}(\tilde{\Omega}_n^c) \le d \frac{\lfloor \log_2 n \rfloor}{n^{4d + c_0}}.
\end{equation}
Indeed, $\tilde{\Omega}_n^c$ is included in $\cup_{j = 1}^d \cup_{h \in H_n} \{ |\mathcal{E}_i^{j,h}| \ge \frac{\tilde{a}_n^j}{h} \}$, with $h \mathcal{E}_i^{j,h}$ distributed as $\frac{2 \kappa}{\beta_n^j} \mathcal{L}(1)$. Hence, 
\begin{align*}
 \mathbb{P}(\tilde{\Omega}_n^c) & \le Card(H_n) \sum_{j = 1}^d \mathbb{P}(\frac{2 \kappa}{\beta_n^j}|\mathcal{L}(1)| \ge \tilde{a}_n^j ) \le \lfloor \log_2 n \rfloor \sum_{j = 1}^d e^{- \frac{\tilde{a}^j_n \beta_n^j}{2 \kappa}} \le d \lfloor \log_2 n \rfloor n^{-4d - c_0}, 
\end{align*}
as we have chosen $\tilde{a}^j_n := \frac{\log n}{ \beta_n^j} 2\kappa (c_0 + 4d)$. \\
We observe that 
\begin{equation}{\label{eq: def M eta 2}}
\left \| \tilde{g}_\eta \right \|_\infty \le \frac{1}{\eta^d} (\kappa + \tilde{a}_n^1)\times ... \times  (\kappa + \tilde{a}_n^d)=: M_\eta.
\end{equation}
Acting as in the proof of Theorem \ref{th: upper density} (see in particular \eqref{eq: variance with sum}) it is moreover easy to see that 
\begin{align}{\label{eq: def v eta 3}}
 var(\tilde{g_\eta}({X_i}^1, ... , X_i^d, \mathcal{E}_i^{1, \eta}, ... , \mathcal{E}_i^{d, \eta})) & \le \frac{c}{\eta^{2d}} \frac{1}{\prod_{j = 1}^d (\beta_n^j)^2} \sum_{k = 0}^d \prod_{j = 1}^k \eta (\beta_n^j)^2 \le \frac{c}{\eta^{2d}} \frac{1}{\prod_{j = 1}^d (\beta_n^j)^2} =: v_\eta,
\end{align}
where we have used that $\eta (\beta_n^j)^2 \le 1$. We then apply Bernstein inequality (similarly to the proof of Proposition \ref{P: maj BBT est cov}) to the random variables $\tilde{g_\eta}$, on $\tilde{\Omega}_n$. It follows, as in \eqref{eq : Berstein applique gamma T},
\begin{multline*}
	\P \left(  \left\{\abs*{\hat{\pi}^z_{\eta}(\bm{x_0})  - \E[\hat{\pi}^z_{ \eta}(\bm{x_0})] }^2 
	-\frac{1}{16} \mathbb{V}_{\eta} \ge t ; ~ \tilde{\Omega}_n \right\}
	\right) \\
	 =	\P \left(  \left\{ \abs{n^{-1} \sum_{i=1}^n \tilde{g}_{\eta}(\bm{X}_i, \bm{\mathcal{E}}^{\eta}_i) - \E[\tilde{g}_{\eta}(\bm{X}_i, \bm{\mathcal{E}}^{\eta}_i)]}  \ge  \sqrt{\frac{1}{16} \mathbb{V}_{\eta}+t}
	;~ \tilde{\Omega}_n\right\}\right)
	  \\ \le 2 \exp\left(-\frac{n \mathbb{V}_{\eta}}{64 v_{\eta}} \right)	  
	  \exp \left( -\frac{nt}{4v_{\eta}}\right)
	  +2 \exp\left(\frac{-n\sqrt{\mathbb{V}_{\eta}}}{32 M_{\eta} }  \right) 
	  \exp\left(\frac{-n\sqrt{t}}{8M_{\eta}} \right).
\end{multline*}
We now have $n \frac{\mathbb{V}_\eta }{v_\eta} = \frac{a_n}{c} = \frac{c_0 \log n}{c}$ for some universal constant $c$. Moreover, 
\begin{align*}
    \frac{n \sqrt{\mathbb{V}_\eta}}{M_\eta} & = n \sqrt{a_n \frac{1}{n \eta^{2d}} \frac{1}{\prod_{j = 1}^d (\beta_n^j)^2}} \eta^d \frac{1}{ \prod_{j = 1}^d(\kappa + \tilde{a}_n^j)} \\
    & = \sqrt{n} \sqrt{c_0 \log n}  \frac{1}{ \prod_{j = 1}^d(\kappa + \frac{\log n}{\beta_n^j} 2\kappa (c_0 + 4d))} \frac{1}{\prod_{j = 1}^d (\beta_n^j)^2} \\
    & = \sqrt{n} \sqrt{c_0 \log n} \frac{1}{ \prod_{j = 1}^d(\kappa + \frac{\log n\, \lfloor \log_2 n \rfloor}{  \alpha_j} 2\kappa (c_0 + 4d))} \frac{(\lfloor \log_2 n \rfloor)^{2d}}{\prod_{j = 1}^d \alpha_j^2} \\
    & \ge c' \sqrt{n} (\log n)^{\frac{1}{2} }
\end{align*}
for some constant $c'$. Then, we can follow the arguing in Proposition \ref{P: maj BBT est cov} and, integrating with respect to $t$ on $[0, \infty)$, we obtain 
\begin{equation}{\label{eq: bound on tilde omega 4}}
 \E \left[ \sup_{\eta \in H_n} \left( \abs*{\hat{\pi}^z_{\eta}(\bm{x_0})  - \E[\hat{\pi}^z_{ \eta}(\bm{x_0})] }^2 
	-\frac{1}{16} \mathbb{V}_{\eta} \right)_+ 1_{\tilde{\Omega}_n}
	\right] \le c n^{- \bar{c}} \sum_{\eta \in H_n} (\frac{v_\eta}{n} + \frac
    {M_\eta^2}{n^2}).
\end{equation}
From \eqref{eq: def M eta 2} and \eqref{eq: def v eta 3} and the definition \eqref{eq: def Hn} it follows 
$$M_\eta \le \frac{1}{\eta^d} \prod_{j = 1}^d(\kappa + \frac{\log n}{\beta_n^j} 2 \kappa (c_0 + 4d) ) \le c n^d \prod_{j = 1}^d(1 + \frac{\log n}{ \alpha_j} \lfloor \log_2 n \rfloor ),$$
$$v_\eta \le \frac{c}{\eta^{2d}} \frac{1}{\prod_{j = 1}^d (\beta_n^j)^2} \le n^{2d} (\log_2 n)^{2d} \frac{1}{\prod_{j = 1}^d \alpha_j^2}.$$
Replacing it in \eqref{eq: bound on tilde omega 4} we obtain 
\begin{align}{\label{eq: final bound on tilde omega 5}}
 & \E \left[ \sup_{\eta \in H_n} \left( \abs*{\hat{\pi}^z_{\eta}(\bm{x_0})  - \E[\hat{\pi}^z_{ \eta}(\bm{x_0})] }^2 
	-\frac{1}{16} \mathbb{V}_{\eta} \right)_+ 1_{\tilde{\Omega}_n}
	\right] \\
 & \le c n^{- \bar{c}} \left( n^{2d-1} (\log_2 n)^{2d + 1} \frac{1}{\prod_{j = 1}^d \alpha_j^2} + n^{2d - 2}\prod_{j = 1}^d(1 + \frac{\log n}{ \alpha_j} \lfloor \log_2 n \rfloor )^2 \lfloor \log_2 n \rfloor  \right). \nonumber 
\end{align}
We then deal with the contribution on $\tilde{\Omega}_n^c$ which, together with \eqref{eq: bound P tilde Omega c 1.5}, provides
$$ \E \left[ \sup_{\eta \in H_n} \left( \abs*{\hat{\pi}^z_{\eta}(\bm{x_0})  - \E[\hat{\pi}^z_{ \eta}(\bm{x_0})] }^2 
	-\frac{1}{16} \mathbb{V}_{\eta} \right)_+ 1_{\tilde{\Omega}_n^c}
	\right] \le 2 \sum_{\eta \in H_n} \E[|\hat{\pi}^z_{\eta}(\bm{x_0})|^4]^\frac{1}{2} \mathbb{P}(\tilde{\Omega}_n^c)^\frac{1}{2}. $$
We observe that, because of Jensen inequality, it is 
$$\E[|\hat{\pi}^z_{\eta}(\bm{x_0})|^4] \le \frac{1}{n} \sum_{i = 1}^n \E[|Z_i^{1, \eta} \times \dots \times Z_i^{d, \eta} |^4].$$
Moreover, for all $q \ge 1$,
\begin{align*}
\E[|Z_i^{j, \eta}|^q] &= \E[|\frac{1}{\eta} K(\frac{X_i^j - x_0^j}{\eta}) + \mathcal{E}_i^{j, \eta}|^q] \le \frac{c}{\eta^q} + (\frac{2 \kappa}{\eta \beta_n^j})^q \E[|\mathcal{L}(1)|^q] \le (\frac{c}{\eta \beta_n^j})^q.
\end{align*}
It yields 
\begin{align}{\label{eq: final tilde omega c 6}}
 \E \left[ \sup_{\eta \in H_n} \left( \abs*{\hat{\pi}^z_{\eta}(\bm{x_0})  - \E[\hat{\pi}^z_{ \eta}(\bm{x_0})] }^2 
	-\frac{1}{16} \mathbb{V}_{\eta} \right)_+ 1_{\tilde{\Omega}_n^c}
	\right] & \le c \sum_{\eta \in H_n} \frac{1}{(\eta \beta_n^1)^2} \times \dots \times \frac{1}{(\eta \beta_n^d)^2} (d \lfloor \log_2 n \rfloor n^{- 4 d - c_0})^\frac{1}{2} \nonumber \\
 & \le c\, \mathop{card}(H_n) \frac{n^{2d}}{\prod_{j = 1}^d (\beta_n^j)^2} \frac{\sqrt{\log n}}{n^{2d + \frac{c_0}{2}}} \\
 & \le c \frac{(\log n)^{2d + \frac{3}{2}}}{n^{\frac{c_0}{2}} \prod_{j = 1}^d \alpha_j^2}. \nonumber 
\end{align}
From \eqref{eq: final bound on tilde omega 5} and \eqref{eq: final tilde omega c 6}, recalling that $c_0$ can be chosen arbitrarily large and that $\alpha_j \le 1$ for any $j \in \{1, \dots , d \}$ we obtain, for some $\bar{c} > 0$, 
$$\E \left[ \sup_{\eta \in H_n} \left( \abs*{\hat{\pi}^z_{\eta}(\bm{x_0})  - \E[\hat{\pi}^z_{ \eta}(\bm{x_0})] }^2 
	-\frac{1}{16} \mathbb{V}_{\eta} \right)_+
	\right] \le c \frac{1}{n^{\bar{c}}\, \prod_{j = 1}^d \alpha_j^2 }.$$
 We have therefore proven \eqref{eq: bound Bh 1.25} for $l=2$. For $l=1$ the proof of the bound in \eqref{eq: bound Bh 1.25} is obtained in the same way, applying Bernstein inequality with the same constants $M_\eta$ and $v_\eta$ on 
 $$\P \left(  \left\{\abs*{\hat{\pi}^z_{\eta, h}(\bm{x_0})  - \E[\hat{\pi}^z_{ \eta, h}(\bm{x_0})] }^2 
	-\frac{1}{16} \mathbb{V}_{\eta} \ge t ; ~ \tilde{\Omega}_n \right\}
	\right).$$
The proof is therefore concluded.   
\end{proof}

\begin{proof} [Proof of Theorem \ref{th: adaptive density}].
From Proposition \ref{prop: adaptive density} above one can remark it is enough to evaluate $\inf_{h \in H_n} (\mathbb{D}_h^2 + \mathbb{V}_h )$. Equation \eqref{eq: bound Dh} entails we want to evaluate, up to a constant, the infimum over $h \in H_n$ of 
\begin{equation}{\label{eq: final bound adap density 6.5}}
h^{2 \beta} + a_n \frac{1}{n h^{2 d} \prod_{j = 1}^d (\beta_n^j)^2} \le  h^{2 \beta} +  \frac{c \log n}{n h^{2 d} } \frac{(\log n)^{2d}}{\prod_{j = 1}^d \alpha_j^2}.
\end{equation}
If we choose 
\begin{equation}{\label{eq: optimal adaptive bandwidth 7}}
h^*(n) := \left( \frac{(\log n)^{2 d + 1}}{n \prod_{j = 1}^d \alpha_j^2} \right)^\frac{1}{2(\beta + d)},
\end{equation}
then we obtain the quantity $\left( \frac{(\log n)^{2 d + 1}}{n \prod_{j = 1}^d \alpha_j^2} \right)^\frac{2 \beta}{2(\beta + d)}$, which is the wanted rate. To conclude the proof we have to check that $h^*(n)$ as in \eqref{eq: optimal adaptive bandwidth 7} belongs to $H_n$. It is true as $H_n$ has been constructed in analogy to $\mathcal{T}$, with $\frac{1}{h}$ playing the same role as $T^{(l)}$, for $l=1,2$. Indeed, following the same argumentation as in the proof of Theorem \ref{th: est gamma adaptive}, as $\frac{n \prod_{j = 1}^d \alpha_j^2}{(\log n)^{2 d + 1}} \ge 1$ and $\alpha_j \le 1$, it is $\frac{1}{h^*(n)} \in [1, n]$. Then, by replacing $\frac{1}{h^*(n)}$ by its closest value in $H_n$ we only modify its value in \eqref{eq: final bound adap density 6.5} by a constant, which provides 
$$\inf_{h \in H_n} (\mathbb{D}_h^2 + \mathbb{V}_h ) \le c \left( \frac{(\log n)^{2 d + 1}}{n \prod_{j = 1}^d \alpha_j^2} \right)^\frac{2 \beta}{2(\beta + d)}.$$
It concludes the proof of Theorem \ref{th: adaptive density}.

\end{proof}

\end{document}